\documentclass{article}

\usepackage[centertags]{amsmath}
\usepackage{hyperref}
\usepackage{amsfonts}
\usepackage{amssymb}
\usepackage{amsthm}
\usepackage{newlfont}
\usepackage{amscd}
\usepackage{amsmath,amscd}

\usepackage[arrow,matrix]{xy}
\usepackage{verbatim}

\usepackage{color}

\newcommand{\CC}{\mathbb{C}}
\newcommand{\QQ}{\mathbb{Q}}
\newcommand{\ZZ}{\mathbb{Z}}
\newcommand{\HH}{\mathbb{H}}
\newcommand{\A}{\mathbb{A}}
\newcommand{\PP}{\mathbb{P}}

\newcommand{\NN}{\mathbb{N}}
\newcommand{\GG}{\mathbb{G}}
\newcommand{\J}{\mathfrak{J}}
\newcommand{\C}{\mathfrak{C}}
\newcommand{\U}{\mcal{U}}

\newcommand{\X}{\SimpS}
\newcommand{\Y}{\mathbf{Y}}
\newcommand{\E}{\mathbf{E}}
\newcommand{\B}{\mathbf{B}}

\def\alp{{\alpha}}
\def\bet{{\beta}}
\def\gam{{\gamma}}

\def\lam{{\lambda}}
\def\sig{{\sigma}}
\def\vphi{{\varphi}}
\def\om{{\omega}}
\def\Om{{\Omega}}
\def\Gam{{\Gamma}}
\def\Del{{\Delta}}
\def\Sig{{\Sigma}}

\def\Lam{{\Lambda}}
\def\thet{{\theta}}
\def\vphi{{\varphi}}

\def\Pro{\mathrm{Pro}}
\def\Var{\mathrm{Var}}
\def\Set{\mathrm{Set}}

\newtheorem{thm}{Theorem}[section]

\newtheorem{cor}[thm]{Corollary}

\newtheorem{lem}[thm]{Lemma}
\newtheorem{prop}[thm]{Proposition}

\newtheorem{example}{Example}
\theoremstyle{definition}
\newtheorem{define}[thm]{Definition}
\theoremstyle{remark}
\newtheorem{rem}[thm]{Remark}

\DeclareMathOperator{\Ker}{Ker}
\DeclareMathOperator{\Gal}{Gal}

\DeclareMathOperator{\Stab}{Stab}

\DeclareMathOperator{\Br}{Br}
\DeclareMathOperator{\loc}{loc}
\DeclareMathOperator{\Spec}{Spec}

\DeclareMathOperator{\im}{Im}

\DeclareMathOperator{\hocolim}{hocolim}
\DeclareMathOperator{\colim}{colim}
\DeclareMathOperator{\Mod}{Mod}
\DeclareMathOperator{\Hom}{Hom}
\DeclareMathOperator{\xt}{xt}
\DeclareMathOperator{\Ext}{Ext}
\DeclareMathOperator{\EExt}{\mathbb{E}\xt}

\DeclareMathOperator{\spec}{Spec}
\DeclareMathOperator{\Ho}{Ho}

\DeclareMathOperator{\cosk}{cosk}
\DeclareMathOperator{\sk}{sk}
\DeclareMathOperator{\tr}{tr}
\DeclareMathOperator{\Ex}{Ex}
\DeclareMathOperator{\Map}{Map}
\DeclareMathOperator{\Kan}{Kan}
\DeclareMathOperator{\strict}{\mathbf{st}}
\DeclareMathOperator{\standard}{sd}

\DeclareMathOperator{\Tors}{Tors}

\def\loc{\textup{loc}}

\def\im{\textrm{Im\,}}
\def\inv{\textup{inv}}

\def\rar{\rightarrow}

\def\lrar{\longrightarrow}
\def\llar{\longleftarrow}
\def\hrar{\hookrightarrow}

\def\x{\stackrel}

\def \mcal{\mathcal}
\def \ovl{\overline}
\def \what{\widehat}
\def \wtl{\widetilde}
\def \bksl{\;\backslash\;}
\def \mcal {\mathcal}
\def \mfrak {\mathfrak}
\def \fns {\widetilde{\triangleleft}}
\def \uline {\underline}
\def   \SimpS {\mathbf{X}}

\DeclareFontEncoding{OT2}{}{} 
\DeclareTextFontCommand{\textcyr}{\fontencoding{OT2}\fontfamily{wncyr}\fontseries{m}\fontshape{n}\selectfont}

\newcommand{\Sha}{\textcyr{Sh}}

\title{ Homotopy Obstructions to Rational Points }

\author{Yonatan Harpaz \;\;\; Tomer M. Schlank}
\begin{document}
\maketitle

\section*{Abstract}
In this paper we propose to use a relative variant of the notion of the \'{e}tale homotopy type of an algebraic variety in order to study the existence of rational points on it. In particular, we use an appropriate notion of homotopy fixed points in order to construct obstructions to the local-global principle. The main results in this paper are the connections between these obstructions and the classical obstructions, such as the Brauer-Manin, the \'{e}tale-Brauer and certain descent obstructions. These connections allow one to understand the various classical obstructions in a unified framework.

\tableofcontents

\section{Introduction}
\subsection{ Obstructions to the Local Global Principle - Overview }
Let $X$ be a smooth variety over a number field $K$. A prominent problem in arithmetic algebraic geometry is to understand the set $X(K)$. For example, one would like to be able to know whether $X(K) \neq \emptyset$. As a first approximation one can consider the set
$$ X(K) \subseteq X(\A) $$
where $\A$ is the ring of adeles of $K$.

It is a classical theorem of Minkowski and Hasse that if $X \subseteq \PP^n$ is hypersurface given by one quadratic equation then
$$ X(\A)\neq \emptyset \Rightarrow X(K)\neq \emptyset $$
When a variety $X$ satisfies this property we say that it satisfies the local-global principle. In the 1940's Lind and Reichardt (~\cite{Lin40}, ~\cite{Rei42} ) gave examples of genus $1$ curves that do not satisfy the local-global principle.

More counterexamples to the local-global principle where given throughout the years until in 1971 Manin (~\cite{Man70}) described a general obstruction to the local-global principle that explained all the examples that were known to that date. The obstruction (known as the Brauer-Manin obstruction) is defined by considering a set $X(\A)^{\Br} $ which satisfies
$$ X(K) \subseteq X(\A)^{\Br} \subseteq X(\A) $$
If $X$ is a counterexample to the local-global principle we say that it is accounted for or explained by the Brauer-Manin obstruction if
$$ \emptyset = X(\A)^{\Br} \subseteq X(\A) \neq \emptyset $$

In 1999 Skorobogatov (~\cite{Sko99}) defined a refinement of the Brauer-Manin obstruction (also known as the \'{e}tale-Brauer-Manin obstruction) and used it to give an example of a variety $X$ for which
$$ \emptyset = X(K) \subseteq X(\A)^{\Br}  \neq \emptyset $$
More precisely, Skorobogatov described a new intermediate set
$$ X(K) \subseteq X(\A)^{fin,\Br}  \subseteq X(\A)^{\Br}  \subseteq X(\A) $$
and found a variety $X$ such that
$$ \emptyset = X(\A)^{fin,\Br} \subseteq X(\A)^{\Br}  \neq \emptyset $$

In his paper from 2008 Poonen (~\cite{Poo08}) constructed the first and currently only known example of a variety $X$ such that
$$ \emptyset = X(K) \subseteq X(\A)^{fin,\Br}  \neq \emptyset $$
In 2009 the second author (~\cite{Sch09}) showed that in some cases Poonen's counter-example can be explained by a showing that some smaller set $X(\A)^{fin,\Br \sim D}$ is empty.

A different approach to define obstructions sets of the form
$$ X(K)\subseteq X(\A)^{obs}\subseteq X(\A) $$
is by using descent on torsors over $X$ under linear algebraic groups. This method was
studied by Colliot-Th\'{e}l\`{e}ne and Sansuc (~\cite{CTS80} and ~\cite{CTS87}) for torsors under groups of multiplicative type and by Harari and Skorobogatov (~\cite{HSk02}) in the general non-abelian case (see also ~\cite{Sko01}).

One can define various variants of the descent obstruction by considering only torsors under a certain class of groups. We shall denote by $X(\A)^{desc}$, $X(\A)^{fin}$, $X(\A)^{fin-ab}$ and $X(\A)^{con}$ the obstructions obtained when considering all, only finite, only finite abelian and only connected linear algebraic groups respectively.

In the case where $X$ is projective Harari (~\cite{Har02}) showed that
$$ X(\A)^{\Br} = X(\A)^{con} $$
Lately, building on this work, Skorobogatov (~\cite{Sko09}) and Demarche (~\cite{De09a}) showed that in this case one also has
$$ X(\A)^{fin,\Br} = X(\A)^{desc} $$

\subsection{ Our Results }

In this paper we use a new method in order to construct natural intermediate sets between $X(K)$ and $X(\A)$. This method uses a relative variant of the \textbf{\'{e}tale homotopy type} $\acute{E}t(X)$ of $X$ which was constructed by Artin and Mazur ~\cite{AMa69}.

This variant, denoted by ${\acute{E}t_{/K}(X)}$, is an inverse system of simplicial sets which carry an action of the absolute Galois group $\Gam_K$ of $K$. We then use an appropriate notion of \textbf{homotopy fixed points} to define a (functorial) set $X(hK)$ which serves as a certain homotopical approximation of the set $X(K)$ of rational points. In fact one obtains a natural map
$$ h: X(K) \lrar X(hK) $$
In order to apply this idea to the theory of obstructions to the local-global principle one proceeds to construct an adelic analogue, $X(h\A)$, which serves as an approximation to the adelic points in $X$. One then obtains a commutative diagram of sets
$$ \xymatrix{
X(K) \ar[r]^-{h}\ar[d]^{\loc} & X(hK) \ar[d]^{\loc_h} \\
X(\A) \ar[r]^-{h} & X(h\A) \\
}$$

We then define $X(\A)^h$ to be the set of adelic points whose corresponding adelic homotopy fixed point is rational, i.e. is in the image of $\loc_h$. This set is intermediate in the sense that
$$ X(K) \subseteq X(\A)^h \subseteq X(\A) $$
and so provides an obstruction to the existence of rational points. By using a close variant of this construction (essentially working with homology instead of homotopy) we define a set $X(\A)^{\ZZ h}$ that satisfies
$$ X(K) \subseteq X(\A)^h \subseteq X(\A)^{\ZZ h} \subseteq X(\A) $$
A second variant consists of replacing ${\acute{E}t_{/K}(X)}$ with its $n$'th Postnikov piece (in the appropriate sense) yielding "truncated" versions of the obstruction above denoted by
$$ X(K) \subseteq X(\A)^{h,n} \subseteq X(\A)^{\ZZ h,n} \subseteq X(\A) $$
To conclude we get the following diagram of inclusions of obstruction sets:
$$
\xymatrix{
\empty & X(\A)^{\ZZ h} \ar@{^{(}->}[r] & \cdots\ar@{^{(}->}[r] & X(\A)^{\ZZ h,2} \ar@{^{(}->}[r] & X(\A)^{\ZZ h,1}\ar@{^{(}->}[r]& X(\A)  \\
X(K) \ar@{^{(}->}[r] & X(\A)^{h} \ar@{^{(}->}[u] \ar@{^{(}->}[r] & \cdots \ar@{^{(}->}[r] & X(\A)^{h,2} \ar@{^{(}->}[r] \ar@{^{(}->}[u] & X(\A)^{h,1}\ar@{^{(}->}[u] & \empty }
$$

The main results of this paper (Theorems ~\ref{t:fin}, ~\ref{t:main},~\ref{t:h-is-fin-br} and Corollary ~\ref{c:cosk2-final}) describe these obstructions in terms of previously constructed obstructions:

\begin{thm}
Let $X$ be smooth geometrically connected variety over a number field $K$. Then
$$ X(\A)^h = X(\A)^{fin,Br} $$
$$ X(\A)^{\ZZ h} = X(\A)^{Br} $$
$$ X(\A)^{h,1} = X(\A)^{fin} $$
$$ X(\A)^{\ZZ h,1} = X(\A)^{fin-ab} $$

Furthermore, for every $n \geq 2$
$$ X(\A)^{h,n} = X(\A)^{h} $$
$$ X(\A)^{\ZZ h, n} = X(\A)^{\ZZ h} $$

In particular, the diagram above is equal to the diagram

$$
\xymatrix{
\empty & X(\A)^{\Br} \ar@{^{(}->}[r] & \cdots\ar@{^{(}->}[r] & X(\A)^{\Br} \ar@{^{(}->}[r]  & X(\A)^{fin-ab}\ar@{^{(}->}[r]& X(\A)  \\
X(K) \ar@{^{(}->}[r] & X(\A)^{fin,
\Br} \ar@{^{(}->}[u] \ar@{^{(}->}[r] & \cdots \ar@{^{(}->}[r] & X(\A)^{fin,\Br} \ar@{^{(}->}[r] \ar@{^{(}->}[u] & X(\A)^{fin}\ar@{^{(}->}[u] & \empty }
$$
\end{thm}

This homotopical description of the classical obstructions can be used to relate them in new ways to each other. For example one gets the following consequences:

\begin{cor}[Theorem ~\ref{t:trivial-pi-2}]
Let $K$ be number field and $X$ a smooth geometrically connected $K$-variety. Assume further that $\pi^{\acute{e}t}_2(\ovl{X}) = 0$ (which is true, for example, when $X$ is a curve such that $\ovl{X} \neq \PP^1$). Then

$$ X(\A)^{fin} =X(\A)^{fin,\Br} $$

\end{cor}

\begin{cor}[Theorem ~\ref{t:product-fin-Br}]
Let $K$ be number field and $X,Y$ be two smooth geometrically connected $K$-varieties then

$$(X\times Y)(\A)^{fin,\Br}  = X(\A)^{fin,\Br} \times Y(\A)^{fin,\Br}$$

\end{cor}

When studying homotopy fixed points for pro-finite groups we relay heavily on ~\cite{Goe95} who defines for a profinite group $\Gam$ and a simplicial set $\SimpS$  with continuous $\Gam$-action the notion of a homotopy fixed points space $\SimpS^{h\Gam}$.

Given a field $K$ and a simplicial set $\SimpS$  with continuous $\Gam_K$-action we define a suitable notion of "adelic homotopy fixed points space" $\SimpS^{h\A}$. An additional result we obtain is a generalization of the finiteness of the Tate-Shafarevich groups for a finite Galois modules.

\begin{prop}\label{p:finite-preim}
Let $K$ be a number field and let $\SimpS$ be an excellent finite bounded simplicial $\Gam_K$-set. Then the map
$$ {\loc_{\SimpS}}:\pi_0\left(\SimpS^{h\Gam_K}\right) \lrar \pi_0\left(\SimpS^{h\A}\right) $$ 
has finite pre-images. i.e. for every $(x_\nu) \in \SimpS(h\A)$ the set ${\loc_{\SimpS}}^{-1}((x_\nu))$ is finite.
\end{prop}

Shortly after we published on the ArXiv website a first draft of this paper, Ambrus P\'{a}l published a paper with some similar ideas (~\cite{Pal10}). In his paper P\'{a}l also uses in a slightly different approach the \'{e}tale homotopy type to study rational points on varieties.

The paper is organized as follows:
In \S ~\ref{s:etale-homotopy-type} we recall the notion of the \'{e}tale homotopy type $\acute{E}t(X)$ of $X$ and construct the variations used later for the obstructions. In \S ~\ref{s:obstructions} we recall the definition of the classical obstructions and define our various obstructions $X(\A)^{h,n}, X(\A)^{\ZZ h,n}$ using ${\acute{E}t^{\natural}_{/K}}(X)$ and an appropriate notion of homotopy fixed points.

In \S ~\ref{s:homotopy-fixed-points} we study methods to work with homotopy fixed points under pro-finite groups. In \S ~\ref{s:intersection} we prove that in order to study the obstruction obtained in this method it is enough to study the obstructions obtained from each space in the diagram ${\acute{E}t^{\natural}_{/K}}(X)$ separately. This will prove very useful in sections ~\ref{s:equivalence-1} and ~\ref{s:equivalence-2}.

In \S ~\ref{s:dimension-zero} we give basic results for varieties of dimension zero and non-connected varieties. In \S ~\ref{s:finite-descent} the equivalence between finite (finite-abelian) descent and $X(\A)^{h,1}$ ($X(\A)^{\ZZ h,1}$) is proven for smooth geometrically connected varieties.

In \S ~\ref{s:equivalence-1} we prove (under the same assumptions on $X$) that the Brauer-Manin obstruction is equivalent to $X(\A)^{\ZZ h}$.  In \S ~\ref{s:equivalence-2}  we prove that the \'{e}tale-Brauer-Manin obstruction is equivalent to $X(\A)^{h}$ (again for smooth geometrically connected $X$). Finally in \S ~\ref{s:applications} we give some applications of the theory developed throughout the paper.

We would like to thank our advisors D. Kazhdan, E. Farjoun and E. De-Shalit for their essential guidance. We would also like to thank J.-L. Colliot-Th\'{e}l\`{e}ne, A. Skorobogatov, H. Fausk, D. Isaksen and J. Milne for useful discussions.

We would like also to thank the anonymous referee for the very useful comments and
detailed corrections, which we found very constructive and helpful to improve our
manuscript.

Part of the research presented here was done while the second author was staying at the "Diophantine Equations" trimester program at Hausdorff Institute in Bonn. Another part was done while both authors were visiting MIT. We would like to thank both hosts for their hospitality and excellent working conditions.

The second author is supported by the Hoffman Leadership Program .

\section{ The \'{E}tale Homotopy Type}\label{s:etale-homotopy-type}

In this paper we will only be interested in (smooth) \textbf{varieties} over fields. By this we mean (smooth) reduced separated schemes of finite type over a field $K$. In this section we do not assume any restriction on the field $K$. We denote the category of smooth varieties over $K$ by $\Var{/K}$. We will also refer to them as smooth $K$-varieties.

\begin{rem}
In this paper we deal a lot with both algebraic varieties and simplicial sets. In order to distinguish in notation we will use regular font (e.g. $X,Y$,etc.) to denote algebraic varieties and bold font (e.g. $\X,\Y$, etc.) to denote simplicial sets.
\end{rem}

This section is partitioned as follows. In $\S\S ~\ref{ss:hypercoverings}$ we will discuss the notions of skeleton and coskeleton and use them to define the notion of a hypercovering. In $\S\S ~\ref{ss:etale-homotopy}$ we will describe the classical construction of the \'etale homotopy type as well as the construction of a \textbf{relative} variant which is the main object of interest in this paper. In $\S\S ~\ref{ss:homotopy-type-of-X-U}$ we will give some computational tools to help understand the homotopy type of the spaces we shall encounter.

\subsection{Hypercoverings}\label{ss:hypercoverings}
In this section we shall recall the notion of a \textbf{hypercovering} which is used in the definition of the \'etale homotopy type. For a more detailed treatment of the subject  we refer the reader to \S 8 in ~\cite{AMa69} or to ~\cite{Fri82}. We start with a discussion of the notions of skeleton and coskeleton.

\subsubsection{ Skeleton and Coskeleton }
For every $n \geq -1$ we will consider the full subcategory $\Del_{\leq n} \subseteq \Del$ spanned by the objects $\{[i]\in Ob\Del| i \leq n\}$. The natural inclusion $\Del_{\leq n} \hookrightarrow \Del$ induces a truncation functor
$$\tr_n: \Set^{\Del^{op}}\lrar \Set^{\Del_{\leq n}^{op}}$$
that takes a simplicial set and ignores the simplices of degree $ > n$.

This functor clearly commutes with both limits and colimits so it has a left adjoint, given by left Kan extension
$$\sk_n: \Set^{\Del_{\leq n}^{op}}\lrar \Set^{\Del^{op}}$$
also called the \textbf{$n$-skeleton}, and a right adjoint, given by right Kan extension
$$\cosk_n: \Set^{\Del_{\leq n}^{op}}\lrar \Set^{\Del^{op}}$$
called the \textbf{$n$-coskeleton}. To conclude we have the two adjunction.
$$\sk_n \dashv \tr_n \dashv \cosk_n$$
The $n$-skeleton produces a simplicial set that is freely filled with degenerate simplices above degree $n$.
Now we write

$$ Q_n = \sk_{n+1} \circ \tr_{n+1}: \Set^{\Del^{op}}\lrar \Set^{\Del^{op}}$$
and

$$ P_n = \cosk_{n+1} \circ \tr_{n+1}: \Set^{\Del^{op}}\lrar \Set^{\Del^{op}}$$
for the composite functors. Now this two functors satisfy the adjunction
$$(Q_n \dashv P_n): \Set^{\Del^{op}}\lrar \Set^{\Del^{op}}$$

One of the important roles of $P_n$ is that if $\SimpS$ is a Kan simplicial set then $ P_n(\SimpS)$ is it's $n$'th Postnikov piece.
i.e. for $k> n$ we have
$$ \pi_{k}(P_n(\SimpS)) = 0 $$
and for $k\leq n$ the natural map $\SimpS \to P_n(\SimpS)$ induces an isomorphism $\pi_{k}(\SimpS) \lrar \pi_{k}(P_n(\SimpS))$.

Now suppose we replace the category $\Set$ of sets with an arbitrary category $C$. Then if $C$ has finite colimits one has the functor
$$ Q_n: C^{\Del^{op}} \lrar C^{\Del^{op}} $$
defined in an analogous way and if $C$ has finite limits then one has the functor
$$ P_n: C^{\Del^{op}} \lrar C^{\Del^{op}} $$

The skeleton and coskelaton constructions are very useful and shall be used repeatedly along  the paper. One useful application is the construction of a contractible space with free group action.

\begin{define}\label{def:EG}
Let $G$ be a finite group. We have an action of $G$ on itself by multiplication on the left. Now consider $G$ as a functor:
$$* = \Del_{\leq 0}^{op} \lrar \Set_G$$
Define
$$ \E G = \cosk(G) \in  \Set^{\Del^{op}}_G$$
We can write an explicit description of this simplicial $G$-set by
$$ \E G_n = G^{n+1} $$
Note that the action of $G$ on $\E G$ is free and and that $\E G$ is contractible.
\end{define}

\begin{define}\label{def:BG}
Let $G$ be a finite group. Define
$$ \B G = \E G/G $$
Note that $\B G$ is a connected and satisfies
$$\pi_1(\B G)= G$$
Further more it can be checked that the natural map
$$ \B G \lrar P_1(\B G) $$
is an isomorphism of simplicial sets.
\end{define}

\subsubsection{ Hypercoverings }
We will apply these concepts for the case of $C$ being the \'etale site of an algebraic $K$-variety $X$.
\begin{define}
Let $X$ be a $K$-variety. A \textbf{hypercovering} $\U_\bullet \lrar X$ is a simpicial object in the \'{e}tale site over $X$ satisying the following conditions:
\begin{enumerate}
\item
$\U_0 \lrar X$ is a covering in the \'etale topology.
\item
For every $n$, the canonical map
$$ \U_{n+1} \lrar (\cosk_{n}(\tr_{n}(\U_\bullet)))_{n+1} $$
is a covering in the \'etale topology.
\end{enumerate}

\end{define}

\begin{example}
The most common and simple kind of hypercoverings are those defined through the classical \v{C}ech resolution:
\begin{define}\label{d:cech}
Let $X$ be a $K$-variety and $Y \lrar X$ an \'{e}tale covering of $X$. Considering $Y$ as a functor from $\Del^{\leq 0}$ to the \'etale site of $X$ we will define the \v{c}ech hypercovering of $Y$ to be
$$ \check{Y}_\bullet = \cosk_0(Y) $$
As above we can write an explicit description of this hypercovering by
$$ \check{Y}_n = \overbrace{Y \times_{X} ... \times_{X} Y}^{n+1} $$
\end{define}
\end{example}

\subsection{ The \'Etale Homotopy Type }\label{ss:etale-homotopy}
We shall start our discussion by recalling the definition of the \'{e}tale homotopy type functor as defined by Artin and Mazur in ~\cite{AMa69}. We will then describe a variant of this construction which we will use for the rest of the paper.

Both the construction of the \'etale homotopy type and its variant use categories enriched over simplicial sets. In all cases the enrichment is defined in a very similar way, so it's worth while to describe the general pattern.

Suppose that $C$ is an ordinary category which admits finite coproducts. Given a set $A$ and an object $X \in C$ we will denote by
$$ A \otimes X = \coprod_{a \in A} X $$
the coproduct of copies of $P$ indexed by $A$. Note that $- \otimes -$ is a functor from $\Set \times C$ to $C$ (where $\Set$ is the category of sets) and that $B \otimes (A \otimes X)$ is naturally isomorphic to $(B \times A) \otimes X$.

Let $C^{\Del^{op}}$ denote the category of simplicial objects in $C$ and $\Set^{\Del^{op}}$ the category of simplicial sets. Given a simplicial set $\mathbf{S}_\bullet \in \Set^{\Del^{op}}$ and an object $X_\bullet \in C^{\Del^{op}}$ we will denote by $\mathbf{S}_\bullet \otimes X_\bullet$ the object given by
$$ (\mathbf{S}_\bullet \otimes X_\bullet)_n = \mathbf{S}_n \otimes X_n $$
Now for every two objects $X_\bullet,Y_\bullet \in C^{\Del^{op}}$ one can define a mapping simplicial set $\Map(X_\bullet,Y_\bullet) \in \Set^{\Del^{op}}$ by the formula
$$ \Map(X_\bullet,Y_\bullet)_n = \Hom_{C^{\Del^{op}}}(\Del^n \otimes X_\bullet, Y_\bullet) $$
There is a natural way to define composition of these mapping simplicial sets (using the natural map $\Del^n \otimes X \lrar \Del^n \otimes \Del^n \otimes X$ induced from the diagonal $\Del^n \lrar \Del^n \times \Del^n$) which is strictly associative. For example, in the case $C = \Set$ we obtain the usual enrichment of $\Set^{\Del^{op}}$ over itself.

Note that the zero simplices of $\Map(X,Y)$ are just the usual maps in $C^{\Del^{op}}$ from $X$ to $Y$. Further more if $C,D$ are two such categories and $F: C \lrar D$ is a functor which respects coproducts then it induces a natural functor of simplicially enriched categories
$$ C^{\Del^{op}} \lrar D^{\Del^{op}} $$

Now let $X$ be a smooth $K$-variety and consider the category $\Var_{/X}$ of $K$-varieties over $X$. This category admits finite coproducts so one obtains a natural simplicial enrichment of the category $\Var^{\Del^{op}}_{/X}$ as above. Let $HC(X) \subseteq \Var^{\Del^{op}}_{/X}$ denote the full (simplicially enriched) subcategory spanned by the hypercoverings with respect to the \'etale topology. We denote by $I(X) = \Ho(HC(X))$ the homotopy category of $HC(X)$ with respect to this simplicial enrichment.

Now consider the connected component functor (over $K$):
$$ \pi_0: \Var_{/X} \lrar \Set $$
This functor preserves finite coproducts and so induces a functor of simplicially enriched categories
$$ \pi_0^{\Del^{op}}: \Var^{\Del^{op}}_{/X} \lrar \Set^{\Del^{op}} $$
Since one wants to think of $\pi_0^{\Del^{op}}(U_\bullet)$ as a topological space, one can either take the realization of $\pi_0^{\Del^{op}}(U_\bullet)$ (as is done in ~\cite{AMa69}), or equivalently, take the Kan replacement $\Ex^{\infty}(\pi_0^{\Del^{op}}(U_\bullet))$. The second option is more convenient because it allows one to continue working inside the world of simplicial sets. It is equivalent in the sense that the subcategory
$$ \Ho^{\Kan}\left(\Set^{\Del^{op}}\right) \subseteq \Ho\left(\Set^{\Del^{op}}\right) $$
consisting of Kan simplicial sets is equivalent to the homotopy category of topological spaces with CW homotopy type.

As the functor $\Ex^{\infty}$ extends to a functor of simplicially enriched categories so does the composed functor
$ \Ex^{\infty}\left(\pi_0^{\Del^{op}}\left(\bullet\right)\right) $.
Restricting this functor to $HC(X)$ and descending to the respective homotopy categories one obtains a functor
$$ \acute{E}t(X): I(X) \lrar \Ho^{\Kan}\left(\Set^{\Del^{op}}\right) $$
Since the category $I(X)$ is cofiltered (Corollary 8.13 ~\cite{AMa69}) we can consider $\acute{E}t(X)$ as a \textbf{pro-object} in $\Ho^{\Kan}(\Set^{\Del^{op}})$, i.e. an object in the pro-category of $\Ho^{\Kan}(\Set^{\Del^{op}})$.

For a category $C$ we will denote the pro-category of $C$ by $\Pro\;C$. Recall that objects of $\Pro\;C$ are diagrams $\{X_\alp\}_{\alp \in I}$ of objects of $C$ indexed by a cofiltered category $I$ and that
$$ \Hom_{\Pro\;C}\left(\{X_\alp\}_{\alp \in I}, \{Y_\bet\}_{\alp \in J}\right) = \lim_{\bet \in J}\colim_{\alp \in I} \Hom_{C}(X_\alp,Y_\bet) $$

Now if $f: X \lrar Y$ is a map of $K$-varieties and $U_\bullet \lrar Y$ is a hypercovering of $Y$ we can pull it back to obtain a hypercovering $f^*U_\bullet \lrar X$. One then gets a natural map of simplicial sets
$$ \Ex^{\infty}\left(\pi_0^{\Del^{op}}(f^*\U)\right) \lrar \Ex^{\infty}\left(\pi_0^{\Del^{op}}(\U)\right) $$
These natural maps fit together to form a map
$$ \acute{E}t(X) \lrar \acute{E}t(Y) $$
in $\Pro \Ho^{\Kan}(\Set^{\Del^{op}})$. This exhibits $\acute{E}t$ as a functor
$$ \acute{E}t:\Var_{/K} \lrar \Pro \Ho^{\Kan}(\Set^{\Del^{op}}) $$
This is the \textbf{\'etale homotopy type} functor defined in ~\cite{AMa69}.

In ~\cite{AMa69} Artin and Mazur work with a certain localization of $\Pro\Ho^{\Kan}\left(\Set^{\Del^{op}}\right)$, via Postnikov towers. Postinkov towers are a way of filtering a topological space by higher and higher homotopical information.
In order to use Postinkov towers here we need a functorial and simplicial way to describe them. Such a description can found in ~\cite{GJa99} by using the functor $P_n = \cosk_{n+1} \circ \tr_n$ defined in the previous section.

Note that if $\SimpS$ is a Kan simplicial set then $P_n(\SimpS)$ will be a Kan simplicial set as well, i.e. $P_n$ can be considered as a functor from Kan simplicial sets to Kan simplicial sets. Further more $P_n$ extends to a functor of simplicially enriched categories. Descending to the homotopy category we get a functor
$$ P_n: \Ho^{\Kan}\left(\Set^{\Del^{op}}_{\Gam_K}\right) \lrar \Ho^{\Kan}\left(\Set^{\Del^{op}}_{\Gam_K}\right) $$
Then for $\SimpS_I = \{\SimpS_\alp\}_{\alp \in I} \in \Pro \Ho^{\Kan}(\Set^{\Del^{op}})$ one defines
$$ \SimpS_I^{\natural} = \{P_n(\SimpS_\alp)\}_{n,\alp} $$
In many ways the object $\SimpS_I^{\natural}$ is better behaved then $\SimpS_I$. We will imitate this stage as well in our construction of the relative analogue of $\acute{E}t(-)$.

We now come to the construction of the relative analogue. If $X$ is a variety over $K$ then it admits a natural structure map $X \lrar \spec(K)$. We wish to replace the functor $\pi_0$ from $K$-varieties to sets with a \textbf{relative} version of it. This relative functor should take a variety over $K$ and return a "set over $K$", i.e. a sheaf of sets on $\spec(K)$ (with respect to the \'etale topology).

Without getting into the formalities of the theory of sheaves let us take a shortcut and note that to give an \'{e}tale sheaf of sets on $\spec(K)$ is equivalent to giving a set with a $\Gam_K$ action such that each element has an open stabilzer. We will denote such objects by the name \textbf{$\Gam_K$-sets} and their category by $\Set_{\Gam_K}$.
\begin{rem}
Some authors use the term \textbf{discrete} $\Gam_K$-set in order to emphasize the continuous action of $\Gam$. Since we will \textbf{never} consider non-continuous actions, and in order to improve readability, we chose to omit this adjective.
\end{rem}

Note that $\Set_{\Gam_K}$ admits coproducts and so we have a natural simplicial enrichment of $\Set^{\Del^{op}}_{\Gam_K}$. The relative version of $\pi_0$ is the functor
$$ {\pi_0}_{/K}: \Var/K \lrar \Set_{\Gam_K} $$
which takes the $K$-variety $X$ to the $\Gam_K$-set of connected components of $$\overline{X} = X\otimes_K\ovl{K}$$. This functor preserves coproducts and so induces a functor of simplicially enriched categories
$$ {\pi_0}^{\Del^{op}}_{/K}: \Var^{\Del^{op}}_{/K} \lrar \Set^{\Del^{op}}_{\Gam_K} $$

As before we don't really want to work in $\Set^{\Del^{op}}_{\Gam_K}$ itself, but with some localization of it. In ~\cite{Goe95} Goerss considers two simplicial model structures on the simplicially enriched category $\Set^{\Del^{op}}_{\Gam_K}$ of simplicial Galois sets. In the first one, which is called the \textbf{strict} model structure, the weak equivalences are equivariant maps of $f: \X \lrar \Y$ of simplicial Galois sets such that the induced map
$$ f^\Lam: \X^{\Lam} \lrar \Y^{\Lam} $$
is a weak equivalence of simplicial sets for every open normal subgroup $\Lam \triangleleft \Gam$ (in what follows we will use the notation $\Lam \fns \Gam$ to denote an open normal subgroup). In the second model structure, which we will refer to as the \textbf{standard} model structure, weak equivalences are equivariant maps $f: \X \lrar \Y$ of simplicial Galois sets which induce a weak equivalence on the underlying simplicial sets. In both cases the cofibrations are just injective maps (and hence all the objects are cofibrant) and for the strict model structure one has also a concrete description of the fibrations: they are maps $f: \X \lrar \Y$ such that the induced maps $ f^\Lam: \X^{\Lam} \lrar \Y^{\Lam} $ are Kan fibrations for every $\Lam \fns \Gam$.

We will denote by $\Ho\left(\Set^{\Del^{op}}_{\Gam_K}\right)$ the homotopy category of $\Set^{\Del^{op}}_{\Gam_K}$ with respect to the simplicial enrichment given above. The localized homotopy categories with respect to the strict and standard model structures will be denoted by $\Ho^{\strict}\left(\Set^{\Del^{op}}_{\Gam_K}\right)$ and $\Ho^{\standard}\left(\Set^{\Del^{op}}_{\Gam_K}\right)$ respectively. They can be realized as a sequence of full sub-categories

$$ \Ho^{\standard}\left(\Set^{\Del^{op}}_{\Gam_K}\right) \subseteq \Ho^{\strict}\left(\Set^{\Del^{op}}_{\Gam_K}\right) \subseteq \Ho\left(\Set^{\Del^{op}}_{\Gam_K}\right) $$
obtained by restricting to standardly fibrant and strictly fibrant objects respectively.

Now similar to above we would like to compose ${\pi_0}^{\Del^{op}}_{/K}$ with the strict fibrant replacement functor. It is convenient to note that strict fibrant replacement can actually be done using the Kan replacement functor $\Ex^{\infty}$: the functor $\Ex^{\infty}$ induces also a functor $\Set^{\Del^{op}}_{\Gam} \lrar \Set^{\Del^{op}}_{\Gam}$ (which by abuse of notation we will also call $\Ex^{\infty}$), and we have the following observation:

\begin{lem}\label{l:Ex_inf_is_fib}
If $\X \in \Set_\Gam^{\Del^{op}}$ is a simplicial $\Gam$-set then object $\Ex^{\infty}(\X)$ is strictly fibrant and the map $\X \lrar \Ex^{\infty}(\X)$ is a strict weak equivalence.
\end{lem}
\begin{proof}
We need to show that for every $\Lambda \fns \Gamma$ the simplicial set $\Ex^{\infty}(\X)^{\Lambda}$ is Kan and the map
$\X^{\Lambda} \lrar \Ex^{\infty}(\X)^{\Lambda}$ is a weak equivalence. Both claims follow easily once one shows that
$$\Ex^{\infty}(\X)^{\Lambda}  = \Ex^{\infty}(\X^{\Lambda}) $$
This in turn follows from the fact that $\Ex^{\infty}(\X)$ is the colimit of the sequence of \textbf{inclusions}:
$$ \X \hrar \Ex(\X) \hrar \Ex(\Ex(\X)) \hrar \Ex(\Ex(\Ex(\X))) \hrar ...$$
and
$$ \Ex(\X)^{\Lambda}  = \Ex(\X^{\Lambda}) $$
since $\Ex$ has a left adjoint (given by barycentric subdivision).
\end{proof}

We now take the functor $\Ex^{\infty}\left(\pi^0_{/K}(\bullet)\right)$ restricted to $HC(X) \subseteq \Var^{\Del^{op}}_{/K}$ and descend to the respective homotopy categories. We end up with a functor

$$ \acute{E}t_{/K}(X): I(X) \lrar \Pro \Ho^{\strict}\left(\Set^{\Del^{op}}_{\Gam_K}\right) $$
We consider $\acute{E}t_{/K}(X)$ to be the \textbf{relative analogue} of $\acute{E}t(X)$. Now just as for $\acute{E}t$ one can use pullbacks of hypercovering in order to make $\acute{E}t_{/K}$ into a functor:

$$\acute{E}t_{/K}: \Var_{/K} \lrar \Pro \Ho^{\strict}\left(\Set^{\Del^{op}}_{\Gam_K}\right) $$

As in the case of the regular \'etale homotopy type it will be better to work with an appropriate Postinkov tower ${\acute{E}t^{\natural}_{/K}}(X)$ of ${\acute{E}t_{/K}}(X)$, i.e. we want a way to estimate an object by a tower of "bounded" objects. When working with strictly fibrant objects the relevant notion of boundedness is stronger:
\begin{define}
Let $\Gam$ be a pro-finite group and $\X$ a simplicial $\Gam$-set. We will say that $\X$ is \textbf{strictly bounded} if it is strictly fibrant and if there exists an $N$ such that for all $\Lam \fns \Gam$ the simplicial set $X^{\Lam}$ is $N$-bounded (i.e. all the homotopy groups $ \pi_n\left(X^{\Lam}\right) $ are trivial for $n > N$. Note in particular that the empty set $\empty$ is considered $N$-bounded for every $N \geq 0$).
\end{define}

Let let $\X$ be a strictly fibrant simplicial $\Gam_K$-set. Recall that the underlying simplicial set of $\X$ is Kan so we can consider
$$ P_n(\X) = \cosk_{n+1}(\tr_{n+1}(\X)) $$
Since $P_n$ is a functor we have an induced action of $\Gam_K$ on $P_n(\X)$. Note that under this action the stabilizer of each simplex in $P_n(\X)$ will be open and so $P_n(\X)$ is a simplicial $\Gam_K$-set. Further more since
$$ P_n(X)^{\Lam} = P_n\left(X^{\Lam}\right) $$
for each $\Lam \fns \Gam_K$ we get that $P_n(X)$ is a strictly bounded simplicial $\Gam_K$-set.

Hence we can think of $P_n$ as a simplicially enriched functor from the category of strictly fibrant simplicial Galois sets to itself. Descending to the homotopy category we get a functor
$$ P_n: \Ho^{\strict}\left(\Set^{\Del^{op}}_{\Gam_K}\right) \lrar \Ho^{\strict}\left(\Set^{\Del^{op}}_{\Gam_K}\right) $$
Then for $\SimpS_I = \{\SimpS_\alp\}_{\alp \in I} \in \Pro \Ho^{\strict}(\Set^{\Del^{op}})$ we define as above
$$ \SimpS_I^{\natural} = \{P_n(\SimpS_\alp)\}_{n,\alp} $$

It will also be convenient to consider the pro-objects
$$ \SimpS_I^{n} = \{P_k(\SimpS_\alp)\}_{k \leq n ,\alp} $$
Where for $n= \infty$  we define:
$$ \SimpS_I^{\infty} = \SimpS_I^{\natural} $$

Note also that since $P_n(\bullet)$ is an augmented functor so is $(\bullet)^{n}$ and so for every $0 \leq n \leq \infty$ we have a natural map
$$ \SimpS_I \lrar \SimpS_I^{n}$$

We will use the following notation:
\begin{define}
Let $X$ be a $K$-variety and $\U \lrar X$ a hypercovering. We will denote
$$ \SimpS_{\U}  = \Ex^{\infty}\left(\pi^{\Del^{op}}_{/K}(\U)\right).$$
and
$$\SimpS_{\U,k} = P_k(\SimpS_{\U}) $$

\end{define}

Under this notation one has
$$ {\acute{E}t_{/K}}(X) = \{\SimpS_{\mcal{U}}\}_{\U \in I(X)} $$
$$ {\acute{E}t^{n}_{/K}}(X) = \{\SimpS_{\mcal{U}, k}\}_{\U \in I(X), k \leq n} $$

There are several important properties that the the underlying simplicial sets in the diagram of ${\acute{E}t^{n}_{/K}}(X)$ satisfy. Here are the properties that we will be interested in:
\begin{define}\label{df:typeof_spaces}
Let $\Gam$ be a pro-finite group. A simplicial $\Gam$-set $\X$ will be called
\begin{enumerate}
\item
\textbf{Finite} if $\pi_n(\X)$ is finite for $n \geq 1$.
\item
\textbf{Excellent} if the action of $\Gam$ on $\X$ factors through a finite quotient of $\Gam$.
\item
\textbf{nice} if the action of $\Gam$ on every skeleton $\X_n$ factors through a finite quotient of $\Gam$.
\end{enumerate}
\end{define}
\begin{rem}
When we say a finite quotient of a pro-finite group we always mean a \textbf{continuous} finite quotient, i.e. a finite quotient in the category of pro-finite groups.
\end{rem}

Now let $X/K$ be an algebraic $K$-variety and $\U \lrar X$ a hypercovering. Then clearly ${\pi_0}_{/K}(\U)$ is nice and it is easy to show that $\SimpS_{\mcal{U}} = \Ex^{\infty}\left({\pi_0}_{/K}(\U)\right)$ is then nice as well (this is one of the advantages of using $\Ex^{\infty}$ as strict fibrant replacement). Then we see that for each $k < \infty$ the simplicial $\Gam_K$-set $\X_{\U,k}$ is \textbf{excellent} and \textbf{strictly bounded}.

It can be shown (~\cite{AMa69}) that if $X$ is a smooth $K$-variety then $\SimpS_\mcal{U}$ is finite. Hence in that case $\SimpS_{\mcal{U},k}$ is finite as well.

We finish this subsection with a basic comparison result connecting the relative notion $\acute{E}t_{/K}(X)$ and the \'etale homotopy type $\acute{E}t(\ovl{X})$ of $\ovl{X} = X \otimes_K \ovl{K}$. Note that by forgetting the group action we obtain a forgetful functor from $\Ho(\Set^{\Del^{op}}_{\Gam_K})$ to $\Ho(\Set^{\Del^{op}})$. Prolonging this functor we obtain a forgetful functor
$$ F: \Pro \Ho^{\strict}(\Set^{\Del^{op}}_{\Gam_K}) \lrar \Pro\Ho^{\Kan}(\Set^{\Del^{op}}) $$
\begin{prop}\label{p:bar-is-bar}
Let $K$ be a field and $X$ a $K$-variety. Then there is an isomorphism in $\Pro\Ho^{\Kan}(\Set^{\Del^{op}})$:
$$  f:{\acute{E}t}^{n}(\ovl{X}) \lrar F\left({\acute{E}t^{n}_{/K}}(X)\right) $$
\end{prop}
\begin{proof}
Note that the indexing category of ${\acute{E}t^{n}_{/K}}(X)$ is naturally contained in the indexing category of ${\acute{E}t}^{n}(\ovl{X})$ and this inclusion identifies $F\left({\acute{E}t^{n}_{/K}}(X)\right)$ with a sub-diagram of ${\acute{E}t}^{n}(\ovl{X})$. This yields a natural map
$$  f:{\acute{E}t}^{n}(\ovl{X}) \lrar F\left({\acute{E}t^{n}_{/K}}(X)\right) $$

In order to show that $f$ is an isomorphism one needs to show that this sub-diagram is cofinal. For concretely, one needs to show that for every hypercovering $\mcal{V}_\bullet \lrar  \ovl{X}$ defined over $\ovl{K}$ and every $0 \leq k \leq n$ there exists a hypercovering $\U_\bullet \lrar X$ (defined over $K$) such that $\SimpS_{\ovl{\U},k}$ dominates  $\SimpS_{\mcal{V},k}$.

First note that it is not restrictive to consider $P_k(\mcal{V})$ instead of $\mcal{V}$. Thus we may assume that $\mcal{V}$ is defined over some finite field extension $L/K$. Then it is clear that we can take
$$ \U =  R_{L/K}\mcal{V} $$
where $R_{L/K}$ is the restriction of scalars functor applied levelwise.

\end{proof}

\subsection{ The Homotopy Type of $\X_{\U}$ }\label{ss:homotopy-type-of-X-U}
Let $X$ be a smooth geometrically connected $K$-variety and $\U_\bullet \lrar X$ a hypercovering. In this section we will give basic results which help to analyze the homotopy type of the spaces $\X_{\U}$ which appear in $\acute{E}t_{/K}(X)$.

Recall that given a map $ f:Y \lrar X $ and an \'{e}tale hypercovering $\U_\bullet \lrar X$ one can levelwise pull $\U$ back to a hypercovering of $Y$. We shall denote this hypercovering by
$$ f^*\U_\bullet \lrar Y $$
Note that this construction gives a natural map
$$ \SimpS^{f} : \SimpS_{f^*\U} \lrar \SimpS_{\U} $$

Now let $X/K$ be a smooth geometrically connected variety. In this case $X$ has a generic point
$$\xi:\spec(K(X)) \lrar X $$
where $K(X)$ is the function field of $X$ (over $K$).

Note that $\spec(K(X))$ is not a $K$-variety. However, the schemes $\spec(K(X))$ and $\spec(K(X)) \otimes \ovl{K} \cong \spec(\ovl{K}(\ovl{X}))$ are \textbf{Noetherian}, which means that the functor $\pi_0$ is well defined on \'etale schemes over them. In fact, since $K(X)$ is a field the \'{e}tale site of $\spec K(X)$ can be identified with the site of finite discrete $\Gam_{K(X)}$-sets. Under this identification a hypercovering of $\spec(K(X))$ corresponds to a Kan contractible simplicial discrete $\Gam_{K(X)}$-set which is levelwise finite. By abuse of notation we shall denote the simplicial $\Gam_{K(X)}$-set corresponding to a hypercovering $\U \lrar \spec(K(X))$ simply by $\U$.

One can then easily check that
$$ {\pi_0}_{/K}(\U) = \U/\Gam_{\ovl{K}(\ovl{X})} $$
where the $\Gam_K$ action on the right hand side is induced from the $\Gam_{K(X)}$ action on $\U$ and the short exact sequence
$$ 1 \lrar \Gam_{\ovl{K}(\ovl{X})} \lrar \Gam_{K(X)} \lrar \Gam_K \lrar 1 $$

Now consider a hypercovering $\U_\bullet \lrar X$ and its pullback $\xi^*\U$ to $\spec(K(X))$. We have a map
$$\SimpS^{\xi} : \SimpS_{\xi^*\U} \lrar \SimpS_{\U} $$
Since connected components can be read of the fiber over the generic point this map is an isomorphism of simplicial $\Gam_K$-sets. This allows one to interpret $\SimpS_{\U}$ as a quotient of the form
$$ \SimpS_{\U} = (\xi^*\U)/\Gam_{\ovl{K}(\ovl{X})} $$
where $\xi^*\U$ is a Kan contractible simplicial $\Gam_{K(X)}$-set. This interpretation of $\SimpS_\U$ will be a useful later in order to calculate invariants of $\SimpS_\U$. For example we immediately get the following conclusion:

\begin{cor}
Let $X/K$ be a smooth geometrically connected variety and
$\U_\bullet \lrar X$ a hypercovering. Then $\SimpS_\U$ connected.
\end{cor}

We can also use this interpretation in order to calculate the fundamental group of
$\SimpS_{\U}$. This is done using the following lemma:

\begin{lem}\label{l:calc_pi1}
Let $G$ be a group and $\X$ a contractible Kan simplicial $G$-set. Then, given a base point $\ovl{x} \in \X/G$, there is a natural short exact sequence
$$ 1 \lrar K \lrar G \lrar \pi_1(\X/G, \ovl{x}) \lrar 1 $$
where
$$ K = \left \langle\bigcup_x\left\{\Stab_G(x)| x\in \X_0\right\}\right\rangle $$

\end{lem}
\begin{proof}
Let
$$ \pi:\X \lrar \X/G $$
be the natural quotient map and $a \in \X_0$ a vertex. We will construct a surjective map
$$\phi_a:G \lrar \pi_1(\X/G,\pi(a))$$
such that
$$ \ker(\phi_a) = K $$

Let $g$ be an element of $G$. Since $\X$ is contractible and Kan there is at least one $1$-simplex $l_g$ in $\X$ joining $a$ and $ga$. We shall take $\phi_a(g)$ to be the element in $\pi_1(\X/G,\pi(a))$ corresponding to $\pi(l_g)$.

We shall first show that $\phi_a$ is well defined. Assume that $l'_g$ is another $1$-simplex connecting $a$ and $ga$. Since $\X$ is contractible there is an end-points-preserving homotopy between the paths corresponding to $l_g$ and $l'_g$ in the realization of $\X$. Projecting this homotopy to $\X/G$ we get that the corresponding elements in  $\pi_1(\X/G,\pi(a))$ are equal.

We will now prove surjectivity. It is not hard to see that $\X/G$ is Kan in dimension one (see for example Lemma $11.6$ in ~\cite{AMa69}) and therefore every element in $\pi_1(\X/G,\pi(a))$ can be represented by a $1$-simplex in $\X/G$ with both end points being $\pi(a)$. We can then lift this $1$-simplex to $\X$ (note that the map $\X \lrar \X/G$ is levelwise surjective) obtaining a simplex joining $a$ and $ga$ for some $g \in G$.

It now remains to calculate the kernel of $\phi_a$. First we shall show that if $b$ is another point in $\X_0$ then
$$ \ker \phi_b = \ker \phi_a $$
Since $\X$ is contractible and Kan there is a $1$-simplex in $\X$ going from $b$ to $a$.
We denote this $1$-simplex by $p_{ba}$. Now for every $g$ consider the three $1$-simplices $p_{ba},l_g$ and $g(p_{ba})$ creating together a path between $b$ and $gb$. Since $\X$ is Kan this path is homotopic to a $1$-simplex $l^b_g$ joining $b$ and $gb$.

We shall use this $l^b_g$ to obtain $\phi_b(g)$. From the discussion above we get that the following diagram commutes
$$
\xymatrix{
G \ar@{=}[d] \ar[r]^-{\phi_b} & \pi_1(\X/G,\pi(b))\ar[d]^{c_{\pi(p_{ba})}}\\
G   \ar[r]^-{\phi_a} & \pi_1(\X/G,\pi(a))
}
$$
where $c_{\pi(p_{ba})}$ is the natural isomorphism defined by "conjugation" by the path $\pi(p_{ba})$ between $\pi(b)$ and $\pi(a)$. This diagram implies that the kernel $\ker \phi_a$ is independent of the choice of $a \in \X$. Since clearly
$$ \Stab_G(a) \subseteq \ker \phi_a $$
we get that $K \subseteq \ker(\phi_a)$.

To complete the proof of the lemma it suffices to show that $\ker \phi_a \subseteq K$. For every $x\in \X_0$ we have $\Stab_G{x} \subseteq K$. Thus there is map of $G$-sets
$$ \phi: \X_0 \lrar G/K $$

Recall that $K$ is normal in $G$ and let $\E(G/K)$ be as in Definition ~\ref{def:EG}. $\E(G/K)$ is a Kan contractible simplicial set
with the free action of $G/K$. Pulling this action to $G$ we get a Kan contractible $G$-simplicial set. The quotient $\E(G/K)/G$ is a weakly equivalent to $\B(G/K)$, whose fundamental group is exactly $G/K$.

The map $\phi$ lifts to a unique equivariant map
$$ \check{\phi}: \X \lrar \E(G/K) $$

and the following commutative diagram
$$
\xymatrix{
G \ar[d]^{\phi_a} \ar[dr]^{\phi}\ar[drr]  & \empty & \empty \\
\pi_1(\X/G) \ar[r]^-{\check{\phi}} & \pi_1(\B(G/K))  \ar@{=}[r] &  G/K ,\\
}
$$
implies that $\ker \phi_a \subseteq K$.

\end{proof}

\section{ The Obstructions }\label{s:obstructions}
\subsection{ The Classical Obstructions }
\subsubsection{ The Brauer-Manin Obstruction }
Let $K$ be a number field and $X$ an algebraic variety over $K$. Given an adelic point $(x_\nu)_\nu \in X(\A)$ and an element $u \in H_{\acute{e}t}^2(X,\GG_m)$ one can pullback $u$ by each $x_\nu$ to obtain an element
$$ x_\nu^* u \in H_{\acute{e}t}^2\left({\spec(K_\nu)},\GG_m\right) $$
There is a canonical map
$$ \inv: H_{\acute{e}t}^2\left({\spec(K_\nu)},\GG_m\right) \lrar \QQ/\ZZ $$
called the \textbf{invariant} which is an isomorphism in non-archimedean places.
Summing all the invariants one obtains a pairing
$$ X(\A) \times H_{\acute{e}t}^2\left(X,\GG_m\right) \lrar \QQ / \ZZ $$
given by
$$ ((x_\nu)_\nu,u) \mapsto \sum_{\nu} \inv(x_\nu^{*}u) \in \QQ / \ZZ $$
Now by the Hasse-Brauer-Noether Theorem we see that if $(x_\nu)$ is actually a rational point then its pairing with any element in $H^2(X,\GG_m)$ would be zero. This motivates the definition of the Brauer set
$$ X(\A)^{\Br} = \left\{(x_\nu)_\nu \in X(\A) | ((x_\nu)_\nu,u) = 0, \forall u \in H^2(X_{\acute{e}t},\GG_m)\right\} $$
and we have
$$ X(K) \subseteq X(\A)^{\Br} \subseteq X(\A) $$

\subsubsection{ Descent Obstructions }
Let $X$ be an algebraic variety over a number field $K$. It is well known (see e.g. ~\cite{Sko01}) that if $f:Y \lrar X$ is a torsor under a linear algebraic $K$-group $G$ then one has the equality

$$ X(K) = \biguplus \limits_{\sigma \in H^1(K,G)} f^{\sigma}(Y^{\sigma}(K)) $$
and so the set
$$ X(\A)^{f} = \bigcup\limits_{\sigma \in H^1(K,G)} f^{\sigma}(Y^{\sigma}(\A)) $$
has to contain $X(K)$. This motivates the definition
$$ X(\A)^{desc} = \bigcap \limits_{f} X(\A)^{f} $$
where $f$ runs over all torsors under linear algebraic $K$-groups. As before we have
$$ X(K) \subseteq X(\A)^{desc} \subseteq X(\A) $$
We shall denote by $X(\A)^{fin}$, $X(\A)^{fin-ab}$ and $X(\A)^{con}$ the analogous sets obtained by restricting $f$ to torsors under finite, finite abelian and connected linear algebraic groups respectively.

\subsubsection{Applying Obstructions to Finite Torsors}

Given a functorial obstruction set
$$ X(K)\subseteq X(\A)^{obs} \subseteq X(\A) $$
and a torsor under finite $K$-group
$$f: Y\lrar X$$
one can always define the set
$$ X(\A)^{f,obs} = \bigcup\limits_{\sigma \in H^1(K,G)} f^{\sigma}(Y^{\sigma}(\A)^{obs}) $$
satisfying

$$ X(K)\subseteq X(\A)^{f,obs} \subseteq X(\A)^{obs}  \subseteq X(\A) $$
Now by going over all such $f$ we get

$$ X(\A)^{fin,obs} = \bigcap \limits_{f} X(\A)^{f,obs} $$
and
$$ X(K)\subseteq X(\A)^{fin,obs} \subseteq X(\A)^{obs}  \subseteq X(\A)$$

In ~\cite{Sko99} Skorobogatov defines in this way the set
$$ X(\A)^{fin, \Br} $$
and constructs a variety $X$ such that
$$ X(\A)^{fin,\Br} = \emptyset $$ but
$$ X(\A)^{\Br} \neq \emptyset $$

\subsection{ The \'{E}tale Homotopy Obstruction  }
\subsubsection{ Homotopy Fixed Points Sets }
Recall from the previous section the strict model structure and the standard model structure introduced by Goerss (~\cite{Goe95}) on the category $\Set^{\Del^{op}}_{\Gam}$ of simplicial (discrete) $\Gam$-sets. We will use the terms strict fibrations/strict weak equivalence for the strict model structure and fibrtion/weak equivalence for the standard model structure. Cofibrations in both case are simply injective maps.

As mentioned in $~\cite{Goe95}$ the operation of (standard) fibrant replacement can actually be done functorially, yielding a functor
$$ (\bullet)^{fib}: \Set^{\Del^{op}}_{\Gam} \lrar \Set^{\Del^{op}}_{\Gam} $$
One then defines the \textbf{homotopy fixed points} of $\X$ to be the fixed points of $X^{fib}$, i.e.:
$$ \X^{h\Gam} \x{def}{=} \left(\X^{fib}\right)^{\Gam} $$
Since fibrant objects are in particular strictly fibrant one gets that $\X^{h\Gam}$ is always a Kan simplicial set. It can be shown that if two maps $f,g:\X \lrar \Y$ induce the same map in $\Ho^{\standard}\left(\Set^{\Del^{op}}_{\Gam}\right)$ then the induced maps in
$$ f_*,g_*: X^{h\Gam} \lrar Y^{h\Gam} $$
are simplicially homotopic. In particular this will be true for $f,g$ which are simplicially homotopic.

Since we will be working only with strictly bounded simplicial sets we would like to have a convenient formula for the homotopy fixed points in that case. Note that if $S$ is strictly fibrant then for for every open normal subgroup $\Lam \fns \Gam_K$ the corresponding fixed points $S^{\Lam}$ form a Kan simplicia

l set. We then have the following formula:
\begin{thm}\label{t:goerss-formula}
Let $\Y$ be a simplicial $\Gam$-set whose underlying simplicial set is Kan. Let $D^{\Gam}_{fin} \subseteq \Ho\left(\Set^{\Del^{op}}_{\Gam}\right)$ be the full subcategory spanned by Kan contractible objects which are levelwise finite. Then one has an isomorphism of sets
$$ \pi_0\left((\Y)^{h\Gam}\right) \simeq \colim_{\E \in D^{\Gam}_{fin}}[\E,\Y]_{\Gam} $$
If in addition $\Y$ is also \textbf{strictly bounded} then the formula can be refined to
$$  \pi_0\left((\Y)^{h\Gam}\right) \simeq \colim_{\Lam \fns \Gam} [\E(\Gam/\Lam),\Y]_{\Gam} \simeq $$
$$ \colim_{\Lam \fns \Gam} \left[\E(\Gam/\Lam),\Y^{\Lam}\right]_{\Gam} \simeq \colim_{\Lam \fns \Gam}\pi_0\left(\left(\Y^{\Lam}\right)^{h(\Gam/\Lam)}\right) $$
\end{thm}
We delay the proof of this formula (which is completely independent of the rest of this chapter) to section \S $4$ (see ~\ref{t:goerss-formula}).

Now let $K$ be a number field and $X$ an algebraic variety over $K$. We will denote by $\Gam_K$ the absolute Galois group of $K$ and similarly by $\Gam_L \fns \Gam_K$ the absolute Galois group of a finite Galois extension $L/K$.

Every $K$-rational point in $X$ is a map
$$ \spec(K) \lrar X $$
in $\Var/K$. Applying the functor ${\acute{E}t_{/K}}^{\natural}$ we get a map
$$ {\acute{E}t_{/K}}^{\natural}(\spec(K)) \lrar {\acute{E}t_{/K}}^{\natural}(X) $$
In order to describe mappings from ${\acute{E}t_{/K}}^{\natural}(\spec(K))$ to some pro-object we need first to understand the pro-object ${\acute{E}t_{/K}}^{\natural}(\spec(K))$ itself.

We start with the simpler task of describing the pro-object ${\acute{E}t_{/K}}(\spec(K))$: The site of finite \'{e}tale varieties over $\spec(K)$ can be identified with the site of finite discrete $\Gam_K$-sets (and surjective maps as coverings) via the fully faithful functor ${\pi_0}_{/K}$. This means that the functor $\pi^{\Del^{op}}_{/K}$ induces a fully faithful embedding of $I(\spec(K))$ into $\Ho\left(\Set^{\Del^{op}}_{\Gam_K}\right)$ whose essential image consists of the full subcategory $D^{\Gam_K}_{fin} \subseteq \Ho\left(\Set^{\Del^{op}}_{\Gam_K}\right)$. In particular we have an equivalence of categories
$$ {\pi_0}_{/K}:I(\spec(K)) \x{\sim}{\lrar} D^{\Gam_K}_{fin} $$
Now for a pro-finite group $\Gam$ we will denote by $\E\Gam$ the pro-object
$$ \E\Gam \x{def}{=} \{\Ex^{\infty}(S)\}_{S \in D^{\Gam}_{fin}} \in \Pro\Ho\left(\Set^{\Del^{op}}_{\Gam}\right) $$
We then have an isomorphism of pro-objects
$$ \acute{E}t_{/K}(\spec(K)) \simeq \E(\Gam_K) $$

A simple corollary of this observation plus formula ~\ref{t:goerss-formula} is the following:
\begin{cor}
Let $\Y_I = \{\Y_\alp\}_{\alp \in I} \in \Pro\Ho^{\strict}\left(\Set^{\Del^{op}}_{\Gam_K}\right)$ be an object. Then
$$ \Hom_{\Pro\Ho^{\strict}\left(\Set^{\Del^{op}}_{\Gam_K}\right)}\left(\acute{E}t_{/K}(\spec(K)), \Y_I\right) \simeq \lim_{\alp \in I}{Y_\alp^{h\Gam_K}} $$
\end{cor}

We now wish to compute $\E\Gam^{\natural}$. Since the skeletons of the simplicial sets in $\E\Gam$ are finite the action of $\Gam$ on each specific skeleton factors through a finite quotient. This means that the action of $\Gam$ on each of the spaces of $\E\Gam^{\natural}$ factors through a finite quotient $G = \Gam/\Lam$ for some $\Lam \fns \Gam$.

Now for every such $G$ one has the Kan contractible levelwise finite simplicial $\Gam$-set $\E G = \cosk_0(G)$ with the $\Gam$-action given by pulling the standard $G$-action. Then $\Ex^{\infty}(\E G)$ appears in the diagram of $\E\Gam^{\natural}$. Since every $\E G$ is strictly fibrant as a simplicial $\Gam$-set (every fixed point space is either empty or equal to all of $\E G$, which is Kan) the map $\E G \lrar \Ex^{\infty}(\E G)$ admits a simplicial homotopy inverse $\Ex^{\infty}(\E G) \lrar \E G$ which is unique up to simplicial homotopy. This gives a map in $\Pro\Ho\left(\Set^{\Del^{op}}_{\Gam}\right)$:
$$ \vphi: \{\E(\Gam/\Lam)\}_{\Lam \fns \Gam} \lrar \E\Gam^{\natural} $$

We claim that $\vphi$ is actually an \textbf{isomorphism}: every Kan contractible simplicial $G$-set $\X$ admits a map $\E G \lrar \X$ which is unique up to simplicial homotopy (this can be seen using the projective model structure on simpicial $G$-sets). These maps fit together to give an inverse to $\vphi$. This finishes the computation of $\E\Gam^{\natural}$.

We now wish to compute the set of maps from
$$ \acute{E}t^{\natural}(\spec(K)) \cong \E(\Gam_K)^{\natural} \cong \{\E G_L\}_{L/K} $$ to a pro-object of the form $X_I = \{\SimpS_\alp\}_{\alp \in I} $ where each $X_\alp$ is strictly bounded. By definition this morphism set is the set
$$ \lim \limits_{\alp} \colim \limits_{L/K} \left[\E G_L, \SimpS_\alp\right]_{\Gam_K}$$
where $L/K$ runs over all finite Galois extensions, $G_L = \Gam_K/\Gam_L$ and $[\X,\Y]_{\Gam_K}$ denotes simplicial homotopy classes of maps. Since $\Gam_L$ stabilizes $\E G_L$ this is the same as
$$
\lim \limits_{\alp} \colim \limits_{L/K} \left[\E G_L, \SimpS_\alp^{\Gam_L}\right]_{G_L} =
\lim \limits_{\alp} \colim \limits_{L/K} \pi_0\left(\left(\SimpS_\alp^{\Gam_L}\right)^{hG_L}\right) = $$
$$ \lim \limits_{\alp} \pi_0\left(\SimpS_\alp^{h\Gam}\right) $$
where the last equality is obtain by applying formula ~\ref{t:goerss-formula} to the strictly bounded simplicial $\Gam$-set $\X_\alp$.

We summarize the above computation in the following definition
\begin{define}\label{d:homotopy-fixed-points-set}
Let $\SimpS_I = \{\SimpS_\alpha\}_{\alpha\in I} \in \Pro\Ho\left(\Set^{\Del^{op}}_{\Gam_K}\right)$ be an object. We define the $\Gam$-\textbf{homotopy fixed points set} of $\SimpS_I$, denoted by $\SimpS_I\left(\E\Gam^{\natural}\right)$, to be
$$ \SimpS_I\left(\E\Gam^{\natural}\right) = \lim \limits_{\alp}  \pi_0\left(\SimpS_\alp^{h\Gam}\right) $$
If $\Gam = \Gam_K$ is the absolute Galois group of a field $K$ then we will also denote this set by $\SimpS_I\left(hK\right)$.

Note that if all the simplicial $\Gam$-sets $\X_\alp$ are strictly bounded then we have the isomorphism of sets
$$ \SimpS_I\left(\E\Gam^{\natural}\right) \simeq \Map_{\Pro\Ho\left(\Set^{\Del^{op}}_{\Gam}\right)}
\left(\E\Gam^{\natural},\SimpS_I\right) $$
\end{define}

When $\SimpS_I$ is the \'{e}tale homotopy type of an algebraic variety we use the following abbreviations
$$ X(hK) = \acute{E}t_{/K}^{\natural}(X)\left(hK\right) = \lim_{\U \in I(X),k \in \NN} \pi_0\left(\X_{\U,k}^{h\Gam_K}\right) $$
and
$$ X^{n}(hK) = \acute{E}t_{/K}^{n}(X)\left(hK\right) = \lim_{\U \in I(X),k \leq n} \pi_0\left(\X_{\U,k}^{h\Gam_K}\right) = \lim_{\U \in I(X)} \pi_0\left(\X_{\U,n}^{h\Gam_K}\right) $$

\begin{rem}
Let $\X$ be a simplicial $\Gam$-set. By considering $\X$ as a pro-object in $\Ho\left(\Set^{\Del^{op}}_{\Gam_K}\right)$ in a trivial way we will write
$$ \X\left(\E\Gam^{\natural}\right) =  \pi_0\left(\SimpS_\alp^{h\Gam}\right) $$
When $\Gam = \Gam_K$ is the absolute Galois group of a field we will use the notation
$$ \X(hK) = \X\left(\E\Gam^{\natural}\right) $$
\end{rem}

Summarizing the discussion so far we see that for every $0 \leq n \leq \infty$ we get a natural map
$$ h_n:X(K) \lrar X^n(hK) $$

It is useful to keep in mind the most trivial example:
\begin{lem}\label{l:point-is-point}
$$ \spec(K)^n\left(hK\right) = * $$

\end{lem}
\begin{proof}
We know that ${\acute{E}t^{\natural}_{/K}}(\spec(K)) \cong \{\E G_L\}_{L/K}$ where $L$ runs over all finite Galois extensions of $K$ and $G_L$ is the Galois group of $L$ over $K$. Since each $\E G_K$ is Kan contractible and strictly bounded we see that
$$ \spec(K)^n\left(hK\right) = \spec(K)(hK) = \lim_{L/K} \E G_L^{h\Gam} \cong * $$

\end{proof}

\begin{rem}
Notions for homotopy fixed points for action of pro-finite groups on pro-spaces where studied by Fausk and Isaksen in ~\cite{FIs07} and by Quick in ~\cite{Qui09}. Fausk and Isaksen's approach is to put a model structure on the category of pro-spaces with a pro-finite group action. In Quick's work one replaces pro-spaces by simplicial pro-sets. Both approaches use a model structure in order to produce a homotopy fixed point \textbf{space}, and not just a set as we have in Definition ~\ref{d:homotopy-fixed-points-set}.

In this paper we work with (a relative variation of) the \'{e}tale homotopy type which is a pro object in the \textbf{homotopy} category of spaces, and not of spaces. Hence one cannot apply to it either of the theories above. In order to use Fausk and Isaksen's approach one would need to replace the \'{e}tale homotopy type with the \'{e}tale topological type (see Friedlander ~\cite{Fri82}) which is a pro-space. Alternatively one can convert the \'{e}tale topological type to a simplicial pro-set and use Quick's theory. There are some drawbacks for working without a model structure. For example one gets only a homotopy fixed points set and not a homotopy fixed point space. However for our needs in this paper we found the aforementioned model structures
uncomfortable to use.

Lately, Ilan Barnea and the second author (~\cite{BSc11}) defined a new model structure which is more suitable for our needs. A detailed description of how to use this model structure in the context of this work would be published by the authors in another publication.
\end{rem}

\subsubsection{ P-adic Homotopy Fixed Points }

Let $K$ be a number field and $K_\nu$ a completion of $K$. We denote by $\Gam_\nu < \Gam_K$ the decomposition group. Then for every $0 \leq n \leq \infty$ we get a map

$$ h_n:X(K_\nu) \lrar X\left(hK_\nu)\right) $$
We shall also use the notation $X^n\left(hK_\nu\right)$ to denote ${\acute{E}t^{n}_{/K}}(X)\left(E\left(\Gam_\nu^{\natural}\right)\right)$.

Taking into account all the completions of $K$ we get a commutative diagram

$$ \xymatrix{
X(K) \ar[r]^{h_n}\ar[d]^{\loc} & X^n(hK) \ar[d]^{\loc_{h,n}} \\
\prod \limits_\nu X(K_\nu) \ar[r]^{h_n} & \prod \limits_\nu X^n(hK_\nu) \\
}$$

Note that we abuse notation and use $h_n$ for the rational case, the p-adic case and the product of p-adics case.

We define the \textbf{p-adic rationally homotopic set}
$$ \left(\prod \limits_\nu X(K_\nu)\right)^{h,n} \subseteq \prod \limits_\nu X(K_\nu) $$
to be the set of all elements $(x_\nu) \in \prod \limits_\nu X(K_\nu)$ such that $h_n((x_\nu)) \in \im(\loc_{h,n})$. Note that
$$ X(K) \subseteq \left(\prod \limits_\nu X(K_\nu)\right)^{h,n} $$

\subsubsection{Adelic Homotopy Fixed Points}
When studying the local global principle on a variety $X$ one sees that in general it is better to work with the set of adelic points on $X(\A)$ rather then the entire product $\prod \limits_\nu X(K_\nu)$. Similarly we would like to replace the set
$ \prod \limits_\nu X^n(hK_\nu) $
with an analogous set $X^n(h\A)$.

Before defining such a notion for an object in $\Pro\Ho^{\strict}(\Set^{{\Del}^{op}}_{\Gam_K})$ we shall define it for the more simple case of a simplicial $\Gam_K$-set.

\begin{define}
Let $K_\nu$ be a non-archimedean local field, $I_\nu \lhd \Gam_\nu$ the inertia group and $$ \Gam^{ur}_\nu = \Gam_\nu/I_\nu $$
the unramified Galois group. Let $\SimpS$ be a simplicial $\Gam_\nu$-set. We define the \textbf{unramified} $\Gam_{K_\nu}$-homotopy fixed points to be the simplicial set

$$ \SimpS^{h^{ur}\Gam_\nu} = \left(\SimpS^{I_\nu}\right)^{h\Gam^{ur}_\nu} $$
\end{define}

\begin{define}

Let $K$ be a number field, $S$ a set of places of $K$ and $\SimpS$ a simplicial $\Gam_K$-set. We define the restricted product of $S$-homotopy fixed points space to be

$$\SimpS^{h\A_S} =  \hocolim_T \prod \limits_{\nu\in T} \SimpS^{h\Gam_\nu} \times \prod \limits_{\nu \in S \backslash T} \SimpS^{h^{ur}\Gam_\nu}$$
where $T$ runs over all the finite subsets of $S$. We also denote
$$\SimpS(h\A_s) = \pi_0\left(\SimpS^{h\A_S}\right)$$
\end{define}
\begin{rem}
Note that as a restricted product of discrete sets $\SimpS(h\A_s)$ carries the restricted product topology.
\end{rem}

Note that
$$\pi_n\left(\SimpS^{h\A_S}\right) = \prod \limits_{\nu \in S}' \pi_n\left(\SimpS^{h\Gam_\nu}\right)$$
when the restricted product is taken with respect to the subsets
$$ \im\left[\pi_n\left(\SimpS^{h^{ur}\Gam_\nu}\right) \lrar  \pi_n\left(\SimpS^{h\Gam_\nu}\right)\right] $$

When $S$ is the set of all places we denote $\SimpS(h\A_s)$ and $\SimpS^{h\A_S}$ by $\SimpS(h\A)$ and $\SimpS^{h\A}$ respectively. Similarly when $S$ is the set of all finite places we denote $\SimpS(h\A_s)$ and $\SimpS^{h\A_S}$ by $\SimpS(h\A_f)$ and $\SimpS^{h\A_f}$ respectively.

\begin{define}
Let $K$ be a number field. Let $\SimpS_I = \{\SimpS_\alpha\}_{\alpha \in I}$ be a pro-$\Gam_K$-simplicial set. We define
$$ \SimpS_I(h\A_s) = \lim \limits_{\alp\in I}  \SimpS_\alp(h\A_s) $$

\end{define}
In the case where we are dealing with the \'{e}tale homotopy type of an algebraic variety $X$ over $K$ we abbreviate
 $$ X^n(h\A) = {\acute{E}t^{n}_{/K}}(X)\left(h\A\right) $$
 and
 $$ X(h\A)= X^{\infty}(h\A)  $$

\begin{lem}\label{l:rational-is-h-adelic}
Let $K$ be number field and $\SimpS$ a strictly bounded $\Gam_K$-simplicial set. Then the natural map
$$ \loc: \SimpS\left(hK\right)  \lrar \prod \limits_{\nu}\SimpS\left(hK_\nu\right) $$
factors through a natural map
$$ f_0: \SimpS\left(hK\right) \lrar \SimpS\left(h\A\right) $$.
\end{lem}
\begin{proof}
Let $L/K$ be finite extension and $T_L$ is the set of places of $K$ that ramify $L$. Since
for $\nu \notin T_L$ we have $I_\nu < \Gam_L$ there is a natural map:
$$ f_L:\SimpS^{\Gam_L} \lrar \prod \limits_{\nu\in T_L} \SimpS \times \prod \limits_{\nu \notin T_L} \SimpS^{I_\nu} $$
Now this map induces a map

$$ f: (\SimpS^{\Gam_L})^{hG_L} \lrar
\prod \limits_{\nu\in T_L} \SimpS^{h\Gam_\nu} \times \prod \limits_{\nu \notin T_L}\left(\SimpS^{I_\nu}\right)^{h\Gam^{ur}_\nu} $$

By passing to homotopy colimit on $L$ we get a map

$$ f: \SimpS^{h\Gam_K} \lrar \SimpS^{h\A} $$
and we can choose $f_0 = \pi_0(f)$.
\end{proof}

\begin{lem}\label{l:continuous}
Let $K_\nu$ be a local field, $X$ a variety over $K_\nu$ and $\U_\bullet \lrar X$ an \'{e}tale hypercovering. Then for every $n \geq 0$ the map
$$ h: X(K_\nu) \lrar \SimpS_{\U,n}\left(K_\nu^h\right) $$
is continuous (where $X(K_\nu)$ inherits a natural topology from the topology of $K_\nu$ and $\SimpS_{\U,n}\left(K_\nu^h\right)$ is discrete).
\end{lem}
\begin{proof}
Let $x\in X(K_\nu)$ be a point. It is enough to find a neighborhood $V$ of $x$ in $X(K_v)$ (with respect to the $K_v$-topology) such that for every $y \in V$ we have $h(x) = h(y)$.
Consider $x$ as a map
$$ x: \spec K_v \lrar X $$
and let $x^*(\U)$ be the hypercovering of $\spec K_v$ which is the pullback of $\U$ by $x$. The underlying simplicial set of $\SimpS_{x^*(\U),n}$ is contractible and hence $\SimpS_{x^*(\U),n}^{h\Gamma_v}$ is also contractible. The map $x$ gives a map
$$ \tilde{x}:\SimpS_{x^*(\U),n} \lrar \SimpS_{\U,n} $$
such that the induced map
$$\tilde{x}^{h\Gamma_v}:\SimpS_{x^*(\U),n}^{h\Gamma_v} \lrar \SimpS_{\U,n}^{h\Gamma_v} $$
sends $\SimpS_{x^*(\U),n}^{h\Gamma_v}$ to the connected component $h(x)$.

Now given any \'{e}tale map $f:U \lrar X$ the inverse function theorem insures that there is an open neighborhood $V$ of $x$ in the $K_\nu$-topology such that for every $y \in V$ there is a natural Galois equivariant identification of the fibers
$$ F_y:f^{-1}(x) \lrar f^{-1}(y) $$
This identification takes any point in $f^{-1}(x)$ to the unique point in $f^{-1}(y)$ sitting with it in the same connected component of $f^{-1}(V)$.

This means that there is an open neighborhood $V_{\U,n}$ of $x$ in the $K_\nu$-topology such that for every $y \in V_{\U,n}$ there is a natural Galois equivariant map
$$ F_{\U,n}:\sk_n (x^*(\U)) \lrar \sk_n (y^*(\U)) $$
Now since for every $y\in V_{\U,n}$ the map $\tilde{x}$ factors trough the map $\tilde{y}$ we see that $\tilde{x}^{h\Gamma_v}$ and $\tilde{y}^{h\Gamma_v}$ must land in the same connected component and so
$$ h(y) = h(x) $$
\end{proof}

\begin{lem}\label{l:adelic-is-adelic}
Let $X$ be an algebraic variety over a number field $K$ and let $\U_\bullet \lrar X$ be a hypercovering. Then the natural continuous map
$$ X(\A) \lrar \prod \limits_\nu \SimpS_{\U,n}\left(hK_\nu\right) $$
factors through a natural continuous map
$$ h_{\U,n}: X(\A) \lrar \SimpS_{\U,n}\left(h\A\right) $$

\end{lem}

\begin{proof}
First choose some model for $\U_\bullet \lrar X$ over $\spec(\ZZ)$
$$ (\U_\ZZ)_\bullet \lrar X_\ZZ $$
which is \'{e}tale outside a finite set of bad places $T_1$. We will show that
$$ h(X_\ZZ(O_\nu))
\subseteq \im\left[
\pi_0\left(\SimpS_{\U,n}^{h^{ur}\Gam_\nu}\right) \lrar \SimpS_{\U,n}\left(hK_\nu\right) \right] $$
for almost all $\nu$. Choose some finite extension $L/K$ such that $\Gam_L$ fixes $\sk_n(\SimpS_\U)$ and denote by $T_0$ the set of ramified places in $L$.

Consider now a place $\nu \notin T_0 \cup T_1$. Since $\nu  \notin T_0$ we have
$$ \SimpS_{\U,n}^{I_\nu} = \SimpS_{\U,n} $$
Hence all we need to show is that the homotopy fixed point $h_n(x)$ comes from a homotopy fixed point of the quotient group $\Gam_\nu^{ur} = \Gam_\nu/I_\nu$.

Since $\nu \notin T_1$ pulling back $\U_\ZZ$ to a point $x \in X_{\ZZ}(O_\nu)$ will yield a (contractible) simplicial set $E$ that is stabilized by some unramified extension of $K_\nu$. Then $h_n(x)$ is the image of the unique point $\pi_0(E^{h \Gam_\nu}) = *$ in $\SimpS_{\U,n}$. Since $E$ is stabilized by $I_\nu$ this homotopy fixed point comes from $\Gam_\nu^{ur}$ and we are done.

This means that in order to prove Lemma ~\ref{l:adelic-is-adelic} it will be enough to prove the following general lemma on the behavior of maps between restricted products of topological spaces:
\begin{lem}\label{l:restricted-maps}
Let
$$ \{f_\lam : X_\lam \lrar Y_\lam\}_{\lam\in \Lam} $$
be a family of continuous maps of topological spaces and let
$$ \{O_\lam,U_\lam\}_{\lam\in \Lam} $$
be a family of open subsets $O_\lam \subseteq X_\lam,U_\lam \subseteq Y_\lam $.

Assume that $f(O_\lam) \subseteq U_\lam$ for almost all $\lam$. Then the map
$$ F_\Lam = \prod \limits_{\lam \in \Lam} f_\lam : \prod \limits_{\lam \in \Lam}X_\lam \lrar \prod \limits_{\lam \in \Lam} Y_\lam $$
induces a continuous map
$$ F_\Lam : \prod' \limits_{\lam \in \Lam}X_\lam \lrar \prod' \limits_{\lam \in \Lam} Y_\lam$$
where the restricted product is taken with respect to $O_\lam, U_\lam$ respectively.
\end{lem}
\begin{proof}
It is clear that
$$ F_\Lam\left(\prod' \limits_{\lam \in \Lam}X_\lam\right)  \subseteq \prod' \limits_{\lam \in \Lam}  Y_\lam $$
Hence it is enough to show that $F_\Lam$ is continuous.

We need to show that if $S \subseteq \Lam$ is a finite set and $\{A_\lam \subseteq Y_\lam\}_{\lam\in S}$ are open then the set
$$ (F_\Lam)^{-1}\left(\prod\limits_{\lam \in S} A_\lam \times \prod\limits_{\lam \in \Lam\bksl S}  X_\lam\right) $$
is open in $\prod' \limits_{\lam \in \Lam}X_\lam $. Let
$$(x_\lam)_\lam \in (F_\Lam)^{-1}\left(\prod\limits_{\lam \in S} A_\lam \times \prod\limits_{\lam \in \Lam\bksl S}  U_\lam\right)$$
and let $T \subseteq \Lam$ be a finite set containing $S$ such that
$$ x_\lam \in O_\lam, f(O_\lam) \subseteq U_\lam $$
for every $\lam \in \Lam \bksl T $. Note that such a set exist due the assumptions of the lemma. Consider the set

$$ N_x = \prod \limits_{\lam \in S} f_\lam^{-1}(A_\lam) \times \prod \limits_{\lam \in T \bksl S} f_\lam^{-1}(U_\lam) \times \prod \limits_{\lam \in \Lam \bksl T } O_\lam $$

It is clear that $N_x$ is an open neighborhood of $x$ and that
$$ N_x \subseteq (F_\Lam)^{-1}\left(\prod\limits_{\lam \in S} A_\lam \times \prod\limits_{\lam \in \Lam\bksl S}  X_\lam\right) $$
\end{proof}

\end{proof}

By Lemma ~\ref{l:adelic-is-adelic} and Lemma ~\ref{l:rational-is-h-adelic} we have a commutative diagram
$$ \xymatrix{
X(K) \ar[r]^-{h_n}\ar[d]^{\loc} & X^{n}\left(hK\right) \ar[d]^{\loc_{h,n}} \\
X(\A) \ar[r]^-{h_n} & X^{n}\left(h\A\right) \\
}$$

Note that we abuse notation and use $h_n$ for both the rational and adelic cases.

We denote by $X(\A)^{h,n} \subseteq X(\A)$ the set of adelic points whose image in $X^{n}\left(h\A\right)$ lies in the image of $\loc_{h,n}$. Note that
$$ X(\A)^{h,n} = \left(\prod \limits_\nu X(K_\nu)\right)^{h,n} \cap X(\A) \subseteq X(\A) $$
and also
$$ X(K) \subseteq X(\A)^{h} \subseteq \cdots \subseteq X(\A)^{h,2} \subseteq \cdots  \subseteq X(\A)^{h,1} \subseteq X(\A) $$

We denote by $X(\A)^{h,\infty}$ simply by $X(\A)^h$. We call the elements of $X(\A)^h$ the set of \textbf{homotopically rational points}.

\begin{define}
We say that the lack of $K$-rational points in $X$ is explained by the \textbf{\'{e}tale homotopy obstruction} if the set $X(\A)^h$ is empty.
\end{define}

\subsection{ The \'{E}tale Homology Obstruction }
Let $\Mod_\Gam$ be the category of discreet $\Gam$-modules.
Consider the augmented functor
$$ \ZZ: \Set_{\Gam} \lrar \Mod_{\Gam} $$
which associates to $\Gam$-set $A$ the free abelian group $\ZZ A$ generated from $A$ with the induced Galois action. The terminal map $A \lrar \{*\}$ defines a map $\ZZ A \lrar \ZZ$ which we will call the \textbf{degree} map. Note that the image of the augmentation map $A \lrar \ZZ A$ lies in the subset of elements of degree $1$.

For a simplicial $\Gam$-set $\SimpS$ we will denote by $\ZZ \SimpS$ the simplicial $\Gam$-module obtained by applying the $\ZZ$ functor levelwise, i.e.
$$ (\ZZ \SimpS)_n = \ZZ (\SimpS_n) $$
The terminal map $\SimpS \lrar *$ induces a map from $\ZZ \SimpS$ to the constant simplicial $\Gam$-module $\ZZ$ (this is the discrete simplicial $\Gam$-module with the trivial action). Note that again the elements in the image of the augmentation map have degree $1$.

The homotopy groups of $\ZZ \SimpS$ can be naturally identified with the homology of $\SimpS$ and the augmentation map induces the Hurewicz map
$$ \pi_*(\SimpS) \lrar \pi_*(\ZZ \SimpS) = H_*(\SimpS) $$
We shall refer to the augmentations map as the Hurewicz map as well.

We can now consider the $\ZZ$-variant of the functor ${\acute{E}t_{/K}}$ applying the $\ZZ$ functor on each simplicial set in the diagram:
$$ \ZZ{\acute{E}t_{/K}} = \{\ZZ \SimpS_{\mcal{U}}\}_{\mcal{U} \in I(X)} $$
As before we prefer to work with bounded simplicial $\Gam$-sets and so we replace this object by its truncated Postnikov tower
$$(\ZZ{\acute{E}t_{/K}})^{n} = \{P_k(\ZZ \SimpS_{\mcal{U}})\}_{\mcal{U} \in I(X), k \leq n} $$
as well as the full Postnikov tower
$$ (\ZZ{\acute{E}t_{/K}})^{\natural} = (\ZZ{\acute{E}t_{/K}})^{\infty} = \{P_k(\ZZ \SimpS_{\mcal{U}})\}_{\mcal{U} \in I(X), k \in \NN} $$
For every $0 \leq n \leq \infty$ we have a natural transformation
$$ {\acute{E}t^{n}_{/K}}(X) \lrar (\ZZ{\acute{E}t_{/K}})^{n}(X) $$

and so we can consider the commutative diagram
$$ \xymatrix{
X(K) \ar[r]^-{h_n}\ar[d] & X^{n}\left(hK\right) \ar[r]\ar^{\loc_{h,n}}[d] & X^{\ZZ,n}\left(hK\right) \ar^{\loc_{\ZZ h, n}}[d] \\
X(\A) \ar^-{h_n}[r] & X^{n}\left(h\A\right) \ar[r] & X^{\ZZ,n}\left(h\A\right) \\
}$$
Where
$$X^{\ZZ,n}\left(hK\right)  = ((\ZZ{\acute{E}t_{/K}})^{n}(X))(hK)$$
and
$$X^{\ZZ,n}\left(h\A\right)  = ((\ZZ{\acute{E}t_{/K}})^{n}(X))(h\A).$$

We say that an adelic point
$$ (x_\nu) \in X(\A) = \prod'_\nu X(K_\nu) $$
is \textbf{n-homologically rational} if its image in
$X^{\ZZ,n}\left(h\A\right)$ is rational, i.e. is in the image of $\loc_{\ZZ h, n}$.

We denote by $X(\A)^{\ZZ h,n} \subseteq X(\A)$ the set of $n$-homologically rational points. We also denote
$$ X(\A)^{\ZZ h} = X(\A)^{\ZZ h,\infty} $$

\begin{define}
We say that the lack of $K$-rational points in $X$ is explained by the \textbf{\'{e}tale homology obstruction} if the set $X(\A)^{\ZZ h}$ is empty.
\end{define}

From the above discussion we immediately see that we have the following diagram of inclusions

$$
\xymatrix{
\empty & X(\A)^{\ZZ h} \ar@{^{(}->}[r] & \cdots\ar@{^{(}->}[r] & X(\A)^{\ZZ h,2} \ar@{^{(}->}[r] & X(\A)^{\ZZ h,1}\ar@{^{(}->}[r]& X(\A)  \\
X(K) \ar@{^{(}->}[r] & X(\A)^{h} \ar@{^{(}->}[u] \ar@{^{(}->}[r] & \cdots \ar@{^{(}->}[r] & X(\A)^{h,2} \ar@{^{(}->}[r] \ar@{^{(}->}[u] & X(\A)^{h,1}\ar@{^{(}->}[u] & \empty }
$$

and so the \'{e}tale $n$-homology obstruction is in general weaker then the \'{e}tale $n$-homotopy obstruction.

\subsection{ The Single Hypercovering  Version }
It will sometimes be convenient to consider the information obtained from a single hypercovering. Let $X$ be an algebraic variety over $K$ and
$$ {\acute{E}t^{\natural}_{/K}}(X) = \{\SimpS_{\mcal{U},n}\}_{\mcal{U} \in I(X),n\in \NN} $$
For each hypercovering $\mcal{U} \in I(X)$ and $n \in \NN$ we can consider the diagram

$$ \xymatrix{
X(K) \ar^-{h_{\U,n}}[r]\ar[d] & \SimpS_{\mcal{U},n}\left(hK\right) \ar^{\loc_{\U,n}}[d] \\
X(\A) \ar^-{h_{\U,n}}[r] & \SimpS_{\mcal{U},n}\left(h\A\right) \\
}$$
We denote by $X(\A)^{\U,n}$ the set of adelic points $(x_\nu) \in X(\A)$ whose image in $\SimpS_{\U,n}\left(h\A\right)$ is rational (i.e. is in the image of $\loc_{\U,n}$).

Similarly we can consider the diagrams
$$ \xymatrix{
X(K) \ar^{h_{\ZZ \U,n}}[r]\ar[d] & \ZZ\SimpS_{\mcal{U},n}\left(hK\right) \ar^{\loc_{\ZZ \U,n}}[d] \\
X(\A) \ar^{h_{\ZZ \U,n}}[r] & \ZZ\SimpS_{\mcal{U},n}\left(h\A\right) \\
}$$
We denote by $X(\A)^{\ZZ\U,n}$ the set of adelic points $(x_\nu) \in X(\A)$ whose image in $\ZZ \SimpS_{\mcal{U},n}\left(h\A\right)$ is rational (i.e. is in the image of $\loc_{\ZZ \U,n}$).

Note that for every object $\{\SimpS_\alpha\}_{\alp \in I} \in \Pro\Ho\left(\Set^{\Del^{op}}_{\Gam}\right)$ and every $\alp_0 \in I$ we have a natural map

$$\{\SimpS_\alpha\}_{\alp \in I} \lrar \SimpS_{\alp_0}$$
Thus
$$ X(\A)^{h,n} \subseteq  X(\A)^{\U,n} $$
and
$$ X(\A)^{\ZZ h,n} \subseteq  X(\A)^{\ZZ \U,n} $$

for every hypercovering $\U_\bullet \lrar X$.

\subsection{ Summary of Notations }
The following is a summary of the notations we've had so far for maps from points to corresponding homotopy fixed points.
$$ h: X(K) \lrar X(hK) $$
$$ h_n: X(K) \lrar X^{n}(hK) $$
$$ h_{\U}: X(K) \lrar \SimpS_{\U}(hK) $$
$$ h_{\U,n}: X(K) \lrar \SimpS_{\U,n}(hK) $$
$$ h_{\ZZ}: X(K) \lrar X^{\ZZ}(hK) $$
$$ h_{\ZZ,n}: X(K) \lrar X^{\ZZ,n}(hK) $$
$$ h_{\ZZ\U}: X(K) \lrar \ZZ\SimpS_{\U}(hK) $$
$$ h_{\ZZ\U,n}: X(K) \lrar P_n(\ZZ\SimpS_{\U})(hK) $$
By abuse of notation we will use the exact same notation for the adelic version of all of these maps.

\section{ Homotopy Fixed Points for Pro-Finite Groups }\label{s:homotopy-fixed-points}
Let $\Gam$ be a pro-finite group. We start with the basic calculative theorem regarding homotopy fixed points of simplicial $\Gam$-sets in Goerss' model category:
\begin{thm}\label{t:goerss-formula}
Let $\Y$ be a simplicial $\Gam$-set whose underlying simplicial set is Kan. Let $D_{fin} \subseteq \Ho\left(\Set^{\Del^{op}}_{\Gam}\right)$ be the full subcategory spanned by Kan contractible objects which are levelwise finite. Then one has an isomorphism of sets
$$ \pi_0\left((\Y)^{h\Gam}\right) \simeq \colim_{\E \in D_{fin}}[\E,\Y]_{\Gam} $$
If in addition $\Y$ is also \textbf{strictly bounded} then the formula can be refined to
$$ (*)\;\; \pi_0\left((\Y)^{h\Gam}\right) \simeq \colim_{\Lam \fns \Gam} [\E(\Gam/\Lam),\Y]_{\Gam} \simeq $$
$$ \colim_{\Lam \fns \Gam} \left[\E(\Gam/\Lam),\Y^{\Lam}\right]_{\Gam} \simeq \colim_{\Lam \fns \Gam}\pi_0\left(\left(\Y^{\Lam}\right)^{h(\Gam/\Lam)}\right) $$
\end{thm}

\begin{proof}
We use a formalism developed by Brown in ~\cite{Bro73} called a \textbf{category of fibrant objects}. This is a notion of a category with weak equivalences and fibrations satisfying certain properties (see ~\cite{Bro73} pages 420-421).

We will apply this formalism to an example analogous to one appearing in $~\cite{Bro73}$ itself: let $C \subseteq \Set^{\Del^{op}}_{\Gam}$ be the full subcategory consisting of simplicial $\Gam$-sets whose underlying simplicial set is Kan. We will declare a morphism in $C$ to be a fibration if it induces a Kan fibration on the underlying simplicial sets and a weak equivalence if it induces a weak equivalence on the underlying simplicial set (so in particular weak equivalences in $C$ coincide with those of the standard model structure). It can be shown that these choices endow $C$ with the structure of a category with fibrant objects.

Now let $\X,\Y$ be two simplicial $\Gam$-sets and let $\vphi:\X' \x{\sim}{\lrar} \X$ be a weak equivalence (with respect to the standard model structure). Let $g:\X' \lrar \Y$ be a map. Then there exists a unique map $h: \X \lrar \Y^{fib}$ such that the square
$$ \xymatrix{
\X' \ar^{g}[r]\ar^{\vphi}[d] & \Y \ar^{f}[d]\\
\X \ar^{h}[r] & \Y^{fib}\\
}$$
commutes up to simplicial homotopy. This gives us a map of sets
$$ \left[\X',\Y\right]_\Gam \lrar \left[\X,\Y^{fib}\right]_{\Gam} $$
for every weak equivalence $\vphi: \X' \x{\sim}{\lrar} \X$. Now Theorem $1$ in ~\cite{Bro73} applied to $C$ (taking into account remark $5$ on page $427$ of ~\cite{Bro73}) says the following: if $\X$ and $\Y$ are in $C$, and if we take the colimit over all weak equivalences $\vphi: \X' \x{\sim}{\lrar} \X$ in $C$, then the resulting map
$$ \colim \limits_{\vphi:\X' \x{\sim}{\lrar} \X}\left[\X',\Y\right]_\Gam \lrar \left[\X,\Y^{fib}\right]_{\Gam} $$
is actually an isomorphism of sets.

In particular if we denote by $D \subseteq \Ho\left(\Set^{\Del}_{\Gam}\right)$ the full subcategory of Kan contractible simplicial $\Gam$-sets then we get an isomorphism of sets
$$ \colim \limits_{\E \in D}[\E,\Y]_{\Gam} \x{\simeq}{\lrar} \left[*,\Y^{fib}\right]_{\Gam} = \pi_0\left(\left(\Y^{fib}\right)^{\Gam}\right) = \pi_0\left(\Y^{h\Gam}\right)  $$
This colimit is indexed by a category which is a bit too big for our purposes, but this problem can easily be mended:
\begin{lem}
The subcategory $D_{fin} \subseteq D$ consisting of objects which are levelwise finite is cofinal.
\end{lem}
\begin{proof}
We need to show that every object $\E \in D$ admits a map $\E' \lrar \E$ where $\E' \in D$ is a levelwise finite Kan contractible simplicial $\Gam$-set. We will construct $\E'$ inductively as follows: let $E'_{-1} = \emptyset$ and $f_{-1}: \E'_{-1} \lrar \E$ the unique map. We will extend $\{\E'_{-1}\}$ to an increasing family of simplicial $\Gam$-sets
$$ \E'_{-1} \subseteq \E'_0 \subseteq ... \subseteq \E'_n \subseteq ... $$
such that $\E'_n$ is an $n$-dimensional levelwise finite simplicial $\Gam$-set which is $(n-1)$-connected in the following sense: every map $\partial \Del^m \lrar E'_n$ with $m \leq n$ extends to $\Del^m$. This will guarantee that
$$\E' = \bigcup_n \E'_n $$
is a Kan contractible simplicial $\Gam$-set. Further more we will construct a compatible family of equivariant maps
$$ f_n: \E'_n \lrar \E $$
which will induce one big equivariant map $\E' \lrar \E$.

Now let $n \geq -1$ be a number and we will describe the construction of $\E'_{n+1}$ out of $\E'_n$ and $f_n$. First for a \textbf{finite} simplicial set $\X$ and a simplicial $\Gam$-set $\Y$ we will denote by
$$ \Y^{\X} = \Hom_{\Set^{\Del}}(\X,\Y) $$ the \textbf{set} of maps of simplicial sets from $\X$ to $\Y$. The action of $\Gam$ on $\Y$ induces an action of $\Gam$ on $\Y^{\X}$ rendering it a $\Gam$-set (i.e. all the stabilizers are open because $\X$ is finite).

Now consider the $\Gam$-set
$$ A = \E'^{\partial \Del^{n+1}} \times_{\E^{\partial \Del^{n+1}}} \E^{\Del^{n+1}} $$
This set parameterizes commutative diagrams in the category of simplicial sets of the form
$$ \xymatrix{
\partial \Del^{n+1}  \ar[d]\ar[r] & \E'_n \ar^{f_n}[d] \\
\Del^{n+1}  \ar[r] & \E \\
}$$
Since $\E$ is Kan contractible the map $A \lrar \E'^{\partial \Del^{n+1}}$ is \textbf{surjective}. Since $\E'$ is levelwise finite the set $\E'^{\partial \Del^{n+1}}$ is finite, so we can choose a finite subset $A' \subseteq A$ such that the restricted map
$$ A' \lrar \E'^{\partial \Del^{n+1}} $$
is still surjective (recall that all the orbits in $A$ are finite). Then we get one big commutative diagram of simplicial $\Gam$-sets and \textbf{equivariant} maps
$$ \xymatrix{
\partial \Del^{n+1} \times A' \ar[d]\ar[r] & \E'_n \ar^{f_n}[d] \\
\Del^{n+1} \times A' \ar[r] & \E \\
}$$
and we define $\E'_{n+1}$ to be the pushout of the diagram
$$ \xymatrix{
\partial \Del^{n+1} \times A' \ar[d]\ar[r] & \E'_n \\
\Del^{n+1} \times A'  &  \\
}$$
which admits a natural equivariant extension
$$ f_{n+1}: \E'_{n+1} \lrar \E $$
of $f_n$. Since $A'$ is finite $\E'_{n+1}$ is still levelwise finite. Further more since the map $ A' \lrar \E'^{\partial \Del^{n+1}} $ is surjective we see that every map $\partial \Del^{n+1} \lrar \E'_{n+1}$ extends to all of $\Del^{n+1}$. This finishes the proof of the lemma.
\end{proof}
We now get the first desired formula:
$$ \colim \limits_{\E \in D_{fin}}[\E,\Y]_{\Gam} \cong \pi_0\left(\Y^{h\Gam}\right)  $$

Now for every finite quotient $G = \Gam/\Lam$ for $\Lam \fns \Gam$ we can consider $\E G$ as a simplicial $\Gam$-set with the action pulled for the action of $G$. Then $\E G$ is Kan contractible and levelwise finite, i.e. $\E G \in D_{fin}$. We then have a map of sets
$$ F_{\Y}:
\colim_{\Lam \fns \Gam}\pi_0\left(\left(\Y^{\Lam}\right)^{h(\Gam/\Lam)}\right) \cong
\colim \limits_{\Lam \fns \Gam}\left[\E(\Gam/\Lam),\Y\right]_\Gam \lrar $$
$$ \colim_{\E \in D_{fin}}[\E,\Y]_{\Gam} \cong
\pi_0\left(\Y^{h\Gam}\right) $$
We will finish the proof by showing that if $\Y$ is nice and strictly bounded then $F_{\Y}$ is an isomorphism of sets.

Let $n$ be big enough so that $\pi_n\left(\Y^{\Lam}\right) = 0$ for all $\Lam \fns \Gam$. Then the map
$$ \Y \lrar P_n(\Y) $$
is a strict weak equivalence. Note that both the domain and range of $F_{\Y}$ are invariant under strict weak equivalence in $\Y$, so it will be enough to prove the theorem for $P_n(\Y)$. We will start by showing that $F_{P_n(\Y)}$ is surjective.

Let $g: \E \lrar P_n(\Y)$ be a map. Then $g$ factors
$$ \E \lrar P_n(\E) \x{g'}{\lrar} P_n(\Y) $$
Since $\E \in D_{fin}$ it is in particular nice and so $P_n(\E)$ is excellent, i.e. the action of $\Gam$ on $P_n(\E)$ factors through a finite quotient $G = \Gam/\Lam$. Further more $P_n(\E)$ is also Kan contractible so it admits a $G$-homotopy fixed point, i.e. a map
$$ h:\E G \lrar P_n(\E) $$
The fact that such a map exists \textbf{simplicially} can be seen by using the projective model structure on simplicial $G$-sets. Now the composition
$$ \E G \x{h}{\lrar} P_n(\E) \x{g'}{\lrar} P_n(\Y) $$
and
$$ g : \E \lrar P_n(\Y) $$
both factor through $g'$ and so represent the same element in
$$ \colim_{\E \in D_{fin}}[\E,\Y]_{\Gam} $$
This means that $F_{P_n(\Y)}$ is surjective. Now consider a diagram
$$ \xymatrix{
\E \ar^{p_1}[r]\ar^{p_2}[d] & \E(\Gam/\Lam_1) \ar^{f_1}[d] \\
\E(\Gam/\Lam_2) \ar^{f_2}[r] & P_n(\Y) \\
}$$
which commutes up to simplicial homotopy. Since $P_n(\E(\Gam/\Lam_i)) = \E(\Gam/\Lam_i)$ we get that this diagram factors through a diagram
$$ \xymatrix{
P_n(\E) \ar^{p_1'}[r]\ar^{p_2'}[d] & \E(\Gam/\Lam_1) \ar^{f_1}[d] \\
\E(\Gam/\Lam_2) \ar^{f_2}[r] & P_n(\Y) \\
}$$
which also commutes up to simplicial homotopy since we have a map
$$  P_n(\E) \times I \x{\sim}{\lrar} P_n(\E) \times P_n(I) \cong P_n(\E \times I) $$
Now there exists a $\Lam_3 \subseteq \Lam_1 \cap \Lam_2$ such that the action of $\Gam$ on $P_n(\E)$ factors through $\Gam/\Lam_3$. Since $P_n(\E)$ is Kan contractible it admits a map $h:\E(\Gam/\Lam_3) \lrar P_n(\E)$. Pulling the diagram by $h$ we obtain a new diagram
$$ \xymatrix{
\E(\Gam/\Lam_3) \ar^{p_1'}[r]\ar^{p_2'}[d] & \E(\Gam/\Lam_1) \ar^{f_1}[d] \\
\E(\Gam/\Lam_2) \ar^{f_2}[r] & P_n(\Y) \\
}$$
which commutes up to simplicial homotopy. This shows that $g_1,g_2$ represent the same element in
$$ \colim_{\Lam \fns \Gam}\left[\E(\Gam/\Lam),\Y\right]_\Gam $$
and so $F_{P_n(\Y)}$ is injective. This finishes the proof of the lemma.
\end{proof}

Aside for having an explicit formula we would also like to have a concrete computation aid, in the form of a spectral sequence. The following theorem appears in Goerss ~\cite{Goe95}:
\begin{thm}\label{t:ss}
Let $\SimpS$ be a bounded simplicial $\Gam$-set and let $x \in \SimpS$ be a fixed point of $\Gam$. Then there exists a spectral sequence of pointed sets
$$ E^r_{s,t} \lrar \pi_{s-t}\left(\SimpS^{h\Gam}, x\right) $$
such that
$$ E^2_{s,t} \cong H^t(\Gam,\pi_s(\SimpS, x)) $$
\end{thm}

\begin{rem}
When we write $H^*(\Gamma, A)$ for $\Gamma$ profinite we always mean \textbf{Galois} cohomology. In ~\cite{Goe95} Goerss uses the notation $H_{\Gal}^*(\Gamma, A)$ for this notion. For simplicity of notation we chose to omit the $\Gal$ subscript. Note that when $\Gam$ is a finite group Galois cohomology coincides with regular group cohomology.
\end{rem}

\begin{rem}
The above spectral sequence is of the form used to compute homotopy groups of homotopy limits. It is concentrated in the domain $s \geq t-1$ and its differential $d^r_{s,t}$ goes from $E^r_{s,t}$ to $E^r_{s+r-1,t+r}$. We call such spectral sequences HL-spectral sequences.
\end{rem}

We wish to drop the assumption that $\SimpS$ has an actual fixed point and replace it by the assumption that $\SimpS$ is Kan and admits a \textbf{homotopy fixed point}. From Theorem ~\ref{t:goerss-formula} this implies the existence of an equivariant map
$$ f:\E \lrar \SimpS $$
for some Kan contractible simplicial $\Gam$-set $\E$. We can then take the cofiber $C_f$ and extend the action of $\Gam$ to it.

Since $\E$ is contractible the map $\SimpS \lrar C_f$ induces a homotopy equivalence of the underlying simplicial sets and so is a weak equivalence. This means that $f$ induces a weak equivalence $\SimpS^{h\Gam} \lrar C_f^{h\Gam}$. But $C_f$ has an actual fixed point and so we can use Goerss' theorem on it and obtain the desired spectral sequence.

Another aspect of homotopy fixed points for finite groups is that of \textbf{obstruction theory}. Let $G$ be a finite group acting on simplicial set $\SimpS$. We want to know whether there exists a homotopy fixed point.

Suppose for simplicity that $\SimpS$ is bounded Kan simplicial set and let $n$ be such that $\SimpS \simeq P_n(\SimpS)$. Then we can reduce the question of whether $\SimpS$ has an homotopy fixed point to the question of whether $P_n(\X)$ has a homotopy fixed point. We then consider the sequence of simplicial $G$-sets and equivariant maps
$$ P_n(\SimpS) \lrar P_{n-1}(\SimpS) \lrar ... \lrar P_0(\SimpS) $$
We can then break the non-emptiness question of $\SimpS^{hG}$ into a finite number of stages: for every $i=0,...,n$ we can ask whether $P_i(\SimpS)^{hG}$ is non-empty. Note that $P_i(\SimpS)^{hG}$ is non-empty if and only if $\pi_0\left(P_i(\SimpS)^{hG}\right)$ is non-empty, so we can work with sets instead of spaces.

For $i = 0$ we have $P_0(\SimpS) = \pi_0(\SimpS)$ and so
$$ \pi_0\left(P_0(\SimpS)^{hG}\right) = P_0(\SimpS)^{hG} = \pi_0(\SimpS)^{G} $$ is just the set of $G$-invariant connected components. Now given a $G$-invariant connected component $x_0 \in \pi_0\left(P_0(\SimpS)^{hG}\right)$ one can ask if it lifts to an element $x_1 \in \pi_0\left(P_1(\SimpS)^{hG}\right)$. This results in a short exact sequence
$$ 1 \lrar \pi_1(X,x_0) \lrar H \lrar G \lrar 1 $$
where $\pi_1(X,x_0)$ here denotes the fundamental group of the component of $X$ corresponding to $x_0$ (choosing different base points in the same connected component will lead to isomorphic short exact sequences). Obstruction theory then tells us that this sequence splits if and only if $x_0$ lifts to $\pi_0\left(P_1(\SimpS)^{hG}\right)$.

Now suppose we have an element $x_1 \in \pi_0\left(P_1(\SimpS)^{hG}\right)$ and let $p_1 \in x_1$ be a point. Note that $p_1$ is a homotopy fixed point, and not an actual point. However, since it encodes a map from a contractible space to $P_1(X)$ we can still use it as if it were a base point for purposes of homotopy groups, i.e. we can write
$$ \pi_n(X,p_1) $$
Further more $p_1$ behaves like a $G$-invariant base point and so we can use it to defines an action of $\Gam$ on each $\pi_n(X,p_1)$.

Obstruction theory then proceeds as follows: let $x_{i-1} \in \pi_0\left(P_{i-1}(\SimpS)^{hG}\right)$ be an element, let $x_1 \in \pi_0\left(P_1(\SimpS)^{hG}\right)$ be its image and let $p_1 \in x_1$ be a point. Then one obtains an obstruction element
$$ o_{x_{i-1}} \in H^{i+1}(G, \pi_i(\SimpS,p_1)) $$
which is trivial if and only if $x_{i-1}$ lifts to $\pi_0\left(P_i(\SimpS)^{hG}\right)$.


Now consider the case of a pro-finite group $\Gam$. Let $\SimpS$ be a bounded Kan simplicial $\Gam$-set. Then there is an $n$ such that the map $\SimpS \lrar P_n(\SimpS)$ is a weak equivalence. Hence as above we reduce the question of emptyness of $\X^{h\Gam}$ to that of $P_n(\X)^{h\Gam}$ which in turn leads to a sequence of lifting problems via the sequence of simplicial $\Gam$-sets
$$ P_n(\SimpS) \lrar P_{n-1}(\SimpS) \lrar ... \lrar P_0(\SimpS) $$

We now claim that the obstruction theory above generalizes to the pro-finite case by replacing the group cohomology
$$ H^{i+1}(G, \pi_i(\SimpS,p_1)) $$
with \textbf{Galois cohomology}
$$ H^{i+1}(\Gam,\pi_i(\SimpS,p_1)) $$
We will prove this here for the case $i \geq 2$. Note that in this case a homotopy fixed point $x \in P_{i-1}(\X)^{h\Gam}$ gives an element in $P_1(\X)^{h\Gam}$ which as above can act as a $\Gam$-invariant base point and determine an action of $\Gam$ on all the homotopy groups of $\X$. The case $i=1$ will be dealt with in subsection \S\S ~\ref{ss:pro-homotopy-quotient} (see Theorem ~\ref{t:homotopy-fixed-point-iff-splits}).

\begin{prop}\label{p:obstruction-theory}
Let $\SimpS$ be a bounded excellent strictly fibrant simplicial $\Gam$-set. Let $x_{i-1} \in \pi_0\left(P_{i-1}(\SimpS)^{h\Gam}\right)$ ($i \geq 2$) be a homotopy fixed point component, $x_1 \in \pi_0\left(P_1(\SimpS)^{h\Gam}\right)$ its image and $p_1 \in x_1$ a point. Then there exists an obstruction in the Galois cohomology group
$$ o(x_{i-1}) \in H^{i+1}(\Gam,\pi_i(\SimpS,p_1)) $$
which vanish if and only $x_{i-1}$ lifts to an element $x_i \in \pi_0\left(P_i(\SimpS)^{h\Gam}\right)$.
\end{prop}
\begin{proof}
Since $\SimpS$ is excellent there exists a $\Lam \fns \Gam$ such that the action of $\Gam$ on $\SimpS$ factors through $\Gam/\Lam$. Since $\SimpS$ is bounded $P_{i-1}(\SimpS)$ and $P_i(\SimpS)$ are strictly bounded and so we can apply formula ~\ref{t:goerss-formula} to get
$$ \pi_0\left(P_{i-1}(\SimpS)^{h\Gam}\right) =
\colim_{\Lam' \fns \Gam} \pi_0\left(\left(P_{i-1}(\SimpS)^{\Lam'}\right)^{h(\Gam/\Lam')}\right) = $$
$$ \colim_{\Lam' \subseteq \Lam, \Lam' \fns \Gam} \pi_0\left(P_{i-1}(\SimpS)^{h(\Gam/\Lam')}\right) $$
and similarly
$$ \pi_0\left(P_i(\SimpS)^{h\Gam}\right) = \colim_{\Lam' \subseteq \Lam, \Lam' \fns \Gam} \pi_0\left(P_i(\SimpS)^{h(\Gam/\Lam')}\right) $$
Let $\Lam_1 \subseteq \Lam$ be such that $x_{i-1}$ comes from $\pi_0\left(P_{i-1}(\SimpS)^{h(\Gam/\Lam_1)}\right)$. Then we want to know if there exists a $\Lam_2 \subseteq \Lam_1$ such that the image of $x_{i-1}$ in $\pi_0\left(P_{i-1}^{h(\Gam/\Lam_2)}\right)$ lifts to $\pi_0\left(P_{i}^{h(\Gam/\Lam_2)}\right)$. From the obstruction theory in the finite case we know that this is equivalent to the vanishing of a certain obstruction element
$$ o_{\Lam_2}(x_{i-1}) \in H^{i+1}(\Gam/\Lam_2,\pi_i(\SimpS,p_1)) $$
Hence we see that $x_{i-1}$ lifts to $\pi_0\left(P_{i}^{h\Gam}\right)$ if and only if $o_{\Lam_2} = 0$ for \textbf{some} $\Lam_2$.

Now each such $o_{\Lam_2}(x_{i-1})$ defines (the same) element
$$ o(x_{i-1}) \in H^{i+1}(\Gam,\pi_i(\SimpS,p_1)) =
\colim_{\Lam_2 \subseteq \Lam, \Lam_2 \fns \Gam}H^{i+1}(\Gam/\Lam_2,\pi_i(\SimpS,p_1)^{\Lam_2}) = $$
$$ \colim_{\Lam_2 \subseteq \Lam, \Lam_2 \fns \Gam}H^{i+1}(\Gam/\Lam_2,\pi_i(\SimpS,p_1)) $$
and $o(x_{i-1})$ vanishes if and only if $o_{\Lam_2}$ vanishes for some $\Lam_2$ so we are done.

\end{proof}

\begin{rem}
It is not hard to show that under the same assumptions for every $x_{i-1} \in P_{i-1}(\SimpS)(h\A)$ there exists an obstruction
$$ o(x_{i-1}) \in H^{i+1}(\A,\pi_i(\SimpS,p_1)) $$
which vanishes if and only if $x_{i-1}$ lifts to $P_i(\SimpS)(h\A)$ where  $H^{i+1}(\A,-)$ is suitable notion of restricted product of local cohomologies. We shall give an exact definition of a generalization of this notion in \S ~\ref{s:equivalence-1}.

\end{rem}

Note that  if
$$ f:\SimpS_1 \lrar \SimpS_2 $$
is a weak equivalence of simplicial $\Gam_K$-sets then the induced map
$$ f^{h\Gam_K}: \SimpS_1^{h\Gam_K} \lrar \SimpS_2^{h\Gam_K} $$
is a weak equivalence. We want a similar property to hold for adelic homotopy fixed points. This will require the additional assumption that the spaces are nice:

\begin{thm}\label{t:A-weak}
Let
$$ f:\SimpS_1 \lrar \SimpS_2 $$
be a weak equivalence of nice and bounded simplicial $\Gam_K$-sets. Then the induced map
$$ f^{h\A}:\SimpS_1^{h\A} \lrar \SimpS_2^{h\A} $$
is a weak equivalence.
\end{thm}
\begin{proof}

We start with two lemmas which give a connection between the connectivity of $f$ and the corresponding connectivity of $f^{h\Gam_K}$ and $f^{h\A}$ .

\begin{lem}\label{l:ss-fin-cd}
Assume that $\Gam_K$ is a group of finite strict (non-strict) cohomological dimension $d$ and let
$$ f:\SimpS_1 \lrar \SimpS_2 $$
be an $n$-connected map of (finite) nice bounded simplicial $\Gam_K$-sets. Then the induced map
$$ f^{h\Gam_K}: \SimpS_1^{h\Gam_K} \lrar \SimpS_2^{h\Gam_K} $$
is $(n-d)$-connected.
\end{lem}
\begin{proof}
Consider the corresponding spectral sequences $E^r_{s,t},F^r_{s,t}$. From our assumption the map
$$ E^2_{t,s} \lrar F^2_{t,s} $$
is an isomorphism for $s \leq n$. Since the differential $d^r_{s,t}$ goes from $(s,t)$ to $(s+r-1,t+r)$ we see that if $s-t \leq n-d$ then the map
$$ E^r_{t,s} \lrar F^r_{t,s} $$
remains an isomorphism for all $r$.
\end{proof}

\begin{lem}\label{l:ss-fin-cd-loc}
Let $K$ be number field $S$ a set of places of $K$ that does not contain the real places. Let
$$ f:\SimpS_1 \lrar \SimpS_2 $$
be an $n$-connected map of nice bounded simplicial $\Gam_K$-sets. Then the induced map
$$ f:\SimpS_1^{h\A_S} \lrar \SimpS_2^{h\A_S} $$
is $(n-3)$-connected.
\end{lem}

\begin{proof}
Let $T_0$ be the set of places $\nu$ in $S$ such that $I_\nu$ does not stabilize the $(n+1)$-skeleton of either $\SimpS_1$ or $\SimpS_2$. From our assumption the set $T_0$ is finite. Then for every finite set of places $T_0 \subseteq T \subseteq S$ we have that the maps
$$ f_\nu:\SimpS_1 \lrar \SimpS_2 , \quad \nu \in T $$
and
$$ f_\nu:\SimpS_1^{I_\nu} \lrar \SimpS_2^{I_\nu}, \quad \nu \notin T $$
are $n$-connected. Thus by Lemma ~\ref{l:ss-fin-cd} we get that the maps

$$ f_\nu:\SimpS_1^{h\Gam_\nu} \lrar \SimpS_2^{h\Gam_\nu} , \quad \nu \in T $$
and
$$ f_\nu:(\SimpS_1^{I_\nu})^{h\Gam^{ur}_\nu} \lrar (\SimpS_2^{I_\nu})^{h\Gam^{ur}_\nu}, \quad \nu \notin T $$
are $(n-3)$-connected. Therefore the map

$$ \prod \limits_\nu f_\nu :  \hocolim_{T_0\subseteq T\subseteq S}
\prod \limits_{\nu\in T} \SimpS_1^{h\Gam_\nu}\times \prod \limits_{\nu \in S\backslash T} \SimpS_1^{h^{ur}\Gam_v} \lrar $$ $$ \hocolim_{T_0\subseteq T\subseteq S} \prod \limits_{\nu\in T} \SimpS_2^{h\Gam_v}\times \prod \limits_{\nu \in S\backslash T} \SimpS_2^{h^{ur}\Gam_\nu}$$

is $(n-3)$-connected.
\end{proof}

We now complete the proof of the theorem. Denote by $S_\infty$ the finite set of archimedean places in $S$ and let
$$ S_f = S \backslash S_\infty $$
Then
$$ f:\SimpS_1^{h\A_{S_\infty}} \lrar \SimpS_2^{h\A_{S_\infty}} $$
is a weak equivalence and by Lemma ~\ref{l:ss-fin-cd-loc}
$$ f:\SimpS_1^{h\A_{S_f}} \lrar \SimpS_2^{h\A_{S_f}} $$
is a weak equivalence as well.
\end{proof}

\begin{cor}
In Lemma ~\ref{l:rational-is-h-adelic} we may replace the assumption that $\SimpS$ is strictly bounded with the assumption that $\SimpS$ is bounded, i.e
if $K$ is number field and $\SimpS$ a bounded $\Gam_K$-simplicial set then the natural map
$$ \loc: \SimpS\left(hK\right)  \lrar \prod \limits_{\nu}\SimpS\left(hK_\nu\right) $$
factors through a natural map
$$ f_0: \SimpS\left(hK\right) \lrar \SimpS\left(h\A\right) $$.
\end{cor}
\begin{proof}
In light of Theorem  ~\ref{t:A-weak} and Lemma ~\ref{l:rational-is-h-adelic} it enough to show that every bounded $\Gam_K$-simplicial set is weakly equivalent strictly to a bounded $\Gam_K$-simplicial set. Indeed, if $\SimpS$ is bounded then for large enough $n$ the map
$$\SimpS \to P_n(\SimpS)$$
is a weak equivalence, and $P_n(\SimpS)$  is always strictly bounded.
\end{proof}

\begin{lem}\label{l:preserves fibrations}
Let $\Gam$ be a pro-finite group and let
$$\SimpS_1 \lrar \SimpS_2 \lrar \SimpS_3 $$ be a fibration sequence of simplicial $\Gam$-sets. Then
$$\SimpS_1^{h\Gam} \lrar \SimpS_2^{h\Gam} \lrar \SimpS_3^{h\Gam} $$
is a fibration sequence of topological spaces.
\end{lem}

\begin{proof}
First note that we can change the map $\SimpS_2 \lrar \SimpS_3$ to a fibration $\widetilde{\SimpS}_2 \lrar \SimpS_3$ in the standard model structure without changing the homotopy types. Now any fibration in the standard model structure is also a Kan fibration and thus the fibre of the map $\widetilde{\SimpS}_2 \lrar \SimpS_3$ is standardly equivalent to $\SimpS_1$.
To conclude we may assume that  $$\SimpS_1 \lrar \SimpS_2 \lrar \SimpS_3 $$ be is a fibration sequence of simplicial $\Gam$-sets
in the standard model structure and the lemma follows from the fact that $\SimpS^{h\Gam}$ is the derived mapping space from the terminal object to $\SimpS$.
\end{proof}

\begin{cor}\label{c:preserves fibrations}
Let $K$ be a number field and let
$$\SimpS_1 \lrar \SimpS_2 \lrar \SimpS_3 $$
be a fibration sequence of nice simplicial $\Gam_K$-set. Then
$$\SimpS_1^{h\A} \lrar \SimpS_2^{h\A} \lrar \SimpS_3^{h\A}$$
is a fibration sequence of topological spaces.
\end{cor}

\begin{proof}
By applying $P_n$ for large enough $n$ it is enough to prove this when the $\SimpS_i$ are excellent. Now take $S$ to be a finite set of places such that all the $\SimpS_i$'s are unramified outside $S$. Now if $T$ is any finite set of place such that $S \subseteq T$ then by Lemma ~\ref{l:preserves fibrations}
$$
\prod \limits_{\nu\in T} \SimpS_1^{h\Gam_\nu} \times \prod \limits_{\nu \notin T} \SimpS_1^{h^{ur}\Gam_\nu} \lrar
\prod \limits_{\nu\in T} \SimpS_2^{h\Gam_\nu} \times \prod \limits_{\nu \notin T} \SimpS_2^{h^{ur}\Gam_\nu} \lrar
\prod \limits_{\nu\in T} \SimpS_3^{h\Gam_\nu} \times \prod \limits_{\nu \notin T} \SimpS_3^{h^{ur}\Gam_\nu}$$
is a fibration sequence. Now by passing to the limit and using the fact that direct homotopy colimits preserve fibration sequences we get that
$$ \SimpS_1^{h\A} \lrar \SimpS_2^{h\A} \lrar \SimpS_3^{h\A} $$
is a fibration sequence.
\end{proof}

\begin{define}
We shall say that a commutative diagram
$$
\xymatrix{
A\ar[r]\ar[d]&B\ar[d]\\
C\ar[r]&D
}
$$
in the category of sets is \textbf{semi-Cartesian}, if the map
$$ A \lrar B\times_D C$$
is onto.
\end{define}

\begin{prop}\label{p:2-is-enough}
Let $K$ be a number field and let $\SimpS$ be an excellent bounded simplicial $\Gam_K$-set. Then the commutative diagram
$$
\xymatrix{
\SimpS(hK) \ar[d] \ar[r] & P_2(\SimpS)(hK)\ar[d] \\
\SimpS(h\A) \ar[r]   & P_2(\SimpS)(h\A)\\
}
$$
is semi-Cartesian.

\end{prop}
\begin{proof}
We shall first prove the proposition for the case that $\SimpS$ is $2$-connected. In that case $P_2(\SimpS)$ is contractible and thus the claim is reduced to the following lemma

\begin{lem}\label{l:cosk2,2-conncted}
Let $K$ be number field and $\SimpS$ be a $2$-connected excellent bounded simplicial $\Gam_K$-set. Then the map
$\loc:\SimpS(hK) \lrar  \SimpS(h\A)$ is surjective
\end{lem}
\begin{proof}
First we will show that if $\SimpS(hK) = \emptyset$ then $\SimpS(h\A) = \emptyset$ as well. Since $\SimpS$ is excellent and bounded we can use the obstruction theory described above (see ~\ref{p:obstruction-theory}).

Since $\SimpS$ is $2$-connected the obstructions fall in the groups
$$ H^{i+1}(K, \pi_i(\SimpS)) $$
$$ H^{i+1}(\A, \pi_i(\SimpS)) $$
for $i \geq 3$ (note that since $P_1(\X)$ is contractible we can suppress the base point). Since the map
$$ H^{i+1}(K, A) \lrar H^{i+1}(\A, A) $$
is an isomorphism for $i \geq 2$ and every finite module $A$ (see ~\cite{Mil06}, Theorem $4.10$ (c)) we get that if $\SimpS(hK)$ is empty then so is $\SimpS(h\A)$.

Now assume that $\SimpS(hK) \neq \emptyset$. We shall prove that the map
$$\loc:\SimpS(hK) \lrar  \prod \SimpS(h K_v)$$
(and thus the map $\loc:\SimpS(hK) \lrar \SimpS(h\A)$) is surjective.
Since $\SimpS(hK) \neq \emptyset$ we have  $\prod \SimpS(h K_v) \neq \emptyset$
and so we can use the spectral sequence from Theorem ~\ref{t:ss}. Let $p \in \SimpS^{h\Gam_K}$ be a chosen base homotopy fixed point.

The lemma will follow by carefully investigating these spectral sequences. We shall denote by $E^r_{s,t}(K)$ the spectral sequence that converges to $\pi_{s-t}\left(\SimpS^{h\Gam_K}, p\right)$ and by $E^r_{s,t}(K_\nu)$ the spectral sequence that converges to $\pi_{s-t}\left(\SimpS^{h\Gam_\nu}, p_\nu\right)$.

We also denote by $\prod \limits_\nu E^r_{s,t}(K_\nu)$ the product spectral sequence (since $\SimpS$ is bounded the spectral sequence collapses after a finite number of pages. This fact together with the exactness of products insures that indeed  $\prod \limits_\nu E^r_{s,t}(K_\nu)$ is a spectral sequence and that its indeed converges to the product $\prod \limits_{\nu} \pi_{s-t}\left(\SimpS^{h\Gam_K}, p_\nu\right)$). Consider the map of spectral sequences
$$ \loc^r_{s,t}: E^r_{s,t}(K) \lrar \prod E^r_{s,t}(K_\nu)$$
This map converges to
$$ \loc_{s-t}: \pi_{s-t}\left(\SimpS^{h\Gam_K}, p\right) \lrar \prod \limits_{\nu} \pi_{s-t}\left(\SimpS^{h\Gam_K}, p_\nu\right) $$

Now since for a finite module $M$ and $k \geq 3$ one has
$$ H^k(K,M)\cong \prod \limits_{\nu} H^k(K_\nu,M) $$
we see that $\loc^2_{s,t}$ is an isomorphism for $t \geq 3$. Since $\SimpS$ is $2$-connected $\loc^2_{s,t}$ is also an isomorphism for $0 \leq s \leq 2$. Hence in particular it is an isomorphism on the diagonals $s-t = 0$ and $s-t= -1$.

We shall use the following lemma
\begin{lem}
If $f_{s,t}^r:E^r_{s,t} \lrar F^r_{s,t} $ is map of HL-spectral sequences such that $f_{s,t}^2$ is injective on the diagonal $s-t = d$ and surjective on the diagonal $s-t = d+1$ then the same is true for all $f^r_{s,t}$.
\end{lem}
\begin{proof}
By induction on $r$ and a simple diagram chase.
\end{proof}

Now since $\SimpS$ is bounded the spectral sequences $E^r_{s,t}(K), \prod E^r_{s,t}(K_\nu)$ collapse in some page $r$ and we get that the map
$$ \loc^{\infty}_{t,t}: E^\infty_{t,t}(K) \lrar \prod E^\infty_{t,t}(K_\nu) $$
is surjective for all $t \geq 0$. Hence
$$ \loc_{0}: \pi_{0}\left(\SimpS^{h\Gam_K}, p\right) \lrar \prod \limits_{\nu} \pi_{0}\left(\SimpS^{h\Gam_K}, p_\nu\right) $$
is also surjective.
\end{proof}

We shall now prove the claim in the case for a general $\SimpS$. Let
$$ ((a_\nu),x_0)\in \SimpS(h\A) \times_{P_2(\SimpS)(h\A)} P_2(\SimpS)(hK) $$
be a general point. Let $\Lam \fns \Gam_K$ be such that $x_0$ can be represented by a $\Gam_K$ equivariant map $\E(\Gam_K /\Lam) \lrar P_2(\SimpS)$. We get a diagram of excellent strictly fibrant simplicial $\Gam_K$-sets:
$$
\xymatrix{
\empty & \E(\Gam_K/\Lam) \ar[d]\\
\SimpS \ar[r] & P_2(\SimpS)
}
$$
We denote the homotopy pullback of this diagram by $\SimpS\langle 2 \rangle$. Note that $\SimpS\langle 2 \rangle$ is a $2$-connected and excellent  and that the sequence:
$$ \SimpS\langle 2 \rangle \lrar \SimpS \lrar P_2(\SimpS) $$
is a fibration sequence of excellent strictly fibrant simplicial $\Gam_K$-sets:. Hence by applying Lemma ~\ref{l:preserves fibrations} and Corollary ~\ref{c:preserves fibrations} we get the following commutative diagram with exact rows
$$
\xymatrix{
\SimpS\langle 2 \rangle(hK) \ar[d]^{\loc_1} \ar[r]^-{p} & \SimpS(hK) \ar[d]^{\loc_2} \ar[r]^-{c_2} & (P_2(\SimpS)(hK),x_0) \ar[d]^{\loc_3} \\
\SimpS\langle 2 \rangle(h\A)  \ar[r]^-{p} & \SimpS(h\A) \ar[r]^-{c_2} & (P_2(\SimpS)(h\A),\loc_3(x_0)) \\
}
$$
where the notation $(A,a)$ means that the element $a\in A$ is the neutral element in the pointed set $A$. Now since
$$ c_2((a_\nu)) = \loc_3(x_0) $$
there is an element $(b_\nu) \in \SimpS\langle 2 \rangle(h\A)$ such that $p((b_\nu)) = (a_\nu)$. Now by Lemma ~\ref{l:cosk2,2-conncted} there is an element $k_0\in \SimpS\langle 2 \rangle(hK)$ such that
$\loc_1(k_0)= (b_\nu)$. We denote $k_1 = p(k_0)$ and get
$$ c_2(k_1) = c_2(p(k_0)) = x_0 $$
and
$$ \loc_2(k_1) =  \loc_2(p(k_0)) = p(\loc_1(k_0)) = p((b_\nu)) = (a_\nu) $$
\end{proof}

\begin{cor}\label{c:cosk2-final}
Let $K$ be a number field, $X/K$ a smooth variety and $\mcal{U} \lrar X$ an hypercovering. Then for every $n \geq 2$ we have
$$ X(\A)^{\mcal{U},n} = X(\A)^{\mcal{U},2} $$
$$ X(\A)^{\mcal{\ZZ U},n} = X(\A)^{\ZZ\mcal{U},2} $$
\end{cor}

\begin{cor}\label{c:cosk2-zzh-h}
Let $K$ be a number field, $X/K$ a smooth variety and $\mcal{U} \lrar X$ a hypercovering such that $\SimpS_\mcal{U}$ is simply connected. Then
$$ X(\A)^{\mcal{U},n} = X(\A)^{\ZZ\mcal{U},n} $$
for every  $n \geq 0$.
\end{cor}

\section{ An Auxiliary Theorem }\label{s:intersection}
In this section we will show prove a theorem that will be very helpful in analyzing our various obstructions.

\begin{thm}\label{t:descent-theorem}
Let $K$ be a number field, $\SimpS_I = \{\SimpS_\alp\}_{\alp \in I} \in \Pro\Ho\left(\Set^{\Del^{op}}_{\Gam_K}\right)$ such that each $\SimpS_\alp$ is finite, bounded and excellent. Let $(x_\nu) \in \SimpS_I(h\A)$ be an adelic homotopy fixed point. Then $(x_\nu)$ is rational if and only if its image in each $\SimpS_\alp(h\A)$ is rational.

\end{thm}
\begin{proof}

Since an inverse system of non-empty finite sets has a non-empty inverse limit (see Lemma ~\ref{l:lim-fin}) it is enough to show that
\begin{prop}\label{p:finite-preim}
Let $K$ be a number field and let $\SimpS$ be an excellent finite bounded simplicial $\Gam_K$-set. Then the map
$$ {\loc_{\SimpS}}:\SimpS(hK) \lrar \SimpS(h\A) $$ 
has finite pre-images. i.e. for every $(x_\nu) \in \SimpS(h\A)$ the set ${\loc_{\SimpS}}^{-1}((x_\nu))$ is finite.
\end{prop}

\begin{proof}
Since
$$ \SimpS(h\A) \subseteq \prod \limits_{\nu} \SimpS(K_\nu^h) $$
we can work with this product instead of $\SimpS(h\A)$.

First note that the theorem is trivial if either of the sets is empty. If both of them are non-empty we can use the spectral sequence of theorem ~\ref{t:ss} in order to compute them. Let $p \in \SimpS^{h\Gam_K}$ be a base homotopy fixed point.

The proposition will follow by carefully investigating these spectral sequences. We shall denote by $E^r_{s,t}(K)$ the spectral sequence that converges to $\pi_{s-t}\left(\SimpS^{h\Gam_K}, p\right)$ and by $E^r_{s,t}(K_\nu)$ the spectral sequence that converges to $\pi_{s-t}\left(\SimpS^{h\Gam_\nu}, p_\nu\right)$. Consider the map of spectral sequences
$$ \loc^r: E^r_{s,t}(K) \lrar \prod E^r_{s,t}(K_\nu)$$
This map converges to
$$ \loc_{s-t}: \pi_{s-t}\left(\SimpS^{h\Gam_K}, p\right) \lrar \prod \limits_{\nu} \pi_{s-t}\left(\SimpS^{h\Gam_K}, p_\nu\right) $$

We are interested in the components which contribute to $\pi_0$ so we would like to understand the pre-images of the maps
$$ \loc_t^\infty: E^\infty_{t,t}(K) \lrar \prod E^\infty_{t,t}(K_\nu)$$
Since $\SimpS$ is bounded these groups/pointed sets are trivial for large enough $t$. For the rest of the $t$'s we will prove the following:
\begin{prop}\label{p:finite pre-image}
The maps
$$ \loc_t^\infty :E^\infty_{t,t}(K) \lrar \prod \limits_\nu E^\infty_{t,t}(K_\nu) $$
have finite pre-images for all $t \geq 0$. 
\end{prop}

Before we begin the proof let us explain how this proves that the pre-image of $(x_\nu)$ is finite.

Note the $E^{\infty}_{t,t}$ terms are pointed sets which filter the set $\pi_0\left(\SimpS^{h\Gam_K}\right)$ in a way which we describe below. The idea is to use this filtration on the pre-image ${\loc_{\SimpS}}^{-1}((x_\nu))$ of some $((x_\nu)) \in \SimpS(h\A)$. We will assume that ${\loc_{\SimpS}}^{-1}((x_\nu))$ in infinite and get a contradiction using this filtration.

Recall that in order to construct the spectral sequence we had to choose some homotopy fixed point $p \in \SimpS^{h\Gam_K}$, which we call the \textbf{base point} of the spectral sequence. We pick it so that its connected component $x \in \pi_0\left(\SimpS^{h\Gam_K}\right)$ is in ${\loc_{\SimpS}}^{-1}((x_\nu))$.

The first filtration map is the map
$$ \pi_0\left(\SimpS^{h\Gam_K}\right) \lrar E^\infty_{0,0} \subseteq H^0(\Gam_K,\pi_0(\SimpS)) $$
which associates to every homotopy fixed point the invariant connected component of $\SimpS$ which it lies in. Note that all of ${\loc_{\SimpS}}^{-1}((x_\nu))$ is mapped to the connected component of $\SimpS$ which $((x_\nu))$ maps to, and so this filtration step is trivial when restricted to ${\loc_{\SimpS}}^{-1}((x_\nu))$.

Now for those homotopy fixed points which are mapped to the same connected component as $x$ we get the next filtration map
$$ f_1:f_0^{-1}(f_0(x)) \lrar E^{\infty}_{1,1} \subseteq H^1(\Gam,\pi_1(\SimpS)) $$
Now if we restrict this filtration map to ${\loc_{\SimpS}}^{-1}((x_\nu))$ we get that their image under $f_1$ lies in the appropriate pre-image of the map
$$ \loc_1^\infty: E^\infty_{1,1}(K) \lrar \prod E^\infty_{1,1}(K_\nu) $$
From Proposition ~\ref{p:finite pre-image} such a pre-image must be finite. Hence if ${\loc_{\SimpS}}^{-1}((x_\nu))$ is infinite then there exists a fiber of $f_1$ which has an infinite intersection with ${\loc_{\SimpS}}^{-1}((x_\nu))$. Let $F \subseteq {\loc_{\SimpS}}^{-1}((x_\nu))$ be this fiber. Then we can assume without loss of generality that $x \in F$.

Now since all the elements in $F$ agree on the first two filtration step it follows from the general construction of the spectral sequence that if we change the base point $p$ to any other base point in $F$ then the spectral sequence will be isomorphic. This makes the rest of the steps of the filtration to be independent of $p$.

Continuing on we have filtration maps
$$ f_t:f_{t-1}^{-1}(f_{t-1}(x)) \lrar E^{\infty}_{t,t} $$
For $t \geq 2$ these are abelian groups and elements of ${\loc_{\SimpS}}^{-1}((x_\nu))$ are mapped to the kernel of the map $\loc^{\infty}_t$, which is finite by our theorem.

Hence we can continue the process of choosing each time the infinite fiber and assume that $x$ is in that infinite fiber. Since $\SimpS$ is bounded there will be only a finite number of filtration steps and
since the spectral sequence will no longer change when we change $x$ within the infinite fiber we get from Proposition ~\ref{p:finite pre-image} the desired contradiction.

We now come to the proof of the proposition:
\begin{proof}
For $t = 0$ note that the map
$$ E^2_{0,0}(K) \lrar \prod E^2_{0,0}(K_\nu) $$
is injective and thus so is the map
$$ E^\infty_{0,0}(K) \lrar \prod E^\infty_{0,0}(K_\nu) $$
For $t > 2$ the set $E^2_{t,t}(K) = H^t(K,\pi_t(\SimpS))$ is finite and therefore the set $E^\infty_{t,t}(K)$ is finite.
For $t = 1$ we have that $E^{\infty}_{1,1} \subseteq E^{2}_{1,1}$ so it is enough to show that the map
$$ \loc^2:E^{2}_{1,1}(K) \lrar \prod E^2_{1,1}(K_\nu) $$
has finite pre-images. This is the map
$$ \loc_{\pi_1}: H^1(K,\pi_1(\SimpS)) \lrar \prod \limits_\nu H^1(K_\nu,\pi_1(\SimpS)) $$
The fact that this map has finite pre-images appears for example in Borel-Serre ~\cite{BSe64} \S 7.

We shall now prove for $t = 2$. We have
$$ E^{\infty}_{2,2} \subseteq E^{2}_{2,2}/d^2\left(E^2_{1,0}\right) $$
Consider first the map
$$ \loc^2_{2,2}:E^{2}_{2,2}(K) \lrar \prod E^2_{2,2}(K_\nu) $$
which is the map of abelian groups
$$ \loc_{\pi_2}: H^2(K,\pi_2(\SimpS)) \lrar \prod \limits_\nu H^2(K_\nu,\pi_2(\SimpS)) $$
and the kernel of this map is $\Sha^2(\pi_2(\SimpS))$
which is finite since $\pi_2(\SimpS)$ is finite (see Milne ~\cite{Mil06} Theorem 4.10).

Hence in order to show that $\loc^{\infty}_{2,2}$ has finite kernel it is enough to show that
$$ d^2:\prod E^2_{1,0}(K_\nu)\lrar \prod E^2_{2,2}(K_\nu) $$
has finite image. Now for each $\nu$ the group
$$ E^2_{1,0}(K_\nu)= H^0(\Gam_\nu,\pi_1(\SimpS)) $$
is finite because $\pi_1(\SimpS)$ is finite and so it is enough to show that for almost all $\nu$ the map
$$ d^2: E^2_{1,0}(K_\nu)\lrar  E^2_{2,2}(K_\nu) $$
is the zero map. This is done in the following lemma:

\begin{lem}\label{l:ss-unram}
Let $\SimpS$ be a $\Gam_K$-simplicial set such that the $3$-skeleton $\SimpS_3$ is stabilized by some open subgroup $\Gam_L \subseteq \Gam_K$ where $L/K$ is a finite Galois extension. Then if $\nu$ is a place of $K$ which is non-ramified in $L$ then the differential
$$ d^2: E^2_{1,0}(K_\nu) \lrar E^2_{2,2}(K_\nu) $$
is zero.
\end{lem}

\begin{proof}
The action of the group $\Gam_\nu$ on $\SimpS_3$ factors through the group $\Gam^{ur}_\nu$. Now consider the natural maps
$$ \SimpS_3^{h \Gam^{ur}_\nu} \x{f_1}{\lrar} \SimpS_3^{h\Gam_\nu} \x{f_2}{\lrar} \SimpS^{h\Gam_\nu}$$

We have corresponding maps of spectral sequences
$$ F^r_{s,t}(K^{ur}_\nu) \x{f^r_1}{\lrar} F^r_{s,t}(K_\nu) \x{f^r_2}{\lrar} E^r_{s,t}(K_\nu)$$

We denote by $f^2_3 = f^2_2 \circ f^2_1$. Now let
$$ a \in E^2_{1,0}(K_\nu) = H^0(\Gam_\nu,\pi_1(\SimpS)) $$
Since the map
$$ \pi_1(\SimpS_3) \lrar \pi_1(\SimpS) $$
is an isomorphism and the action on $\pi_1(\SimpS_3)$ factors through $\Gam_\nu$ we see that
$$ f^2_3: H^0(\Gam^{ur}_\nu,\pi_1(\SimpS_3)) \lrar H^0(\Gam_\nu,\pi_1(\SimpS)) $$
is an isomorphism. This means that there exists a
$$ b \in H^0(\Gam^{ur}_\nu,\pi_1(\SimpS_3)) = F^2_{1,0}(K^{ur}_\nu) $$
such that $f^2_3(b) = a$ .

Now since $\Gam^{ur}_\nu$ has cohomological dimension 1 and $\pi_2(\SimpS_3) \cong \pi_2(\SimpS)$ is finite we have
$$ F^2_{2,2}(K^{ur}_\nu) = H^2(\Gam^{ur}_\nu,\pi_2(\SimpS_3)) = 0 $$
and therefore $d^2(b) = 0$. Now
$$ d^2(a) = d^2(f^2_3(b)) = f^2_3(d^2(b)) = f^2_3(0) = 0 $$
\end{proof}
\end{proof}
\end{proof}

\end{proof}

\begin{cor}\label{c:descent-theorem}
Let $X$ be a smooth algebraic variety over a number field $K$. Then for every $0 \leq n \leq \infty$ one has
$$ X(\A)^{h,n} = \bigcap \limits_{\U,k \leq n} X(\A)^{\U,k} $$
$$ X(\A)^{\ZZ h,n} = \bigcap \limits_{\U,k \leq n} X(\A)^{\ZZ\U,k} $$

\end{cor}

Using Corollary  ~\ref{c:cosk2-final} we then get the following important conclusion:
\begin{cor}\label{c:cosk2-final-2}
Let $K$ be a number field and $X/K$ a smooth variety. Then for every $2 \leq n \leq \infty$ we have
$$ X(\A)^{h, n} = X(\A)^{h, 2} $$
$$ X(\A)^{\ZZ h, n} = X(\A)^{\ZZ h,2} $$
\end{cor}
In particular, the homotopy and homology obstructions \textbf{depend only on the $2$-truncation of the \'etale homotopy type}.

\section{ Sections and Homotopy Fixed Points}

Let $G$ be a finite group acting on a space $\X$. Then it is known that the space of homotopy fixed points $\X^{hG} = [\E G,X]_{G}$ is naturally equivalent to the space of sections of the
classifying map
$$ p:\X_{hG}\lrar \B G $$
where
$$ \X_{hG} = (X\times \E G)/ G$$
is the homotopy quotient. In particular the question of existence of a homotopy fixed point can be translated to the question of existence of a section to $p$.

In this section we will discuss the generalization of this alternative approach to the case of pro-spaces and pro-homotopy types. A similar approach is taken by P\`{a}l in his paper ~\cite{Pal10}.

\subsection{ The Pro Fundamental Group }\label{ss:pro-fundamenral-group}

In order to study the notion of fundamental groups one needs to be able to work with base points. Since the \'etale homotopy type is not naturally a pro object in the homotopy category of pointed spaces one needs to make some choices in order to identify base points.

One way to tackle this issue is to lift pro-homotopy-types to pro-spaces. We will use the following observation:
\begin{lem}
\begin{enumerate}
\item
Let $\X_I = \{\X_\alp\}_{\alp \in I} \in \Pro\Ho\left(\Set^{\Del^{op}}_{\Gam_K}\right)$ be an object such that $I$ is \textbf{countable}. Then there exists a object $\wtl{\X}_{I'} \in \Pro\Set^{\Del^{op}s}_{\Gam_K}$ whose image in $\Pro\Ho\left(\Set^{\Del^{op}}_{\Gam_K}\right)$ is isomorphic to $\X_I$. Further more one can choose $I'$ to be the poset $\NN$ of natural numbers such that $\wtl{\X}_0$ is fibrant and all the maps $\wtl{\X}_n \lrar \wtl{\X}_{n-1}$ are fibrations in the standard model structure. We will refer to such pro-objects as \textbf{fibration towers}.
\item
Let $f: \X_I \lrar \Y_J$ be a map in $\Pro\Ho\left(\Set^{\Del^{op}}_{\Gam_K}\right)$ and assume that both $I,J$ are countable. Then one can lift $f$ to a map
$$ \wtl{f}: \wtl{\X}_{I'} \lrar \wtl{\Y}_{J'} $$
in $\Pro\;\Set^{\Del^{op}}_{\Gam_K}$. Further more one can choose $I'$ and $J'$ to be $\NN$ and $\wtl{f}$ to be levelwise.
\end{enumerate}

\end{lem}
\begin{proof}
\begin{enumerate}
\item
Since $I$ is countable it contains the poset $\NN \subseteq I$ as cofinal subcategory. One then constructs the lifts $\wtl{\X}_n$ by induction on $n$ by each time representing the homotopy class $\X_n \lrar \X_{n-1}$ by a fibration.
\item
By choosing $\wtl{\X}_{I'}$ and $\wtl{\Y}$ to be towers of fibrations one can lift $f$ to $\wtl{f}$ using standard lifting properties of fibrations. One can then replace $\wtl{\X}_{I'}$ with a sub-tower in order to make $f$ into a levelwise map.
\end{enumerate}
\end{proof}

\begin{rem}
The lift described above is unique up to homotopy but is \textbf{not functorial}. There is a way to lift the \'etale homotopy type functorially to a pro-space using Friedlander's construction of the \'etale topological type (~\cite{Fri82}). However we will not make use of this construction in this paper.
\end{rem}
\begin{rem}\label{r:etale-is-ok}
Given a $K$-variety $X$, the category $I(X)$ is in general not countable. However, if $K$ is a countable field (e.g. a number field), and we consider the object $\acute{E}t^{\natural}_{/K}(X)$ in which all the simplicial sets are truncated, we see that it is isomorphic to its sub-diagram indexed only by \textbf{truncated} hypercoverings (i.e. hypercoverings $\U_\bullet$ which satisfy $\U_\bullet \cong \cosk_n(\tr_n(U_\bullet))$ for some $n$). The subcategory of truncated hypercoverings is equivalent to a countable category and so we will be able to apply the technics developed here to this case.
\end{rem}

\begin{define}
Let $\{\X_\alp\} \in \Pro\;\Set^{\Del^{op}}$ be an object and $\{x_\alp\} \in \lim_\alp \X_\alp$ a base point. We will define the $n$'th pro-homotopy group to be the pro-group
$$ \pi_n(\X_\NN,\{x_\alp\}) \x{def}{=} \{\pi_n(\X_\alp,x_\alp)\} $$

\end{define}

\begin{rem}
Note that if $\{\X_\alp\}$ is a pro-simplicial $\Gam$-set we will define it pro-homotopy groups by forgetting the group action. In particular we will not require base points to be $\Gam$-invariant.
\end{rem}

We would like to restrict our selves to the following nice class of objects:
\begin{define}
An object $\X_I \in \Pro\;\Set_{\Gam}^{\Del^{op}}$ will be called \textbf{pro-finite} if each $\X_\alp$ is an excellent, truncated and finite simplicial $\Gam$-set.
\end{define}
The issue of uniqueness of the base-point is dealt with in the following lemma:
\begin{lem}\label{l:pro-connected}
Let $\X_{\NN}$ be a tower of Kan fibrations such that each $\X_n$ is finite, non-empty and connected. Then $\lim_n \X_n$ is non-empty and connected.
\end{lem}
\begin{proof}
Since $\X_{\NN}$ is a tower of Kan fibrations the non-emptiness of each $\X_n$ implies the non-emptiness of $\lim_n\X_n$.

Now let $\{x_n\}, \{x_n'\} \in \lim_n\X_n$ be two points. For each $n$, let $P(\X_n,x_n,x_n')$ denote the space of paths from $x_n$ to $x_n'$ in $\X_n$. Then $\{P(\X_n,x_n,x_n')\}$ is a tower of Kan fibrations as well and so it is enough to show that
$$ \lim_n \pi_0(P(\X_n,x_n,x_n')) \neq \emptyset $$
But this now follows from the following standard lemma (since by our assumptions each $\pi_0(P(\X_n,x_n,x_n'))$ is non-empty and finite):
\begin{lem}\label{l:lim-fin}
Let $\{A_\alp\}_{\alp \in I}$ be an inverse system of finite non-empty sets. Then
$$ \lim \limits_{\x{\llar}{\alp}} A_\alp \neq \emptyset $$
\end{lem}
\begin{proof}
A standard compactness argument.

%
%
\end{proof}

\end{proof}

It will be useful then to keep in mind the following standard observation:
\begin{lem}\label{pro-finite-is-pro-finite}
The pro-category of finite groups is equivalent to the category of pro-finite groups (and continuous homomorphisms). The equivalence is given by $\{G_\alp\} \mapsto \lim_\alp G_\alp$ with the natural pro-finite topology.
\end{lem}
\begin{proof}
Standard.
\end{proof}

In particular we think of the pro-homotopy groups of pro-finite objects $\X_I \in \Pro\;\Set^{\Del^{op}}$ as pro-finite groups. From now on we will not distinguish between pro-finite groups and pro-objects in the category of finite groups.

\subsection{ The Pro Homotopy Quotient }\label{ss:pro-homotopy-quotient}
Recall that the categorical product of $\{\X_\alp\},\{\Y_\bet\}$ in the category $\Pro\;\Set^{\Del^{op}}_{\Gam_K}$ is given by
$$ \{\X_\alp\} \times \{\Y_\bet\} \x{def}{=} \{\X_\alp \times \Y_\bet\}_{\alp,\bet} $$

Let $\wtl{\E\Gam^{\natural}} \in \Pro\;\Set^{\Del^{op}}_{\Gam}$ be the lift of $\E\Gam^{\natural}$ given by
$$ \wtl{\E\Gam^{\natural}} = \{\E(\Gam/\Lam)\}_{\Lam \fns \Gam} $$

We will now define the pro-homotopy quotient. Let $\X_\NN \in \Pro\;\Set^{\Del^{op}}_{\Gam}$ be a pro-finite fibration tower. We define its \textbf{pro-homotopy quotient} to be the levelwise quotient (which is also the categorical quotient, see ~\ref{c:pro-quotient-categorical}) of $\X_{\NN}\times \wtl{\E\Gam^{\natural}}$ by $\Gam$, which we write as
$$ (\X_{\NN})_{h\Gam} = \left(\X_{\NN}\times \wtl{\E\Gam^{\natural}}\right)/\Gam = \{\X_{n,\Lam}\}_{n,\Lam \fns \Gam} $$
where
$$ \X_{n,\Lam} \x{def}{=} (\X_n \times \E(\Gam/\Lam))/\Gam $$

Note that whenever $\Lam$ fixes $\X_n$ we get that $\X_{n,\Lam}$ is truncated and finite. Since the pairs $(n,\Lam)$ for which $\Lam$ fixes $\X_n$ are a cofinal subfamily we get that $\left(\X_{\NN}\right)_{h\Gam}$ is isomorphic to a pro-finite object. In particular we can consider its pro-fundamental group as a pro-finite group.

This pro-simplicial $\Gam$-set fits into a short sequence
$$ \X_{\NN} \lrar (\X_{\NN})_{h\Gam} \lrar \{\B(\Lam/\Gam)\}_{\Lam \fns \Gam} $$
Choosing a base point $\{x_n\} \in \lim_n \X_n$ we get an induced base point in $(\X_{\NN})_{h\Gam}$ (which we will denote $\{x_n\}$ as well) which is mapped to the unique base point of $\{\B(\Lam/\Gam)\}$ yielding a short sequence of pro-finite groups:
$$ 1 \lrar \pi_1\left(\X_\NN,\{x_n\}\right) \lrar \pi_1\left((\X_\NN)_{h\Gam},\{x_n\}\right) \lrar \Gam \lrar 1 $$

which is exact because whenever $\Lam$ fixes $\X_n$ the short sequence of groups
$$ 1 \lrar \pi_1(\X_n,x_n) \lrar \pi_1(\X_{n,\Lam},x_n) \lrar \Gam/\Lam \lrar 1 $$
is exact.

Our main claim of this section is the following:
\begin{thm}\label{t:homotopy-fixed-point-iff-splits}
Let $\X_\NN$ a pro-finite fibration tower and
$$ \X^1_\NN = \left\{P_1\left(\X_n\right)\right\} $$
its $1$-truncation. Then the sequence
$$ 1 \lrar \pi_1\left(\X_\NN,\{x_n\}\right) \lrar \pi_1\left(\left(\X_\NN\right)_{h\Gam},\{x_n\}\right) \lrar \Gam \lrar 1 $$
%
splits if and only if $\X^1_\NN(E\Gam^{\natural}) \neq \emptyset$.
\end{thm}
\begin{proof}

Since each $\X^1_n$ is strictly bounded the non-emptyness of $\X^1_\NN(E\Gam^{\natural})$ is equivalent to the existence of a map
$$ \wtl{\E\Gam^{\natural}} \lrar \X^1_{\NN}$$
in $\Pro\;\Set^{\Del^{op}}_{\Gam}$ (we can lift maps into $\X^1_{\NN}$ from $\Pro\Ho\left(\Set^{\Del^{op}}_{\Gam}\right)$ to $\Pro\;\Set^{\Del^{op}}_{\Gam}$ because $\X^1_{\NN}$ is a tower of fibrations).

Now given a map
$$ h:\wtl{\E\Gam^{\natural}} \lrar \X^1_{\NN} $$
we can take the map
$$ \wtl{h} \times Id: \wtl{\E\Gam^{\natural}} \lrar \X^1_{\NN} \times \wtl{\E\Gam^{\natural}} $$
and descend it to the $\Gam$-quotients
$$ s:\{\B(\Lam/\Gam)\} = \wtl{\E\Gam^{\natural}}/\Gam \lrar (\X^1_{\NN})_{h\Gam} $$
obtaining a section of the natural map
$$ (\X^1_{\NN})_{h\Gam} \lrar \{\B(\Lam/\Gam)\} $$

\begin{lem}\label{l:lift_B_to_E}
Every section
$$ s:\{\B(\Lam/\Gam)\} \lrar (\X^1_{\NN})_{h\Gam} $$
is induced by a map
$$ h:\wtl{\E\Gam^{\natural}} \lrar \X^1_{\NN} $$
in this way.
\end{lem}
\begin{proof}
The map $s$ can be described by the following information: for each $n$ and $\Lam$ that stabilizes $\X_n$ we are given a normal subgroup $\Lam' \fns \Gam$ which is contained in $\Lam$ and a map
$$ s_{n,\Lam}:\B (\Gam/\Lam') \lrar \X^1_{n,\Lam} = (\X^1_n \times \E(\Lam/\Gam))/\Gam $$
such that the composition
$$ \B (\Gam/\Lam') \lrar \X^1_{n,\Lam} \lrar \B (\Gam/\Lam) $$
is equal to the map $q_*:\B (\Gam/\Lam') \lrar \B (\Gam/\Lam)$ induced by the natural projection $q: \Gam/\Lam' \lrar \Gam/\Lam$. Since $\Lam$ stabilizes $\X^1_n$ the action $\Gam$ on $\X^1_n \times \E(\Gam/\Lam)$ factors through a free action of $\Lam/\Gam$. Hence we have a pullback square
$$ \xymatrix{
\X^1_n \times \E(\Gam/\Lam) \ar[d]\ar[r] & \E (\Gam/\Lam) \ar[d] \\
\X^1_{n,\Lam} \ar[r] & \B(\Gam/\Lam) \\
}$$
The maps $q_* : \E(\Gam/\Lam') \lrar \E(\Gam/\Lam)$ and the composition
$$ \E(\Gam/\Lam') \lrar \B(\Gam/\Lam')  \x{s_{i,j}}\lrar \X^1_{n,\Lam} $$
combine to form a map
$$ \wtl{s}_{n,\Lam}:\E(\Gam/\Lam') \lrar \X^1_{n,\Lam} $$ which lifts $s_{n,\Lam}$.

It is left to show that the maps $\wtl{s}_{n,\Lam}$ can be chosen in a compatible way. Since the map $\X^1_n \times E(\Gam/\Lam) \lrar \X^1_{n,\Lam}$ is a covering map with fiber of size $|\Gam/\Lam|$ we see that there are no more then $|\Gam/\Lam|$ equivariant maps $\E(\Gam/\Lam') \lrar \X^1_{n,\Lam}$ which lift $s_{n,\Lam}$. Since a filtered colimit of sets of size $\leq |\Gam/\Lam|$ is of size $\leq |\Lam/\Gam|$ the result will now follow from lemma ~\ref{l:lim-fin}.
\end{proof}

%

To summarize so far we see that the non-emptiness of $\X^1_\NN(E\Gam^{\natural})$ is equivalent to the existence of section
$$ s:\{\B(\Lam/\Gam)\} \lrar (\X^1_{\NN})_{h\Gam} $$
to the natural map
$$ (\X^1_{\NN})_{h\Gam} \lrar \{\B(\Lam/\Gam)\} $$

It is left to show that the existence of such a section is equivalent to a section in the respective fundamental groups. Since $\X^1_{n,\Lam}$ is $1$-truncated whenever $\Lam$ fixes $n$ we see that a section of the pro-fundamental groups induces a section
$$ s:\{\B(\Lam/\Gam)\} \lrar (\X^1_{\NN})_{h\Gam} $$

In the other direction let
$$ s:\{\B(\Lam/\Gam)\} \lrar (\X^1_{\NN})_{h\Gam} $$
be a section. Note that $s$ might send the base point of $\{B(\Lam/\Gam)\}$ to a point other then our chosen base point $\{x_n\}$. However from ~\ref{l:pro-connected} the simplicial set $\lim_n \X^1_n$ is connected and so we will know how to translate $s$ into a section in the fundamental groups.

\end{proof}

We finish by the following useful criterion:
\begin{lem}\label{l:section-iff-homotopy-section}
The map ${\X^1_\NN}_{h\Gam} \lrar \{B(\Gam/\Lam)\}$ has a section if and only if the induced map in $\Pro\Ho\left(\Set^{\Del^{op}}\right)$ has a section.
\end{lem}
\begin{proof}
We will use the following lemma whose proof is easy and classical:
\begin{lem}\label{l:ho_bgbh}
Let $H,G$ be two groups then
$$\Hom_{\Set^{\Del^{op}}}(\B H,\B G) \cong \Hom_{Grp}(H,G) $$
$$\Hom_{\Ho(\Set^{\Del^{op}})}(\B H,\B G) \cong \Hom_{Grp}(H,G) / \sim  $$
where for $p_1,p_2:H\to G$ we have $p_1 \sim  p_2$ if there exist $g\in G$ such that $p_1(\bullet)  = g p_2(\bullet)g^{-1}$
\end{lem}

Now let
$$ s:\{\B(\Lam/\Gam)\} \lrar (\X^1_{\NN})_{h\Gam} $$
be a section in $\Pro\Ho\left(\Set^{\Del^{op}}\right)$. As above the map $s$ can be described by the following information: for each $n$ and $\Lam$ that stabilizes $\X_n$ we are given a normal subgroup $\Lam' \fns \Gam$ which is contained in $\Lam$ and a map
$$ s_{n,\Lam}:\B (\Gam/\Lam') \lrar \X^1_{n,\Lam} = (\X^1_n \times \E(\Lam/\Gam))/\Gam $$
such that the composition
$$ \B (\Gam/\Lam') \lrar \X^1_{n,\Lam} \lrar \B (\Gam/\Lam) $$
is \textbf{homotopic} to the map $q_*:\B (\Gam/\Lam') \lrar \B (\Gam/\Lam)$ induced by the natural projection $q: \Gam/\Lam' \lrar \Gam/\Lam$.

By Lemma ~\ref{l:ho_bgbh} the two maps differ by a conjugation by some element of $\Gam/\Lam$. Lifting this element to $\Gam/\Lam'$ we can find a map $s'_{n,\Lam}$ homotopic to $s_{n,\Lam}$ such that the composition
$$ \B (\Gam/\Lam') \lrar \X_i/(\Gam/\Lam) \lrar \B (\Gam/\Lam) $$
is exactly the natural projection map $q_*:\B (\Gam/\Lam') \lrar \B (\Gam/\Lam) $. Now similarly to the proof of Lemma ~\ref{l:lift_B_to_E} there are only finitely many such possible maps so by Lemma ~\ref{l:lim-fin} we have a section in $\Pro\;\Set^{\Del^{op}}$ and we are done.
\end{proof}

\subsection{ The \'Etale Fundamental Group }
In this section we will connect the notions described in the previous section to the \'etale fundamental groups.
\begin{prop}\label{p:etale-is-quotient}
Let $K$ be a field and $X$ a $K$-variety. Then
$$ \left(\acute{E}t^{n}_{/K}(X)\right)_{h\Gam_K} \cong \acute{E}t^{n}(X) $$
in $\Pro\Ho\left(\Set^{\Del^{op}}_{\Gam_K}\right)$ for every $n \leq \infty$.
\end{prop}

\begin{proof}
We start by showing that
$$ \acute{E}t^{n}_{/K}(X) \times \E\Gam_K^{\natural} \cong \acute{E}t^{n}_{/K}(X) $$
for every $n \leq \infty$. Since $P_k$ commutes with products and the simplicial $\Gam_K$-sets in $\E\Gam_K^{\natural}$ are $0$-truncated it will be enough to prove that
$$ \acute{E}t_{/K}(X) \times \E\Gam_K^{\natural} \cong \acute{E}t_{/K}(X) $$
Now we have a natural projection map
$$ p: \acute{E}t_{/K}(X) \times \E\Gam_K^{\natural} \cong \acute{E}t_{/K}(X) $$
and we will show that it is in fact an isomorphism. Let $\U \lrar X$ be a hypercovering, $n$ a natural number and $\Lam \fns \Gam$ an open normal subgroup. Let $L/K$ be the finite Galois extension corresponding to it.
We will use the following construction:

\begin{define}\label{d:X-L}
Let $L/K$ be a finite Galois extension. Let $X_L$ be the restriction of scalars of $X \otimes_K L$ from $L$ to $K$. Note that there is a natural map
$$ X_L \lrar X $$
which is an \'{e}tale cover. We then denote by
$$ \check{X}_L\lrar X $$
the hypercovering obtained by the \v{C}ech construction (see \S ~\ref{ss:hypercoverings}, Definition ~\ref{d:cech}).
\end{define}

The connected components of $X_L \otimes_K \ovl{K}$ (each of which is isomorphic to $X$) can be identified (not uniquely) with $G_L = \Gal(L/K)$ where the Galois action of $G_L$ on itself is by left translations. Such an identification induces an isomorphism of simplicial $\Gam_K$-sets

$$ \SimpS_{\check{X}_L} \lrar \E G_L $$

In particular the action of $\Gam_K$ on $\SimpS_{\check{X}_L}$ factors through a \textbf{free} action of $G_L$.

Now let
$$ \U_L = \U \times_X \check{X}_L $$
Then it is easy to verify that the natural map
$$ \vphi_{\U,L}:\X_{\U_L} \lrar \X_{\U} \times \X_{\check{X}_L} = \X_{\U} \times \E G_L  $$
is an isomorphism of simplicial $\Gam_K$-set. The maps $\vphi_{\U,L}$ fit together to form a map
$$ \acute{E}t_{/K}(X) \lrar \acute{E}t_{/K}(X)  \times \E\Gam^{\natural} $$
which is an inverse to $p$. We leave it to the reader to verify that this is indeed an inverse to $p$. The proof is analogous to the case appearing in the proof of theorem ~\ref{t:MIT-1}.

In view of Corollary ~\ref{c:pro-quotient} (whose proof is independent of the rest of the paper) we can finish the proof by showing that
$$ \acute{E}t^{n}(X) \cong \acute{E}t^{n}_{/K}(X)/\Gam = \{\X_{\U,k}/\Gam\}_{\U \in I(K), k \leq n} $$
Now by definition we have that
$$ \acute{E}t(X) \cong \acute{E}t_{/K}(X)/\Gam $$
Hence in order to get our result we need to verify that taking the quotient commutes with truncation in this case. This will be done using the following two lemmas:

\begin{lem}\label{l:free-cofinal}
Let $K$ be a field and $X$ a $K$-variety. Then there is a cofinal subcategory $J(X) \subseteq I(X) \times \NN$ such that for each $(\U,n) \in J(X)$ the action of $\Gam_K$ on $\SimpS_{\U,n}$ factors through a \textbf{free} action of a finite quotient of $\Gam_K$.
\end{lem}

\begin{proof}
Let $\U \lrar X$ be a hypercovering and $n \in \NN$ a number. Let $L/K$ be a finite Galois extension such that the action of $\Gam_K$ on the $(n+1)$-skeleton of $\U$ factors through $G_L = \Gam_K/\Gam_L$.

Consider as above
$$ \U_L = \U \times_X \check{X}_L $$
Then the action of $\Gam_K$ on the $(n+1)$-skeleton of $\SimpS_{\U_L}$ factors through $G_L$. Since we have a map
$$ \SimpS_{\U_L} \lrar \SimpS_{\check{X}_L}  $$
we see that the action of $G_L$ on the $(n+1)$-skeleton of $\SimpS_{\U_L,n}$ is free. We also have a map of hypercoverings $\U_L \lrar U$. Hence the full subcategory on
$$ J(X) = \{(\U_L,n)\} \subseteq I(X) \times \NN $$
is cofinal and we are done.

\end{proof}
\begin{lem}\label{l:postinikov-quotient}
Let $\X$ be a free simplicial $G$-set and $n \geq 1$ be a number. Then $G$ acts freely on $P_n(\X)$ and
$$ P_n(\X/G) \cong P_n(\X)/G $$

\end{lem}
\begin{proof}
Let $G$ be finite group, $\X$ a simplicial $G$-set and $\Y$ a simplicial set considered as a simplicial $G$-set with trivial action. Let $f:\X \lrar \Y$ be a map of simplicial $G$-sets. We say that \textbf{$\Y$ is a free $G$-quotient by $f$} if $G$ acts freely on $X$ and $f$ induces an isomorphism os simplicial sets
$$ \wtl{f}: \X/G \x{\simeq}{\lrar} \Y $$
Hence what we want to show that if $\Y$ is a free $G$-quotient by
$$ f: \X \lrar \Y $$
then $P_n(\Y)$ is a free $G$ quotient  by
$$ P_n(f): P_n(\X) \lrar P_n(\Y) $$

Now it is easy to see that $\Y$ is a free $G$-quotient by $f$ if and only if $f$ fits in a pullback diagram of the form
$$\xymatrix{
\X \ar[d]^f \ar[r] & \E G\ar[d]\\
\Y \ar[r] & \B G
}$$

Since $P_n$ has a left adjoint it commutes with pullbacks we can apply it to the above diagram and get another pullback square
$$\xymatrix{
P_n(\X) \ar[d]^f \ar[r] & P_n(\E G) \ar[d]\\
P_n(\X/G) \ar[r]  & P_n(\B G)
}$$

However both $\E G$ and $\B G$ are nerves of categories and thus since $n \geq 1$ we have
$$ P_n(\E G) = \E G $$
and
$$ P_n(\B G) = \B G $$
Thus we get that pull back diagram:
$$\xymatrix{
P_n(\X) \ar[d]^f \ar[r] & \E G \ar[d]\\
P_n(\X/G) \ar[r]  &  \B G
}$$
and so $P_n(\Y)$ is indeed a free $G$-quotient by $P_n(f)$.
\end{proof}
We now proceed with the proof of Proposition ~\ref{p:etale-is-quotient}. From Lemma ~\ref{l:free-cofinal} we may replace
$ \{\SimpS_{\U,n}\}_{\U \in I(X), n \in \NN} $
with
$ \{\SimpS_{\U,n}\}_{(\U,n) \in J(X)} $.
Now the action of $\Gam_K$ on the $(n+1)$-skeleton $\SimpS_{\U}$ factors through a free action of a finite quotient $G$ of $\Gam_K$. Then we have a free action of $G$ on $Q_n(\X_{\U})$. Applying Lemma ~\ref{l:postinikov-quotient} to $Q_n(\X_{\U})$ we get that
$$ P_n(\X_{\U}/\Gam_K) = P_n(Q_n(\X_{\U}/\Gam_K)) \cong P_n(Q_n(\X_{\U})/G) \cong $$
$$ P_n(Q_n(\X_{\U}))/\Gam_K = P_n(\X_{\U}) = \X_{\U,n}/\Gam_K $$
for every $n \geq 1$. This means that
$$ \acute{E}t^{n}(X) = {\acute{E}t^{n}_{/K}}(X)/\Gam_K $$
and we are done.
\end{proof}

\subsection{ A Generalized Version of Grothendieck's Obstruction}
Going back to the map
$$ h_n:X(K) \lrar X^n(hK) $$

We see that
$$ X^n(hK) = \emptyset \Rightarrow X(K) =\emptyset $$
and so one can consider the emptiness of each $X^n(hK)$ as an obstruction to existence of a rational point. The following lemma shows that for $n=1$ this obstruction is actually Grothendieck's section obstruction. See also Pal ~\cite{Pal10} for a similar discussion.

Keeping in mind Remark ~\ref{r:etale-is-ok} and the fact that the pro-fundamental group of (a pointed lift of) $\acute{E}t(X)$ can be identified with the \'etale fundamental group of $X$ we can use Theorem ~\ref{t:homotopy-fixed-point-iff-splits}, Proposition ~\ref{p:etale-is-quotient} and  Lemma ~\ref{l:section-iff-homotopy-section} to obtain the following Corollary:

\begin{cor}\label{l:grothendieck-h-1}
Let $K$ be a field and $X$ a geometrically connected smooth variety. Then the following conditions are equivalent:
\begin{enumerate}
\item
The set $X^1(hK)$ is nonempty.
\item
The sequence of pro-finite groups
$$ 1 \lrar \pi_1(\ovl{X})\rar \pi_1(X) \lrar \Gam_K \lrar 1 $$
admits a continuous section.
\item
The map
$$ \acute{E}t^{1}(X) \lrar \acute{E}t^{1}(\spec(K)) $$
admits a section in $\Pro\Ho\left(\Set^{\Del^{op}}_{\Gam_K}\right)$.
\end{enumerate}
\end{cor}

\section{ Homotopy Fixed Point Sets and Pro-Isomorphisms }
In this section we assume that $K$ is a number field. Let $X$ be a $K$-variety. Let
$$ f: \acute{E}t_{/K}(X) \lrar \Y_I $$
be map in $\Pro\Ho\left(\Set^{\Del^{op}}_{\Gam_K}\right)$. One can then consider the subset of adelic points
$$ X(\A)^{\Y_I} \subseteq X(\A) $$
containing the points $(x_\nu)$ whose corresponding homotopy fixed point
$$ f_*(h((x_\nu))) \in \Y_I(h\A) $$ is rational, i.e. comes from $\Y_I(hK)$. In general this obstruction can only be weaker then the \'etale homotopy obstruction, i.e. $X^h(\A) \subseteq X(\A)^{\Y_I}$. However the freedom to replace $\acute{E}t_{/K}(X)$ with $\Y_I$ can be useful.

In this section we will prove the following theorem:
\begin{thm}\label{t:obstruction-iso}
Let $K$ be a number field and $X$ a $K$-variety. Let $\X_I,\Y_I \in \Pro\Ho\left(\Set^{\Del^{op}}_{\Gam_K}\right)$ be two objects such that $I$ and $J$ are countable and all the $\X_\alp$'s and $\Y_\alp$'s are finite, excellent and connected. Let
$$ \xymatrix{
& \X_I \ar^{f}[dd] \\
\acute{E}t_{/K}(X) \ar[ur]\ar[dr] & \\
& \Y_I \\
}$$
be a commutative triangle such that $f$ induces an isomorphism
$$ \ovl{f}: \ovl{\X}^{\natural}_I \lrar \ovl{\Y}^{\natural}_I $$
where
$$ \ovl{(\bullet)}: \Pro\Ho\left(\Set^{\Del^{op}}_{\Gam_K}\right) \lrar
\Pro\Ho\left(\Set^{\Del^{op}}\right) $$
is the forgetful functor. Then
$$ X(\A)^{\X_I} = X(\A)^{\Y_I} $$
\end{thm}

Now before we come to the proof of ~\ref{t:obstruction-iso} we will need to develop some terminology regarding torsors of pointed sets.

\begin{define}
Let $A_*$ be a pointed set. An \textbf{$A_*$-Torsor} is a \textbf{non-empty} set $B$ together with an map
$$ a: A_* \times B \lrar B$$
such that
\begin{enumerate}
\item
$a(*,b) = b$ for all $b\in B$.
\item
The map
$$ A_* \times B \x{a\times p_2}{\lrar} B \times B $$
is an isomorphism of sets.
\end{enumerate}
We call the map $a$ the "action" of $A_*$ on $B$.
\end{define}
We denote by $\Tors(\Set_*,\Set)$ the category whose objects are triples $(A_*,B,a)$ such that $B$ is an $A_*$-torsor with action $a$ and whose morphisms are maps of pairs $(A_*,B) \lrar (A_*',B')$ which commute with the respective actions.

We have two natural functors:

$$\mathfrak{A}_* : \Tors(\Set_*,\Set) \lrar \Set_* $$
$$\mathfrak{A}_*((A_*,B,a))  =  A_*$$
and
$$\mathfrak{B} : \Tors(\Set_*,\Set) \lrar \Set $$
$$\mathfrak{B}((A_*,B,a))  =  B$$

\begin{lem}\label{l:pro-set-torsor}
Let $I$ be a cofiltered indexing poset and $\{(A_\alp,B_\alp,a_\alp)\}_{\alp \in I}, \{(A'_\bet,B'_\bet,a'_\bet)\}_{\bet \in I}$ two pro-object in $\Tors(\Set_*,\Set)$. Let
$$ F:\{(A_\alp,B_\alp,a_\alp)\} \lrar \{(A'_\alp,B'_\alp,a'_\alp)\} $$
be a levelwise map such that $\mfrak{A}(F)$ is an isomorphism in $\Pro\;\Set_*$. Then $\mfrak{B}(F)$ is an isomorphism in $\Pro\;\Set$.
\end{lem}
\begin{proof}
We can consider $F$ as a compatible family of commutative diagrams
$$\xymatrix{
A_{\alp} \times B_\alp \ar[r]^{a_\alp} \ar[d]^{f_\alp \times g_\alp }& B_\alp  \ar[d]^{g_\alp}\\
A'_{\alp}\times B_\alp'\ar[r]^{a'_\alp}  & B_\alp'
}$$

Now the map $\mathfrak{A}_*(F)$ in $\Pro(\Set_*)$  is represented by the levelwise map $f_\alp$. Since
$f_I = \{f_\alp\}_{\alp\in I}$ is an isomorphism we have a map
$$f_I^{-1}: \{A'_\alp\}_{\alp\in I} \lrar   \{A_\alp\}_{\alp\in I}$$
which is the inverse of $f_I$. Note that $f^{-1}_I$ can be represented by a map $\theta : I  \lrar I$
and a collection of maps  $f^{-1}_\alp: A'_{\theta(\alp)} \to A_{\alp}$  satisfying some natural compatibly conditions (See ~\cite{EHa76}). Without loss of generality we may assume that  $\theta(\alp) \geq \alp$ for all $\alp$.

We would like to construct an inverse to the map $\mathfrak{B}(F)$ in $\Pro(\Set)$  which is represented by the levelwise map $g_\alp$. We denote this inverse by
$$g^{-1}_I = \{B'_\alp\}_{\alp\in I} \lrar   \{B_\alp\}_{\alp\in I}$$
To describe $g^{-1}_I$ we will give a collection of maps
$$ g^{-1}_\alp: B'_{\theta(\alp)} \lrar B_{\alp} $$
and we leave it to the reader to verify the easy diagram chasing that is required to check that indeed the maps $g^{-1}_\alp$ are defining a pro-map which is an inverse to $g_I = \mathfrak{B}(F)$.

Now let $\alp \in I$ be an index. Choose an arbitrary element $b_0 \in B_{\theta(\alp)}$. We shall denote by $\tilde{b_0} \in B_\alp$ the image of $b_0$ by the structure map $B_{\theta(\alp)} \lrar  B_{\alp}$ and by
$$ b_0' = g_{\theta(\alp)}(b_0) \in B'_{\theta(\alp)} $$
the respective image in $B'_{\thet(\alp)}$.

Now for every $b'\in B'_{\theta(\alp)}$ we can apply the inverse of the map
$$ A'_{\theta(\alp)} \times B'_{\theta(\alp)} \lrar
   B'_{\theta(\alp)} \times B'_{\theta(\alp)} $$
to the tuple $(b',b_0')$ and get an element $(c',b_0')\in A'_{\theta(\alp)} \times B'_{\theta(\alp)}$. We then define
$$ g^{-1}_{\alp}(b') = a(f^{-1}_\alp(c'),\tilde{b_0}) \in B_{\alp} $$
\end{proof}

Now in the proof of Theorem ~\ref{t:obstruction-iso} we will need to work with pro-objects of various kinds which carry actions of pro-objects in the category of finite groups. In order to work well with such objects it will be useful to recall first that pro objects in the category of finite groups are the same as pro-finite groups in the usual sense (see ~\ref{pro-finite-is-pro-finite}). We will be interested in the following kind of actions

\begin{define}
Let $C$ be a category, $G$ a pro-finite group and $\{\X_\alp\}$ an object in $\Pro\;C$. An \textbf{excellent} action of $G$ on $\{\X_\alp\}$ is a an action of $G$ on $\{\X_\alp\}$ via levelwise maps such that the induced action on each $\X_\alp$ factors through a finite (continuous) quotient of $G$. In this case we will say $\{\X_\alp\}$ is an excellent $G$-pro-object.

We say that a map in $\Pro\;C$ between two excellent $G$-pro-objects
$$ f: \{\X_\alp\} \lrar \{\Y_\bet\} $$
is \textbf{equivariant} if $g \circ f \circ g^{-1} = f$ for every $g \in G$.
\end{define}

We then have the following basic lemma:
\begin{lem}
Let $C$ be a category, $G$ a pro-finite group and let
$$ f:\{\X_\alp\} \lrar \{\Y_\bet\} $$
be an equivariant map between two excellent $G$-objects. Then $f$ can be represented by a compatible family of maps
$$ f_\bet: \X_{\alp(\bet)} \lrar \Y_\bet $$
such that each $f_\bet$ is equivariant.
\end{lem}
\begin{proof}
It is enough to prove for $\{\Y_\bet\}$ which is a single space $\Y$.

Start with any map $f': \X_\alp \lrar \Y$ representing $f$ and consider the \textbf{finite} orbit $\{f' \circ g\}_{g \in G}$. Since $f \circ g = f$ for every $g \in G$ we get that all the maps $f' \circ g$ represent the same map $f$. Thus there exist some $\alp \leq \beta \in I$ such that the images of all the $f' \circ g$'s in $\Hom(\X_\beta ,\Y)$ are the same. We will denote this unified image by
$$ f_\beta: \X_\beta \lrar \Y $$
Then we see that $f_\beta$ is an equivariant map representing $f$ and we are done.
\end{proof}
In particular we see that any equivariant map
$$ f:\{\X_\alp\} \lrar \{\Y_\bet\} $$
between two excellent $G$-objects induces a map
$$ f':\{\X_\alp\} \x{\simeq}\lrar \{\Y_\bet\} $$
in the pro-category of $G$-objects in $C$.

This simple observation as several useful corollaries:
\begin{cor}\label{c:equivariant-iso}
Let
$$ f:\{\X_\alp\} \lrar \{\Y_\bet\} $$
be an equivariant map between two excellent $G$-objects such that the underlying map of pro-objects is an isomorphism. Then $f$ is induces an isomorphism in the pro-category of $G$-objects.
\end{cor}

\begin{cor}\label{c:pro-quotient-categorical}
Let $\X_I = \{\X_\alp\}$ be an excellent $G$-object. Then the levelwise quotient
$$ \X_I/\Gam = \{\X_\alp/\Gam\} $$
coincides with the categorical quotient (i.e. with the corresponding colimit) in $\Pro\;C$.
\end{cor}

\begin{cor}\label{c:pro-quotient}
Let
$$ f:\{\X_\alp\} \lrar \{\Y_\bet\} $$
be an equivariant map between two excellent $G$-objects such that the underlying map of pro-objects is an isomorphism. Then $f$ is induces an isomorphism in $\Pro\;C$
$$ \{\X_\alp/G\} \x{\simeq}\lrar \{\Y_\bet/G\} $$
\end{cor}

We are now ready to prove the main ingredient of the proof:
\begin{prop}\label{p:pro-homotopy-fixed-points-iso}
Let $\Gam$ be a pro-finite group and $\X_I = \{\X_\alp\},\Y_J = \{\Y_\bet\}$ be two objects in $\Pro\Ho\left(\Set^{\Del^{op}}_{\Gam_K}\right)$ such that $I,J$ are countable and all the $\X_\alp$'s and $\Y_\alp$'s are excellent, finite, connected and $2$-truncated. Let $f: \X_I \lrar \Y_J$ be a map such that
$$ \ovl{f}: \ovl{\X}_I \lrar \ovl{\Y}_I $$
is an isomorphism. Then $f$ induces an isomorphism of sets
$$ \X_I\left(\E\Gam^{\natural}\right) \simeq \Y_I\left(\E\Gam^{\natural}\right) $$
\end{prop}
\begin{proof}
Recalling the discussion in the beginning of subsection \S\S ~\ref{ss:pro-fundamenral-group} we start by lifting $f$ to a levelwise map
$$ \wtl{f}: \wtl{\X}_{\NN} \lrar \wtl{\Y}_{\NN} $$
in $\Pro\Set^{\Del^{op}}_{\Gam_K}$ between two pro-finite fibration towers.

Let $\wtl{\X}^1_\NN = \left\{P_1\left(\wtl{\X}_n\right)\right\}$ and $\wtl{\Y}^1_{\NN} = \left\{P_1\left(\wtl{\Y}_n\right)\right\}$. We will start by showing that
$$ \wtl{\X}^1_{\NN}\left(\E\Gam^{\natural}\right) \simeq \wtl{\Y}^1_{\NN}\left(\E\Gam^{\natural}\right) $$

Choose a compatible set of (not-necessarily-$\Gam$-invariant) base points $x_n \in \X_n$ and let $y_n = f_n(x_n) \in \Y_n$. Let $G_i = \pi_1\left(\wtl{\X}_n,x_n\right)$ and $H_n = \pi_1\left(\wtl{\Y}_n,y_n\right)$.

Now the map $f$ induces an isomorphism of short exact sequences (note that any bijective map of pro-finite groups is an isomorphism because the topology on both is Hausdorf-compact):
$$ \xymatrix{
1 \ar[r] & \pi_1\left(\X_\NN,\{x_n\}\right) \ar[r]\ar[d] & \pi_1\left((\X_\NN)_{h\Gam},\{x_n\}\right) \ar[r]\ar[d] &
\Gam \ar@{=}[d] \ar[r] & 1 \\
1 \ar[r] & \pi_1\left(\Y_\NN,\{x_n\}\right) \ar[r] & \pi_1\left((\Y_\NN)_{h\Gam},\{x_n\}\right)  \ar[r] &
\Gam \ar[r] & 1 \\
}$$

and so the first sequence splits if and only if the second does. Hence from Theorem ~\ref{t:homotopy-fixed-point-iff-splits} we see that
$$ \wtl{\X}^1_{\NN}\left(\E\Gam^{\natural}\right) \neq \emptyset $$
if and only if
$$ \wtl{\Y}^1_{\NN}\left(\E\Gam^{\natural}\right) \neq \emptyset $$
Hence it is enough to deal with the case both of them are non-empty. Since $\wtl{\X}^1_n$ is a fibration tower we get that $\{(\wtl{\X}^1_n)^{h\Gam}\}$ is a tower of Kan fibrations and so one can choose a compatible family $\wtl{h}_n \in (\wtl{\X}^1_n)^{h\Gam}$.

Now from Corollary ~\ref{c:equivariant-iso} we get that the equivariant isomorphism of pro-finite groups
$$ \left\{\pi_1\left(\wtl{\X}^1_n,\wtl{h}_n\right)\right\} \lrar \left\{\pi_1\left(\wtl{\Y}^1_n,f_n(\wtl{h}_n)\right)\right\} $$
together with the spectral sequence of ~\ref{t:ss} induce an isomorphism
$$ \wtl{\X}^1_I\left(\E\Gam^{\natural}\right) \simeq \lim_n \wtl{\X}^1_n(hK) = \lim_n H^1\left(\Gam,\pi_1\left(\wtl{\X}^1_n,\wtl{h}_n\right)\right) \cong $$
$$ \lim_n H^1\left(\Gam,\pi_1\left(\wtl{\Y}^1_n,f_n(\wtl{h}_n)\right)\right) \cong \lim_n \wtl{\Y}^1_n(hK) \cong \wtl{\Y}^1_\NN\left(\E\Gam^{\natural}\right) $$

Now consider the commutative square
$$ \xymatrix{
\X_\NN\left(\E\Gam^{\natural}\right) \ar^{f_*}[r]\ar[d] & \Y_\NN\left(\E\Gam^{\natural}\right) \ar[d] \\
\X^1_\NN\left(\E\Gam^{\natural}\right) \ar[r] & \Y^1_\NN\left(\E\Gam^{\natural}\right) \\
}$$

In order to finish the proof we will show that for every $\{h_n\} \in \X_\NN^1\left(\E\Gam^{\natural}\right)$ the map $f_*$ maps the fiber over $\{h_n\}$ isomorphically to the fiber over $\{f_n(h_n)\}$.

As above choose a compatible family $\wtl{h}_n \in (\wtl{\X}^1_n)^{h\Gam}$ such that $h_n$ is the connected component of $\wtl{h}_n$.

First of all we need to address the possibility that this fiber above $\{h_n\}$ is empty. By Proposition ~\ref{p:obstruction-theory} we see that for every $n$ we have an obstruction element
$$ o_n \in H^3\left(\Gam,\pi_2\left(\wtl{\X}_n,\wtl{h}_n\right)\right) $$
which is trivial if and only if $h_n$ lifts to $\pi_0\left(\wtl{\X}_n^{h\Gam}\right) \cong \pi_0\left(P_2\left(\wtl{\X}_n\right)^{h\Gam}\right)$. From Corollary ~\ref{c:equivariant-iso} the equivariant isomorphism of pro-finite groups
$$ \left\{\pi_2\left(\wtl{\X}_n,\wtl{h}_n\right)\right\} \lrar \left\{\pi_2\left(\wtl{\Y}_n,f_n(\wtl{h}_n)\right)\right\} $$
induces an isomorphism of pro-finite groups
$$ \left\{H^3\left(\Gam,\pi_2\left(\wtl{\X}_n,\wtl{h}_n\right)\right)\right\} \lrar \left\{H^3\left(\Gam,\pi_2\left(\wtl{\X}_n,\wtl{h}_n\right)\right)\right\} $$
and so the element
$$ \{o_n\} \in \lim_nH^3\left(\Gam,\pi_2\left(\wtl{\X}_n,\wtl{h}_n\right)\right) $$
is zero if and only if the element
$$ \{f_n(o_n)\} \in \lim_nH^3\left(\Gam,\pi_2\left(\wtl{\Y}_n,f_n\left(\wtl{h}_n\right)\right)\right) $$
is zero. If both of them are non-zero then both fibers are empty and we are done. Hence we can assume that both are zero.

Let $A_n \subseteq \pi_0\left(\wtl{\X}_n^{h\Gam}\right)$ be the (non-empty) pre-image of $h_n$ and $B_n \subseteq \pi_0\left(\wtl{\Y}_n^{h\Gam}\right)$ the pre-image of $f_n(h_n)$. We need to show that the induced map
$$ f_*:\lim A_n \lrar \lim B_n $$
is an isomorphism.

Now given any homotopy fixed point $x \in \wtl{\X}_n^{h\Gam}$ we can run the spectral sequence of $~\ref{t:ss}$ to compute $\pi_0\left(\wtl{\X}_n^{h\Gam}\right)$. Since $\wtl{\X}_n$ is $2$-truncated this spectral sequence will degenerate in the third page and hence we see that in order to compute the cell $E^{\infty}_{2,2} = E^{3}_{2,2}$ it is enough to choose the image of $x$ in $P_1(\wtl{\X}_n)^{h\Gam}$. In particular we can calculate this cell using the homotopy fixed point $\wtl{h}_n$.

Now by analyzing the spectral sequence we see that the pointed set $E^3_{2,2}(\wtl{h}_n)$ is obtained from the group $H^2\left(\Gam_K,\pi_2\left(\wtl{\X}_n,\wtl{h}_n\right)\right)$ by taking the quotient \textbf{pointed set} under the action of $H^0\left(\Gam_K,\pi_1\left(\wtl{\X}_n,\wtl{h}_n\right)\right)$ which is induced by the action of $H^0\left(\Gam_K,\pi_1\left(\wtl{\X}_n,\wtl{h}_n\right)\right)$ on $\pi_2\left(\wtl{\X}_n,\wtl{h}_n\right)$.

Now the equivariant isomorphism of pro-groups
$$ \left\{\pi_2\left(\wtl{\X}_n,\wtl{h}_n\right)\right\} \lrar \left\{\pi_2\left(\wtl{\Y}_n,f_n(\wtl{h}_n\right)\right\} $$
induces an isomorphism of pro-groups
$$ \left\{H^2\left(\Gam_K, \pi_2\left(\wtl{\X}_n,\wtl{h}_n\right)\right)\right\} \lrar
\left\{H^2\left(\Gam_K,\pi_2\left(\wtl{\Y}_n,f_n\left(\wtl{h}_n\right)\right)\right)\right\} $$

Now from Corollary ~\ref{c:equivariant-iso} we get that the equivariant isomorphism of pro-finite groups
$$ \left\{\pi_1\left(\wtl{\X}_n,\wtl{h}_n\right)\right\} \lrar \left\{\pi_1\left(\wtl{\Y}_n,f_n(\wtl{h}_n)\right)\right\} $$
induces an isomorphism of pro-finite groups
$$ \left\{H^0\left(\pi_1\left(\wtl{\X}_n,\wtl{h}_n\right)\right)\right\} \lrar \left\{H^0\left(\pi_1\left(\wtl{\Y}_n,f_n(\wtl{h}_n)\right)\right)\right\} $$

Hence from Corollary ~\ref{c:pro-quotient} we get that the induced map of pro-pointed-sets
$$ \left\{E^{\infty}_{2,2}\left(\wtl{h}_n\right)\right\} \lrar \left\{E^{\infty}_{2,2}\left(f_n(\wtl{h}_n\right)\right\} $$
is an isomorphism. Now from the spectral sequence of ~\ref{t:ss} we see that the set $A_n$ admits a natural torsor structure under $E^{\infty}_{2,2}\left(\wtl{h}_n\right)$ and similarly $B_n$ under $E^{\infty}_{2,2}\left(f_n(\wtl{h}_n)\right)$. From Lemma ~\ref{l:pro-set-torsor} we then get that the induced map of pro-sets
$$ \{A_n\} \lrar \{B_n\} $$
is an isomorphism and induces an isomorphism
$$ \lim A_n \x{\simeq}{\lrar} \lim_n B_n $$

\end{proof}

We now come to the proof of Theorem ~\ref{t:obstruction-iso}. From Theorem ~\ref{t:descent-theorem} and Proposition ~\ref{p:2-is-enough} we can assume that all the spaces in the diagrams of $\X_I$ and $\Y_I$ are $2$-truncated. The Theorem will then follow from the following corollary:
\begin{cor}\label{c:obstruction-invariant}
Let $K$ be a number field and $\X_I = \{\X_\alp\},\Y_J = \{\Y_n\}$ two objects in $\Pro\Ho\left(\Set^{\Del^{op}}_{\Gam_K}\right)$ such that $I$ and $J$ are countable and all the $\X_\alp$'s and $\Y_\alp$'s are finite, excellent, connected and $2$-truncated. Let $f: \X_I \lrar \Y_J$ be a map such that
$$ \ovl{f}: \ovl{\X}_I \lrar \ovl{\Y}_I $$
is an isomorphism. Assume that $\X_I(h\A) \neq \emptyset$. Then an element $(x_\nu) \in \X_I(h\A)$ is rational (i.e. is the image of am element in $\X_I(hK)$) if and only if $f_*(x_\nu) \in \Y_I(\A)$ is rational.
\end{cor}
\begin{proof}
From Proposition ~\ref{p:pro-homotopy-fixed-points-iso} we get that $f$ induces isomorphisms of sets:
$$ \X_I(hK) \x{\simeq}{\lrar} \Y_I(hK) $$
and
$$ \prod_\nu\X_I(hK_\nu) \x{\simeq}{\lrar} \prod_\nu\Y_I(hK_\nu) $$
Now since $\X_I(h\A) \subseteq \prod_\nu\X_I(hK_\nu)$ the result follows.
\end{proof}

\section{Varieties of Dimension Zero}\label{s:dimension-zero}
\begin{prop}\label{p:zero-dim}
Let $X$ be a zero dimensional variety. Then for every $0 \leq n \leq \infty$ one has
$$ X(\A)^{h,n} = X(\A)^{\ZZ h,n} = X(K) $$
\end{prop}
\begin{proof}
Consider the hypercovering $X_\bullet \lrar X$ where $X_\bullet$ is the constant simplicial variety $X_n = X$. Then $P_0(\ZZ\SimpS_{X_\bullet})$ is a discrete simplicial set which is the $\Gam_K$ module freely generated by $X(\ovl{K})$. From the discreteness of $P_0(\ZZ\SimpS_{X_\bullet})$ and the fact that
$$ X(\ovl{K}) \subseteq P_0(\ZZ\SimpS_{X_\bullet}) $$
we see that

$$ X(\A)^{\ZZ X_\bullet, 0} = X(K) $$

Now since
$$ X(K) \subseteq X(\A)^{h,n} \subseteq X(\A)^{\ZZ h,n} \subseteq X(\A)^{\ZZ X_\bullet, 0} = X(K) $$
we get that
$$ X(K) = X(\A)^{\ZZ h,n} = X(\A)^{h,n} $$
as required.
\end{proof}
\begin{cor}\label{c:zero-dim-fin}
Let $X$ be a zero dimensional variety. Then
$$ X(\A)^{fin,h} = X(\A)^{fin,\ZZ h} = X(K) $$
\end{cor}
\begin{proof}
Immediate from prop ~\ref{p:zero-dim} and the fact that
$$ X(\A)^{fin,h} \subseteq X(\A)^{fin,\ZZ h} \subseteq X(\A)^{\ZZ h} $$
\end{proof}

\begin{cor}\label{c:coprod}
If $X = X_1 \coprod  X_2$ then
$$ X(\A)^{h,n} = X_1(\A)^{h,n} \coprod  X_2(\A)^{h,n} $$
$$ X(\A)^{\ZZ h,n} = X_1(\A)^{\ZZ h,n} \coprod  X_2(\A)^{\ZZ h,n} $$
$$ X(\A)^{fin,h} = X_1(\A)^{fin,h} \coprod  X_2(\A)^{fin,h} $$
$$ X(\A)^{fin,\ZZ h} = X_1(\A)^{fin,\ZZ h} \coprod  X_2(\A)^{fin,\ZZ h} $$

\end{cor}
\begin{proof}
Consider the obvious map
$$ X \lrar \spec(K) \coprod \spec(K) $$
and use functoriality together with Proposition ~\ref{p:zero-dim} and Corollary ~\ref{c:zero-dim-fin}.
\end{proof}

\begin{cor}\label{c:catches-pi-0}
If $\ovl{X}$ does not have a $\Gam_K$-invariant connected component then
$$ X(\A)^{fin,h} = X(\A)^{fin,\ZZ h} = X(\A)^{h,n} = X(\A)^{\ZZ h,n} = \emptyset $$
\end{cor}
\begin{proof}
Again consider the map $X \lrar \pi_0(X)$ where $\pi_0(X)$ is considered as a zero dimensional $K$-variety. Since $X$ has no $\Gam_K$-invariant connected component $\pi_0(X)$ has no rational points. The result then follows from functoriality and Proposition ~\ref{p:zero-dim}.
\end{proof}

\section{ Connection to Finite Descent}\label{s:finite-descent}

In this section we will consider torsors
$$ f: Y \lrar X $$
under a finite $K$-group $G$. We will connect the classical obstructions \textbf{finite descent} and \textbf{finite abelian descent} to our homotopy and homology obstructions respectively.

We would often think of $G$ just as the finite group $G(\ovl{K})$ with the action of $\Gam_K$ on it. We then consider the standard model for $\B G$ as the nerve of the groupoid ${\mcal{B}G}$ with one object and morphism set $G$. We have a $\Gam_K$ action on this groupoid and so an action of $\Gam_K$ on our model for $\B G$.

In this model $\B G$ has a single vertex giving it a natural base point preserved by $\Gam_K$. Since $\B G$ is also bounded we can use this base point to compute
$$ \B G(hK) =\pi_0\left(\B G^{h\Gam_K}\right) $$
via the spectral sequence. Since $\B G$ is connected and $\pi_1$ is its only non-trivial homotopy group we can identify
$$ \B G(hK) = H^1(K, G) $$

Let us try to make this identification more explicit. Let $\Lam \fns \Gam_K$ and define $N = \Gam_K/\Lam$. Then the  simplicial set $\E N$ can be realized as the nerve of the groupoid $\mcal{E}N$ whose objects are $N$ and whose morphism sets are all singletons. Now suppose that our homotopy fixed point is given by a map
$$ H:E N \lrar \B G $$
which is induced by a map of groupoids
$$ \mcal{H}:\mcal{E}N \lrar {\mcal{B}G} $$
We can identify $N$ with the union
$$ \cup_{\sig \in N} \Hom(1,\sig) $$
Then $\mcal{H}$ maps this set to the morphism set from the single object of ${\mcal{B}G}$ to itself which can be identified with $G$. The obtained map $\alp:N \lrar G$ gives the desired $1$-cocycle in $H^1(K,G)$.

Note that since $\mcal{H}$ is $N$-equivariant it is completely determined by $\alp$. If one uses this determination in order to write the condition that $\mcal{H}$ respects composition in terms of $\alp$ one will find exactly the familiar $1$-cocycle condition.

Now consider the torsor $Y \lrar X$ as an \'{e}tale covering. By the \'{C}ech construction we can construct a hypercovering $\check{Y}$ given by
$$ \check{Y}_n = \overbrace{Y \times_{X} ... \times_{X} Y}^{n+1} $$
Let us now analyze the simplicial $\Gam_K$-set $\SimpS_{\check{Y}}$. We have a natural isomorphism (over $K$)
$$ \overbrace{Y \times_{X} ... \times_{X} Y}^{n+1} \cong
Y \times G^n $$
and so
$$ \left(\SimpS_{\check{Y}}\right)_n = \ovl{\pi}_0(\check{Y}_n) \cong \ovl{\pi}_0(Y) \times G^n $$
We then observe that we can identify $\SimpS_{\check{Y}}$ with the nerve of the groupoid $\mcal{Y}$ whose object set is $\ovl{\pi}_0(Y)$ and whose morphism sets are
$$ \Hom(C_1,C_2) = \{g \in G | gC_1 = C_2 \}\;\;, C_1,C_2 \in \ovl{\pi}_0(Y) $$

Note the group $\Gam_K$ acts naturally on $\mcal{Y}$ and this action is compatible with its action on $\SimpS_{\check{Y}}$. Hence $\mcal{Y}$ is actually a $\Gam_K$-groupoid.

${\mcal{B}G}$ is also a $\Gam_K$-groupoid and we have an equivariant map of groupoids
$$ \mcal{Y} \lrar {\mcal{B}G} $$
induced by the natural inclusion $\Hom(C_1,C_2) \subseteq G$. This groupoid map induces a map
$$ c_Y: \SimpS_{\check{Y}} \lrar \B G $$
and hence a map (which by abuse of notation we denote)
$$ c_Y: \SimpS_{\check{Y}}(hK) \lrar \B G(hK) = H^1(K,G) $$
By abuse of notation we will call the adelic version of this map by the same name:
$$ c_Y: \SimpS_{\check{Y}}(h\A) \lrar \B G(h\A) = H^1(\A,G) $$

Note that we have a map
$$ p_Y:X\left(hK\right) \lrar \SimpS_{\check{Y}}(hK) $$
We would sometimes abuse notation and refer to the composition $c_Y \circ p_Y$ as $c_Y$ as well.

\begin{lem}\label{l:classifies-fiber}
Let $p \in X(K)\;\; (X(\A))$ be a point and
$$ x_p = h(p) \in X\left(hK\right) \;\;\; \left(X(h\A)\right) $$
the corresponding homotopy fixed point. Then
$$ {c_Y}(x_p) \in H^1(K,G) \;\;\; \left(H^1(\A,G)\right) $$
is the element which classifies the $G$-torsor  $Y_p = f^{-1}(p)$.
\end{lem}
\begin{proof}
We will prove the lemma for $K$. The proof of the adelic version is completely analogous. From functoriality it is enough to prove the claim for $X = p$ and $Y = Y_p$.

In this case $\mcal{Y}$ is the groupoid whose objects set is a principle homogenous $G$-set classified by an element $u \in H^1(K,G)$ and there is a unique morphism between any two objects.

Let $y_0 \in Y$ be an arbitrary base point and let $\Lam \fns \Gam_K$ be such that both the action of $\Gam_K$ on $Y$ and the element $u$ factor through $N = \Gam_K/\Lam \lrar G$. Then we can represent $u$ by a $1$-cocycle $\sig \lrar u_\sig$ such that
$$ \sig y_0 = u_\sig(y_0) $$
for every $\sig \in N$.

Now since the elements in $\mcal{Y}$ don't have automorphisms we see that $\SimpS_{\check{Y}}$ is contractible and so $\SimpS_{\check{Y}}(hK) = *$. This means that if we construct any concrete element in $\SimpS_{\check{Y}}(hK)$ then it has to coincide with $x_p$.

The homotopy fixed point which we will construct is the one given by the equivariant map of groupoids
$$ \vphi: \mcal{E}N \lrar \mcal{Y} $$
which sends the object $\sig \in N$ to the object $\sig y_0 \in \mcal{Y}$ (since all the morphism sets are singletons there is no problem to construct such a map). We then need to compose the map $\vphi$ with the groupoid map $\mcal{Y} \lrar {\mcal{B}G}$ and look and the obtained map
$$ \psi:\mcal{E}N \lrar {\mcal{B}G} $$
In order to find the corresponding element in $H^1(K,G)$ we need to check to which morphism $\psi$ sends the unique morphism from $1$ to $\sig \in N$. Now $\vphi$ sends this morphism to the unique morphism from $y_0$ to $\sig y_0$. Since
$$ \sig y_0 = u_\sig(y_0) $$
we see that $\psi$ sends this morphism to the morphism $u_\sig \in G$. This finishes the proof of the lemma.
\end{proof}

\begin{cor}\label{c:cor-of-classifies-fiber}
Let $K$ and $X$ be as above. Then an adelic point $(x_\nu) \in X(\A)$ is in $X(\A)^{f}$ (i.e. survives finite descent) if and only if the homotopy fixed point
$$ c_Y(h((x_\nu))) \in \B G(h\A) $$
is rational for all torsors $f:Y \lrar X$ under finite $K$-groups.
\end{cor}

The main result of this section is the following
\begin{thm}\label{t:fin}
Let $K$ be field and $X$ a smooth connected $K$-variety. Then
$$ X(\A)^{fin} = X(\A)^{h,1} $$
$$ X(\A)^{fin-ab} = X(\A)^{\ZZ h,1} $$
\end{thm}

\begin{proof}
\begin{define}
Let
$$ X(\A)^{h-tor} \subseteq X(\A) $$
be the set of adelic points such that for every torsor $f: Y \lrar X$ as above the adelic homotopy fixed point
$$ h_{\check{Y}}((x_\nu)) \in \SimpS_{\check{Y}}(h\A) $$
is rational.
\end{define}
\begin{define}
Let
$$ X(\A)^{\ZZ h-tor} \subseteq X(\A) $$
be the set of adelic points such that for every $f: Y \lrar X$ as above the adelic homotopy fixed point
$$ h_{\ZZ\check{Y},1}((x_\nu)) \in P_1((\ZZ \SimpS_{\check{Y}}))(h\A) $$
is rational.
\end{define}

We will show that $X(\A)^{h,1} = X(\A)^{fin}$ by showing that both sets are equal to $X(\A)^{h-tor}$. Similarly we will show that $X(\A)^{\ZZ h,1} = X(\A)^{fin-ab}$ by showing that both sets are equal to $X(\A)^{\ZZ h-tor}$.

\begin{prop}\label{p:descent-is-check}
Let $K$ and $X$ be as above. Then
$$ X(\A)^{fin} = X(\A)^{h-tor} $$
$$ X(\A)^{fin-ab} = X(\A)^{\ZZ h-tor} $$
\end{prop}
\begin{proof}



From Corollary ~\ref{c:cor-of-classifies-fiber} we get immediately that
$$ X(\A)^{h-tor} \subseteq X(\A)^{fin} $$
In order to get
$$ X(\A)^{\ZZ h-tor} \subseteq X(\A)^{fin-ab} $$
one should notice that for abelian $G$ the natural map
$$\B G \lrar P_1(\ZZ \B G)^1$$
is a weak equivalence, when the $()^1$ represents restricting to the degree one component.

We now want to show the inverse inclusions. We shall write the proof of the non-abelian case and explain the relevant modifications for the abelian group in the end of this proof.

If $X(\A)^{fin}$ is empty we are done. Otherwise
\begin{lem}\label{l:from-stoll}
Let $X,K$ be as above and assume that $ X(\A)^{fin} \neq \emptyset $.
Let
$$ f:Y \lrar X $$
be a torsor under a finite $K$-group $G$. Then there is a twist
$$ f^\alp:Y^{\alp} \lrar X $$
such that $\ovl{Y}$ has a connected component defined over $K$. If one assumes instead that $X(\A)^{fin-ab} \neq \emptyset$ then the result is true for abelian $G$.
\end{lem}

\begin{proof}
Note that over $\ovl{K}$ there is an inclusion
$$(\ovl{Y}_0,\ovl{G_0}) \lrar (\ovl{Y},\ovl{G})$$
which is a map of torsors over $\ovl{X}$, where $\ovl{Y}_0$ is any connected component of $\ovl{Y}$. Since $X(\A)^{fin} \neq \emptyset$ we have from Lemma 5.7 of ~\cite{Sto07} that there exists a map
$$ (\ovl{Y'},\ovl{G'}) \lrar (\ovl{Y}_0,\ovl{G_0}) $$
where $(Y',G')$ is a geometrically connected torsor over $X$. By composing we get a map
$$ (\ovl{Y'},\ovl{G'}) \lrar (\ovl{Y},\ovl{G}) $$
Now by Lemma 5.6 ~\cite{Sto07} there exists an $\alp \in H^1(K,G)$ for which there is a map
$$ ({Y'},{G'}) \lrar ({Y^{\alp}},{G^{\alp}}) $$
Since $(Y',G')$ is geometrically connected and defined over $K$ the same is true for its image. The abelian version is similar.
\end{proof}

Now assume that $(x_\nu) \in X(\A)^{fin}$ and let $f: Y \lrar X$ be a torsor under a finite $K$-group $G$. Since $X(\A)^{fin} \neq \emptyset$ we get from Lemma ~\ref{l:from-stoll} that there exists a twist
$$ f^{\alp}: Y^{\alp} \lrar X $$
such that $\ovl{Y^\alp}$ has a $\Gam_K$-invariant connected component. Let $Y^{\alp}_0 \subseteq Y^{\alp}$ be that connected component and consider it as torsor over $X$ under its stablizer $G^\alp_0 \subseteq G^\alp$. Since $Y^{\alp}_0$ is geometrically connected the map
$$ \SimpS_{Y^{\alp}_0} \lrar B(G^{\alp}_0) $$
is a weak equivalence. From Corollary ~\ref{c:cor-of-classifies-fiber} we get that $h_{\check{Y}^{\alp}_0}((x_\nu))$ is rational and since $Y^{\alp}_0$ maps to $Y^{\alp}$ we get that $h_{\check{Y}^{\alp}}((x_\nu))$ is rational as well.

Let $L/K$ a finite Galois extension such that $\alp$ can be written as a map $\alp: \Gal(L/K) \lrar G(L)$. Let $f^{\alp}:Y^{\alp} \lrar X$ be the corresponding twist by $\alp$ which is a finite torsor under $G^{\alp}$. Let
$$ X_L \lrar X $$
be as in Definition ~\ref{d:X-L}. Note that this is a $\Gal(L/K)$-torsor.

Now consider the torsors (over $X$)
$$ Y_L = Y\times_X X_L $$
$$ Y^{\alp}_L = Y^{\alp}\times_X X_L $$
and the natural projections
$$ Y_L \lrar Y $$
$$ Y^{\alp}_L \lrar Y^{\alp} $$
Note that these maps induce $\ovl{K}$-isomorphism from each $\ovl{K}$-connected component in their domain to its image. Hence from our groupoid analysis above we see that they induce weak equivalences
$$ \SimpS_{\check{Y}_L} \lrar \SimpS_{\check{Y}} $$
$$ \SimpS_{\check{Y}^{\alp}_L} \lrar \SimpS_{\check{Y}^{\alp}} $$

\begin{lem}\label{l:twist-relations-1}
There exist an isomorphism of \'{e}tale coverings over $X$
$$ T_\alp: Y_L \lrar Y^{\alp}_L $$
\end{lem}
\begin{proof}
After choosing an identification of $\pi_0(\bar{X_L})$ with $\Gal(L/K)$ it is enough to construct an isomorphism
$$ T_\alp: \Gal(L/K) \times \bar{Y} \lrar \Gal(L/K) \times \bar{Y} = \Gal(L/K) \times \bar{Y^{\alp}}$$
which commute with the action of $\Gam_K$ (which is different between both sides).

We shall define
$$ T_\alp(\tau,y) = (\tau,\alp(\tau)^{-1}y) $$
Now note that indeed
$${}^\sig (T_\alp(\tau,y)) = {}^\sig(\tau,\alp(\tau)^{-1}y)= (\sig\tau , {}^\sig(\alp(\tau)^{-1}y)) =$$
$$ (\sig\tau , \alp(\sig)^{-1}{}^\sig\alp(\tau)^{-1}{}^\sig y) =  (\sig\tau , \alp(\sig\tau)^{-1}{}^\sig y) = $$
$$ T_\alp(\sig\tau,{}^\sig y) = T_\alp({}^\sig(\tau,y)) $$
This finishes the proof of the lemma.
\end{proof}

Now from Lemma ~\ref{l:twist-relations-1} it follows that we have an isomorphism
$$ \SimpS_{\check{Y}_L} \lrar \SimpS_{\check{Y}^{\alp}_L} $$

This means that $h_{\check{Y}^{\alp}}((x_\nu))$ is rational if and only if $h_{\check{Y}}((x_\nu))$ is rational and so $(x_\nu) \in X(\A)^{h-tor}$ and we are done.

In the abelian case one should use the same arguments but apply $P^1(\ZZ \bullet )^1$ to the spaces $\SimpS_{\check{Y}}$,$\SimpS_{\check{Y}^{\alp}}$, $\SimpS_{\check{Y}_L}$ and $\SimpS_{\check{Y}^{\alp}_L}$ and again use the fact that for an abelian group $G$ the map
$$\B G \lrar P_1(\ZZ \B G)^1$$
is a weak equivalence.

\begin{prop}~\label{p:h-1-is-check}
Let $K$ and $X$ be as above. Then
$$ X(\A)^{h,1} = X(\A)^{h-tor} $$
$$ X(\A)^{\ZZ h,1} = X(\A)^{\ZZ h-tor} $$
\end{prop}
\begin{proof}
Again we shall state the proof only for the non-abelian case the modification needed to the abelian case is just applying the functor $\ZZ$ to all the $\Gam_K$-Spaces in the proof.

First suppose that $(x_\nu) \in X(\A)^{h,1}$. Then for every torsor $f: Y \lrar X$ under a finite $K$-group $G$ the adelic homotopy fixed point
$$ h_{\check{Y},1}((x_\nu)) \in \SimpS_{\check{Y},1}(h\A) $$
is rational.

We know that $\SimpS_{\check{Y}}$ is the nerve of a groupoid $\mcal{Y}$. Now it is fairly straight forward to verify that if $\mcal{G}$ is a groupoid then the nerve $N(\mcal{G})_\bullet$ satisfies the Kan condition and the map
$$ N(\mcal{G})_\bullet \lrar P_1(N(\mcal{G})_\bullet) $$
is an isomorphism. Hence in our case
$$\SimpS_{\check{Y}} \lrar  \SimpS_{\check{Y},1}$$
is an isomorphism. This means that if $(x_\nu) \in X(\A)^{h,1}$ then
$$ h_{\check{Y}}((x_\nu)) \in \SimpS_{\check{Y}}(h\A) $$
is rational for every torsor $f: Y \lrar X$ under a finite $K$-group $G$ and so
$$ (x_\nu) \in X(\A)^{h-tor} $$

Now assume that $ (x_\nu) \in X(\A)^{h-tor} $. We want to show that $(x_\nu) \in X(\A)^{h,1}$. First of all we have that
$$ h_{\check{Y},1}((x_\nu)) \in \SimpS_{\check{Y},1}(h\A) $$
is rational for all such $f: Y \lrar X$. Now the following two lemmas will finish the proof:

\begin{lem}\label{l:Upsilon}
Let $K$ be a field and $X$ a geometrically connected $K$-variety. Let $\U_\bullet \lrar X$ be a hypercovering and $G = \pi_1(\SimpS_{\U})$. Then there exist a $G$-torsor
$$ f:\Upsilon \lrar \ovl{X} $$
defined over $\ovl{K}$ and a map
$$ \gamma :\U_0 \lrar \Upsilon $$
of \'{e}tale covering over $\ovl{X}$ such that the induced map
$$ \SimpS_{\U,1} \lrar \SimpS_{\check{\Upsilon},1} $$
is a weak equivalence.
\end{lem}
\begin{proof}
By abuse of notation we shall refer to the pullback of $\U$ to the generic point of $X$ by $\U$ as well. Now by the discussion in subsection \S\S ~ref{ss:homotopy-type-of-X-U} $\U_0$ can be considered just as a finite $\Gam_{K(X)}$-set.

We denote by
$$ \Stab_{\Gam_{K(X)}} (\U_0) = \{H_1,...,H_t\} $$
the set of stabilizers in  $\Gam_{K(X)}$ of the points of $\U_0$. Similarly we denote by
$$ \Stab_{\Gam_{\ovl{K}(\ovl{X})}} (\U_0) = \{H_1 \cap {\Gam_{\ovl{K}(\ovl{X})}},...,H_t\cap {\Gam_{\ovl{K}(\ovl{X})}} \} $$
the set of stabilizers in  ${\Gam_{\ovl{K}(\ovl{X})}}$ of the points of $\U_0$.

Now consider the group
$$ E = \left<H_1 \cap {\Gam_{\ovl{K}(\ovl{X})}},...,H_t\cap {\Gam_{\ovl{K}(\ovl{X})}}\right> $$
i.e $E$ is the group generated by all the groups in $\Stab_{\Gam_{\ovl{K}(\ovl{X})}}(\U_0)$. Since the set $\Stab_{\Gam_{K(X)}}(\U_0)$ is invariant under conjugation by elements from $\Gam_{K(X)}$ we have that $E \lhd \Gam_{K(X)}$.

$E$ is also contained in $\Gam_{\ovl{K}(\ovl{X})}$ and is of finite index there. Let $F/\ovl{K}(\ovl{X})$ be the finite field extension that corresponds to $E$.

We claim that the ramification locus of $F$ in $X$ is empty. This is because every divisor $D \subseteq X$ intersects the image of some element in $\U_0$ which has some stabilizer $G_i \in \Stab_{\Gam_{\ovl{K}(\ovl{X})}} (\U_0)$. Then the field extension corresponding to $G_i$ is unramified at $D$ and since $G_i \subseteq E$ we get that $F$ is unramified at $D$ as well.

Thus $E$ contains the kernel of the map
$$ \rho:\Gam_{\ovl{K}(\ovl{X})} \lrar \pi_1(\ovl{X}) $$
Then $\rho(E)$ corresponds to a finite Galois cover
$$ \phi: \Upsilon \lrar \ovl{X} $$
with $\ovl{K}(\Upsilon) = F$. Further, by considering $\ovl{\U_0}$ and $\Upsilon$ as $\Gam_{\ovl{K}(\ovl{K})}$-sets, we see that there is  map
$$\gamma: \ovl{\U_0} \lrar \Upsilon$$
over $\ovl{X}$. Now by using Lemma ~\ref{l:calc_pi1} we see that the induced map
$$\SimpS_{\ovl{\U_0}} \lrar \SimpS_{\check{\Upsilon}} $$
induces an isomorphism on the fundamental groups and so a weak equivalence
$$ \SimpS_{\ovl{\U_0},1} \lrar \SimpS_{\check{\Upsilon},1} $$

\end{proof}

\begin{lem}\label{l:check-is-enough}
Let $K$ and $X$ be as above. Let $\U_\bullet \lrar X$ be a hypercovering. Then there exists a torsor $f:Y \lrar X$ under a finite $K$-group and a map
$$ g:\U_0 \lrar Y $$
of \'{e}tale covering over $X$ such that the induced map
$$ \SimpS_{\U,1} \lrar \SimpS_{\check{Y},1} $$
is a weak equivalence.
\end{lem}
\begin{proof}
By Lemma ~\ref{l:Upsilon} we have a torsor under some finite group
$$ \phi: \Upsilon \lrar \ovl{X} $$
defined over $\ovl{K}$ and map
$$\gamma: \ovl{\U_0} \lrar \Upsilon $$
which induces a weak equivalence
$$ \SimpS_{\U,1} \lrar \SimpS_{\check{\Upsilon},1} $$

Now let
$$\gamma_n: \ovl{\U_0} \lrar \Upsilon_n$$
be a complete set of twists of
$$\gamma: \ovl{\U_0} \lrar \Upsilon $$
Then we have a map

$$ \prod \gamma: \ovl{\U_0} \lrar \Upsilon_1 \times_X \cdots \times_X \Upsilon_nõ$$
We shall choose
$$ Y = \Upsilon_1 \times_X \cdots \times_X \Upsilon_n$$
and
$$ g = \prod \gamma $$

Now by using Lemma ~\ref{l:calc_pi1} we see that the induced map
$$ \SimpS_{\U} \lrar \SimpS_{\check{Y}} $$
induces an isomorphism on $\pi_1$ and so the induced map
$$ \SimpS_{\U,1} \lrar \SimpS_{\check{Y},1} $$
is a weak equivalence.
\end{proof}

\end{proof}
\end{proof}
\end{proof}

\begin{rem}\label{r:non-connected-fin}
One might wonder what is the role of connectivity in the proof of Theorem ~\ref{t:fin}. To see this one should consider the case
$$ X = \spec(K) \coprod \spec(K) $$
By Proposition ~\ref{p:zero-dim} we know that
$$ X(\A)^{\ZZ h,1} = X(\A)^{h,1} =  X(K) $$
On the other hand when $K$ has a complex place then
$$ X(\A)^{fin-ab} = X(\A)^{fin} \neq  X(K) $$
Indeed if $(x_\nu),(x'_\nu) \in X(\A)$ are two adelic points that differ only in a complex coordinates then $(x_\nu) \in  X(\A)^{fin}$ if and only if  $(x'_\nu) \in X(\A)^{fin}$, and similarly for $X(\A)^{fin-ab}$.

This is, however, the only fault in the desired equality of $X(\A)^{fin}$ and $X(K)$. In fact, in  ~\cite{Sto07} Stoll gives a proof of this equality (Lemma 5.10 in ~\cite{Sto07}). However this proof is erroneous and the error becomes a problem exactly in the complex places.

To conclude, when we ignoring the complex places (or if $K$ is totally real) one can write
$$ X(\A)^{h,1} = X(\A)^{fin} $$
and
$$ X(\A)^{\ZZ h,1} = X(\A)^{fin-ab} $$
even when $X$ is not connected.
\end{rem}

The following corollary is originally due to  Harari and Stix  (unpublished work) and is closely related to Lemma $5.7$ in ~\cite{Sto07}:

\begin{cor}\label{c:Harari-Stix}
Let $K$ and $X$ a smooth geometrically connected variety over $K$. Then if
$$ X(\A)^{fin} \neq \emptyset $$
then the short exact sequence
$$ 1 \lrar \pi_1(\ovl{X})\lrar \pi_1(X)\lrar \Gam_K \lrar 1 $$
splits.
\end{cor}

\begin{proof}

Since
$$ X(\A)^{h,1} = X(\A)^{fin} \neq \emptyset $$
we know that ${\acute{E}t_{/K}}^1(X)(hK)$ is non-empty thus we get the claim from Lemma ~\ref{l:grothendieck-h-1}.

\end{proof}

\section{ The Equivalence of the Homology Obstruction and the Brauer-Manin Obstruction }\label{s:equivalence-1}
In this section we shall prove to following theorem:
\begin{thm}\label{t:main}
Let $X$ be a smooth algebraic variety over a number field $K$. Then
$$ X(\A)^{\ZZ h} \subseteq X(\A)^{\Br} $$
If $X$ is also connected then
$$ X(\A)^{\ZZ h} = X(\A)^{\Br} $$
\end{thm}

In light of Corollary ~\ref{c:descent-theorem} we see that in order to understand
$$ X(\A)^{\ZZ h}$$
we first need to understand the sets $P_n(\ZZ \SimpS_\U)(h\A)$ and $P_n(\ZZ \SimpS_\U)(hK)$ for each hypercovering $\U$. Note that $P_n(\ZZ \SimpS_\U)$ is not only a bounded simplicial $\Gam_K$-set but also a bounded simplicial $\Gam_K$-module i.e. a simplicial object in the category $\Mod_{\Gam_K}$ of $\Gam_K$ modules.

The category $\Mod_{\Gam_K}$ is abelian. An important tool of homotopy theory is the Kan-Dold correspondence which allows one to reduce the study of the homotopy theory of simplical objects in an abelian category $\mcal{A}$ to the study of homological algebra of complexes in $\mcal{A}$.

\subsection{The Kan-Dold Correspondence }
\begin{define}
Let $\mcal{A}$ be an abelian category. We denote by $\mcal{CA}$ the category of complexes over $A$ with a differential of degree $-1$. We denote by $\mcal{C}^{\geq 0}\mcal{A}$ the category of complexes which are bounded from below by dimension $0$.
\end{define}

\begin{define}
Let $\mcal{A}$ be an abelian category. If $\X$ is a simplicial object in $\mcal{A}$ we denote by $\uline{\X} \in \mcal{C}^{\geq 0}\mcal{A}$ the complex over $\mcal{A}$ given by
$$ \uline{\X}_n = \X_n $$
with the differential $d = \sum_i (-1)^i d_i$.
\end{define}
\begin{prop}(The Kan-Dold Correspondence)\label{p:Kan-dold}
Let $\mcal{A}$ be an abelian category. The category $\mcal{A}^{\Del^{op}}$ of simplicial objects in $\mcal{A}$ is equivalent to the category $\mcal{C}^{\geq 0}\mcal{A}$. Moreover for every two simplicial $\mcal{A}$ objects one has
$$ [\X,\Y] \cong [\uline{\X},\uline{\Y}] $$
where the first denotes simplicial homotopy and the second chain homotopy.
\end{prop}
\begin{proof}
For the Moore complex it is written in Dold's original paper ~\cite{Dol58} and the unnormalized complex is homotopy equivalent to the Moore complex.
\end{proof}

\begin{rem}
The functor $\uline{\bullet}$ admits a right adjoint
$$ \ovl{\bullet}:  \mcal{C}^{\geq 0}\mcal{A} \lrar \mcal{A}^{\Del^{op}} $$
which is its "homotopy inverse"  i.e.
$$ \X \x{\simeq}{\lrar} \ovl{(\uline{\X})} \quad \forall \X \in \mcal{A}^{\Del^{op}} $$
$$ \uline{(\overline{\Y})} \x{\simeq}{\lrar} \Y \quad \forall \Y \in  \mcal{C}^{\geq 0}\mcal{A} $$

In particular for every $\X \in \mcal{CA}, \Y \in \mcal{A}^{\Del^{op}}$ one has an isomorphism of sets
$$ [\X,\ovl{\Y}] \cong [\uline{\X},\Y] $$
where the first denotes simplicial homotopy and the second chain homotopy.

\end{rem}

\begin{rem}
In case $\mcal{A}$ is the category of abelian groups then the Kan-Dold correspondence replaces homotopy groups with homology groups, i.e
$$ H_n(\uline{\X}) \cong \pi_n(\X) \quad \forall \X \in \mcal{A}^{\Del^{op}}$$
\end{rem}

Proposition ~\ref{p:Kan-dold} allows us to work in the category of complexes $\mcal{CA}$ instead in the category of simplicial objects and thus do homological algebra instead of homotopy theory.

\subsection{ Postnikov Towers For Complexes }

In order to continue with this line we would like to have for complexes a construction similar to the Postnikov construction which allows us to truncate homology groups. This analogue is actually quite simpler then the homotopical notion and one can choose to truncate homology groups from below as well as above.

Let $C \in \mcal{CA}$ be a complex. We denote by $P^+_n(C)$ the complex
$$ ... \lrar 0 \lrar C_n/d(C_{n+1}) \lrar C_{n-1} \lrar C_{n-2} \lrar ... $$
Note that $H_i(P^+_i(C)) = 0$ for $i > n$ and there is a natural map $C \lrar P^+_n(C)$ which induces isomorphisms
$$ H_i(C) \x{\simeq}{\lrar} H_i(P^+_n(C)) $$
for $i \leq n$. The functor $C \mapsto P^+_n(C)$ from general complexes to complexes bounded from above by dimension $n$ is left adjoint to the inclusion functor.

We similarly define $P^-_n(C)$ to be the complex
$$ ... \lrar C_{n+2} \lrar C_{n+1} \lrar \Ker d \lrar 0 \lrar ... $$
Note that $H_i(P^-_i(C)) = 0$ for $i < n$ and there is a natural map $P^-_n(C) \lrar C$ which induces isomorphisms
$$ H_i(P^-_n(C)) \x{\simeq}{\lrar} H_i(C) $$
for $i \geq n$. The functor $C \mapsto P^-_n(C)$ from general complexes to complexes bounded from below by dimension $n$ is right adjoint to the inclusion functor.

\begin{define}\label{df:typeof_complx}
Analogously to Definition ~\ref{df:typeof_spaces} we shall define

\begin{enumerate}
\item
A complex $C \in \mcal{C}{\Mod_{\Gam}}$ will be called \textbf{finite} if $H_i(C)$ if finite for all $i \in \ZZ$.
\item
A complex $C \in \mcal{C}{\Mod_{\Gam}}$ will be called \textbf{bounded} if there exists an $N >0$ such that
$$ C_i = 0 $$
for $|i| > N$.

\item
A complex $C \in \mcal{C}{\Mod_{\Gam}}$ will be called \textbf{homologically bounded} if there exists an $N >0$ such that
$$ H_i(C) = 0 $$
for $|i| > N$.
\item
A complex $C \in \mcal{C}{\Mod_{\Gam}}$ will be called \textbf{excellent} if the action of $\Gam$ on $C$ factors through a finite quotient.
\item
A complex $C \in \mcal{C}{\Mod_{\Gam}}$ will be called \textbf{nice} if for every $i \in \ZZ$ the action of $\Gam$ on $C_i$ factors through a finite quotient.
\end{enumerate}
\end{define}

The following lemma will be useful later:
\begin{lem}\label{l:assume_bounded}
Let $C \in \mcal{C}{\Mod_{\Gam}}$ be a homologically bounded complex. Then there exists a bounded complex $C^b$ such that $C^b$ is quasi-isomorphic to $C$. Furthermore, if $C$ is nice then $C^b$ can be chosen to be excellent.

\end{lem}
\begin{proof}
Let $N$ be such that $H_i(C)= 0$ for $|i|>N$. Now it is clear that
$$ C \lrar P^+_N(C) \llar P^-_{-N}(P^+_N(C)) $$
is a zigzag of quasi-isomorphisms and that $P^-_{-N}(P^+_N(C)) \in \mcal{C}^b{\Mod_{\Gam}}$. Thus we can take
$$ C^b = P^-_{-N}(P^+_N(C)) $$
In particular if $C$ is nice then $C^b$ is excellent.
\end{proof}

Now suppose we are given a general complex $C \in \mcal{CA}$ and we want to construct from it a simplicial object in $\mcal{A}$. What we can do is to truncate it at zero to get a complex
$$ P^-_0(C) \in \mcal{C}^{\geq 0}\mcal{A} $$
We can then associate with it a simlicial object $\ovl{P^-_0(C)}$ in $\mcal{A}$.

The following lemma shows that this operation has some merit:
\begin{lem}\label{l:ses-gives-fibration}
Let $\mcal{A}$ be the category of $\Gam$-modules. Let
$$ 0 \lrar A \lrar B \lrar C \lrar 0 $$
be a short exact sequence of complexes in $\mcal{CA}$. Then
$$ \ovl{P^-_0(A)} \lrar \ovl{P^-_0(B)} \lrar \ovl{P^-_0(C)} $$
is a fibration sequence of simplicial abelian groups.
\end{lem}
\begin{proof}
Observe that that
$$ 0 \lrar P^-_0(A)_n \lrar P^-_0(B)_n \lrar P^-_0(C)_n \lrar 0 $$
is exact for $n > 0$ and that
$$ 0 \lrar P^-_0(A)_0 \lrar P^-_0(B)_0 \lrar P^-_0(C)_0 $$
is exact. This means that we have a long exact sequence
$$ ... \lrar H_n(P^-_0(A)) \lrar H_n(P^-_0(B)) \lrar H_n(P^-_0(C)) \lrar ... $$
$$ \lrar H_0(P^-_0(A)) \lrar H_0(P^-_0(B)) \lrar H_0(P^-_0(C)) $$
and so the natural map from $\ovl{P^-_0(A)}$ to the homotopy fiber of the map $\ovl{P^-_0(B)} \lrar \ovl{P^-_0(C)}$ is a homotopy equivalence.
\end{proof}

\subsection{ Galois Hypercohomology }\label{sbs:homology-pro}
Recall the suspension functor $\Sig: \mcal{C}{\Mod_{\Gam}} \lrar \mcal{C}{\Mod_{\Gam}}$ defined by
$$ (\Sig C)_n = C_{n-1} $$
The Kan-Dold correspondence allows one to use homotopy theory to \textbf{define} the hypercohomology $\HH^n\left(\Gam,C\right)$ for any $n \in \ZZ$ and a completely general complex $C\in \mcal{C}{\Mod_{\Gam}}$ by the formula
$$ \HH^n\left(\Gam, C\right)  = \pi_{0}\left(\ovl{P^{-}_0(\Sigma^{n} C)}^{h\Gam}\right)$$
When $\Gam = \Gam_K$ is the absolute Galois group of a field $K$ we will denote this group also by
$$ \HH^n\left(K, C\right) \x{def}{=} \HH^n\left(\Gam_K, C\right) $$
It is easy to see that this definition is invariant to quasi-isomorphisms and satisfies
$$ \HH^n\left(\Gam,\Sig C\right) \cong \HH^{n+1}\left(\Gam,C\right) $$
It also transforms short exact sequences to long exact sequences:
\begin{lem}\label{l:rational-hyper-ses}
Let
$$ 0 \lrar A \x{i}{\lrar} B \x{p}{\lrar} C \lrar 0 $$
be a short exact sequence of complexes. Then there is a natural long exact sequence

$$... \lrar  \HH^{n-1}(\Gam,C) \lrar \HH^n(\Gam,A) \lrar \HH^n(\Gam,B) \lrar \HH^n(\Gam,C) \lrar \HH^{n+1}(\Gam,A) \lrar ...$$
\end{lem}
\begin{proof}
By Lemma ~\ref{l:ses-gives-fibration} we get for each $n$ a fibration sequence
$$\ovl{P^-_0(\Sig^n A)} \lrar \ovl{P^-_0(\Sig^n B)} \lrar \ovl{P^-_0(\Sig^n C)} $$
is a fibration sequence. By lemma ~\ref{l:preserves fibrations} the sequence
$$ 0 \lrar \overline{P^-_0(\Sig^n A)}^{h\Gam} \lrar \overline{P^-_0(\Sig^n B)}^{h\Gam} \lrar \overline{P^-0(\Sig^n C)}^{h\Gam} \lrar 0 $$
is a fibration sequence as well, which gives in particular an exact sequence
$$ \pi_1\left(\overline{P^-_0(\Sig^n A)}^{h\Gam}\right) \lrar
\pi_1\left(\overline{P^-_0(\Sig^n B)}^{h\Gam}\right) \lrar
\pi_1\left(\overline{P^-_0(\Sig^n C)}^{h\Gam}\right) \lrar $$
$$ \pi_0\left(\overline{P^-_0(\Sig^n A)}^{h\Gam}\right) \lrar \pi_0\left(\overline{P^-_0(\Sig^n B)}^{h\Gam}\right) $$

Note that for any complex $D$ we have
$$ \Om \overline{P^-_0(D)} = \overline{P^-_0(\Sig^{-1} D)} $$
Since $\Om(X^{h\Gam}) = (\Om X)^{h\Gam}$ by commutation of homotopy limits the exact sequence above becomes
$$ \pi_0\left(\overline{P^-_0(\Sig^{n-1} A)}^{h\Gam}\right) \lrar
\pi_0\left(\overline{P^-_0(\Sig^{n-1} B)}^{h\Gam}\right) \lrar
\pi_0\left(\overline{P^-0(\Sig^{n-1} C)}^{h\Gam}\right) \lrar $$
$$ \pi_0\left(\overline{P^-_0(\Sig^{n} A)}^{h\Gam}\right) \lrar \pi_0\left(\overline{P^-_0(\Sig^n B)}^{h\Gam}\right) $$
which by definition gives the exact sequence
$$ H^{n-1}(\Gam,A) \x{i_*}{\lrar} H^{n-1}(\Gam,B) \x{p_*}{\lrar} H^{n-1}(\Gam,C) \lrar H^{n}(\Gam,A) \x{i_*}{\lrar} H^{n}(\Gam,B) $$
Putting all this exact sequences together gives the desired long exact sequence.

\end{proof}

In this paper we will sometimes need to use a more concrete formula for the hypercohomology. This is enabled by Theorem ~\ref{t:goerss-formula} for complexes which are bounded from above:
\begin{lem}\label{l:hyper-formula}
Let $C \in \mcal{C}\Mod_{\Gam}$ be a complex bounded from above (i.e. $C_n = 0$ for large enough $n$). Then
$$ \HH^i(\Gam,C) = \lim \limits_{\x{\lrar}{\Lam \fns \Gam}}\left[\uline{\ZZ \E(\Gam/\Lam)},C\right]^i_{\Gam} $$
In particular if $K$ is field and $\Gam = \Gam_K$ then we have
$$ \HH^i(K,C) = \lim \limits_{\x{\lrar}{L/K}}\left[\uline{\ZZ \E G_L},C\right]^i_{\Gam_K} $$
where the colimit is taken over finite Galois extension $L/K$ and $G_L$ is the relative Galois group of $L$ over $K$.
\end{lem}
\begin{proof}
Let $C$ be a complex bounded from above. Let $\SimpS$  be the underlying simplicial $\Gam_K$-set of the simplicial $\Gam_K$-module
$ \SimpS = \ovl{P^{-}_0(\Sig^i C)} $.
Note that $\SimpS$ is strictly bounded so we can use ~\ref{t:goerss-formula} in order to compute:
$$
\HH^i(K,C) = \pi_0\left(\SimpS^{h\Gam}\right) =
\lim \limits_{\x{\lrar}{\Lam \fns \Gam}} \pi_0\left(\Map_{\Gam}\left(\E(\Gam/\Lam),\SimpS\right)\right) = $$
$$ \lim \limits_{\x{\lrar}{\Lam \fns \Gam}}
\left[\uline{\ZZ \E(\Gam/\Lam)},P^{-}_0(\Sig^i C)\right]^0_{\Gam} =
\left[\uline{\ZZ \E(\Gam/\Lam)},\Sig^i C\right]^0_{\Gam} =
 \lim \limits_{\x{\lrar}{\Lam \fns \Gam}}
\left[\uline{\ZZ \E(\Gam/\Lam)},C\right]^i_{\Gam} $$
\end{proof}

\begin{rem}
In ~\cite{Spa88} Spaltenstein gives a general descriptions of Galois hypercohomology for unbounded complexes by showing that every such complex admits an appropriate notion of injective resolution (which agrees with the classical notion for complexes bounded from above). His definition is equivalent to ours but verifying this might be quite a lengthy computation. The worried reader can note that if $C$ is just a Galois module (i.e. a complex consecrated of degree $0$) this fact follows from the computation of Lemma ~\ref{l:hyper-formula}. From this one can extend the claim by induction to general bounded complexes by using the suspension isomorphism and Lemma ~\ref{l:rational-hyper-ses}. From lemma ~\ref{l:assume_bounded} this will be true for all homologically bounded complexes.

The only place where we will need the coincidence of our definition with the classical one is when we will want to use freely the classical notion of cup product in hypercohomology. We will only use this notion for homologically bounded complexes.

\end{rem}

Now let $K$ be a field and $\mathbf{A}$ a simplicial $\Gam_K$-module. Then $\underline{\mathbf{A}}$ is bounded below by dimension $0$ and so
$$ P^-_0(\underline{\mathbf{A}}) = \underline{\mathbf{A}} $$
which means that
$$ \mathbf{A}(hK) = \pi_0\left(\mathbf{A}^{h\Gam}\right) = \HH^0(\Gam,\underline{\mathbf{A}}) $$


\subsection{ Adelic Hypercohomology }\label{ss:adelic-hypercohomology}


Let $K$ be a number field and $S$ a set of places of $K$. For a complex $C \in \mcal{C}\Mod_{\Gam_K}$ we define the \textbf{n'th $S$-adelic hypercohomology of $C$} by

$$ \HH^n\left(\A_S, C\right)  = \pi_{0}\left(\overline{P^{-}_0(\Sigma^{n} C)}^{h\A_S}\right) $$
\begin{rem}\label{r:adelic-invariant}
Note that by Theorem ~\ref{t:A-weak} the adelic hypercohomology groups are preserved
by quasi-isomorphism between \textbf{nice} complexes.
\end{rem}

\begin{lem}\label{l:adelic-hyper-ses}
Let
$$ 0 \lrar A \lrar B \lrar C \lrar 0 $$
be a short exact sequence of nice complexes in $\mcal{C}\Mod_{\Gam_K}$. Then there is a natural long exact sequence

$$... \lrar  \HH^{i-1}(\A_S,C) \lrar \HH^i(\A_S,A) \lrar \HH^i(\A_S,B) \lrar \HH^i(\A_S,C) \lrar \HH^{i+1}(\A_S,A) \lrar ...$$
\end{lem}
\begin{proof}
The proof is the same as lemma ~\ref{l:rational-hyper-ses} with the usage of Corollary ~\ref{c:preserves fibrations} instead of Lemma ~\ref{l:preserves fibrations}.
\end{proof}

Now let $K$ be a number field and $\mathbf{A}$ a simplicial $\Gam_K$-module. Then since $P^-_0(\underline{\mathbf{A}}) = \underline{\mathbf{A}}$ we have by definition
$$ \mathbf{A}(h\A) = \HH^0\left(\A, \uline{\mathbf{A}}\right) $$

\subsection{Proof of the Main Theorem }

We will now give an outline of the proof of Theorem ~\ref{t:main}. Given an adelic point $(x_\nu) \in X(\A)$ we get that it is homologically rational if and only if for every hypercovering $\mcal{U} \lrar X$ and $n\in N$ its image in
$$ P_n(\ZZ\SimpS_{\mcal{U}})(h\A) = \HH^0(\A,\uline{P_n(\ZZ\SimpS_{\mcal{U}})}) $$
is rational. i.e. lies in the image of the map
$$ \loc:\HH^0(K,\uline{P_n(\ZZ\SimpS_{\mcal{U}})})\lrar \HH^0(\A,\uline{P_n(\ZZ\SimpS_{\mcal{U}})}) $$

On the other hand the Brauer set is defined via pairings with elements in $H_{\acute{e}t}^2(X,\GG_m)$. Now the idea is to show that this pairing actually factors through the map
$$ X(\A) \lrar \lim \limits_{\stackrel{\llar}{\mcal{U},n}}\HH^0(\A,\uline{P_n(\ZZ\SimpS_{\mcal{U}})}) $$

First we will need some terminology:

Let $X/K$ an algebraic variety over $K$ with $t:X\to \spec K$ being the structure map. Given a Galois module $A$ we can consider it as an \'{e}tale sheaf over $\spec(K)$.

We denote by $t^{*}A$ be its pull back to $X$. The sheaf $t^*A$ can be described more concretely as follows: it associates to an \'{e}tale map $V \lrar X$ the group of $\Gam_K$ equivariant maps from the $\Gam_K$-set $\ovl{\pi}_0(V)$ to $A$. We refer to $t^*A$ as the sheaf of locally constant maps to $A$.

We will need the following two definitions:
\begin{define}
Let $C,D \in  \mcal{C}{\Mod_{\Gam}}$ be complexes. We denote by  $\uline{\Hom}(C,D) \in \mcal{C}{\Mod_{\Gam}}$ to be the complex of maps between $C$ and $D$ as complexes of abelian groups (without respecting the $\Gam$-action or the differential). We will refer to it as the mapping complex from $C$ to $D$.
\end{define}

\begin{define}
Let $K$ be a field and $\Gam = \Gam_K$. Consider $\GG_m$ as a complex concentrated in degree zero such that ${\GG_m}_0 = \GG_m(\ovl{K}) = \ovl{K}^*$. We denote by $\widehat{C}$ the mapping complex $\uline{\Hom}(C,\GG_m)$ and refer to it as the \textbf{dual complex}.
\end{define}

We now return to the outline of the proof. For every hypercovering $\mcal{U} \lrar X $ and a natural number $n\geq 0$ we will construct a commutative diagram:
$$ \xymatrix{
X(\A) \ar@{=}[d] & \times & H^2_{\acute{e}t}(X,{\GG_m}) \ar[r] & \QQ / \ZZ \ar@{=}[d]\\
X(\A) \ar[d]^{h} & \times & H^2_{\acute{e}t}(X,t^*{\GG_m})\ar@{->>}[u]^{i_*} \ar[r] & \QQ / \ZZ \ar@{=}[d]\\
\HH^0(\A,\uline{P_n(\ZZ\SimpS_{\mcal{U}})}) & \times &
\HH^2\left(K,\widehat{\uline{P_n(\ZZ \SimpS_{\mcal{U}})}}\right) \ar^{\psi_{\mcal{U},n}}[u]\ar[r]\ar[r] & \QQ / \ZZ \\
}$$
where the map $i_*$ is induced by the natural inclusion of sheafs.

We will then prove the following:
\begin{enumerate}
\item
$i_*$ is a \textbf{surjective}.
\item
The maps $\Psi_{\mcal{U},n}$ induce together a single \textbf{isomorphism}:
$$ \lim \limits_{\stackrel{\lrar}{\mcal{U},n}} \HH^2\left(\Gam_K,\widehat{\uline{P_n(\ZZ \SimpS_{\mcal{U}})}}\right) \lrar H^2_{\acute{e}t}(X,t^*{\GG_m}) $$
\end{enumerate}

These two facts together will imply that an adelic point is in the Brauer set if and only if its image in $\HH^0(\A,\uline{P_n(\ZZ\SimpS_{\mcal{U}})})$ is orthogonal to all the elements in
$\HH^2\left(\Gam_K,\widehat{\uline{P_n(\ZZ \SimpS_{\mcal{U}})}}\right)$.

From the Hasse-Brauer-Noether theorem in class field theory it follows that if an element in
$\HH^0(\A,\uline{P_n(\ZZ \SimpS_{\mcal{U}})})$ is rational then it is orthogonal to every element in
$\HH^2(K,\what{\uline{P_n(\ZZ \SimpS_{\mcal{U}})}})$ .

This gives us the first part of the theorem, i.e. that
$$ X(\A)^{\ZZ h}\subseteq X(\A)^{\Br} $$

Now assume that our variety $X$ is connected. Then $C = \uline{P_n(\ZZ \SimpS_{\U})}$ satisfies the following properties
\begin{enumerate}
\item
It is excellent.
\item
It is homologically bounded.
\item
It is bounded below by dimension $0$.
\item
$ H_0(C) = \ZZ $
\item
$H_i(C)$ is finite for $i > 0$.
\end{enumerate}

From property $3$ we have a natural map $C \lrar \ZZ$ which induces a map
$$ \pi: \HH^0(\A,C) \lrar \HH^0(\A,\ZZ)  $$

Then for every $(x_\nu) \in X(\A)$ we have that
$$ h_{\ZZ \U,n}((x_\nu))\in \HH^0(\A,C) $$
is an element which is mapped by $\pi$ to the element
$$ (1,1,...,1) \in \prod  \limits_{\nu} \ZZ \cong \HH^0(\A,\ZZ) $$
which is clearly rational. Since $H^3(K,\what{Z}) = H^3(K,\ovl{K}^*) = 0$ we get from Theorem ~\ref{t:arithmetic-duality} that $h_{\ZZ \U,n}((x_\nu))$ is rational if and only if it is orthogonal to all the elements in
$\HH^2\left(\Gam_K,\what{P_n(\uline{\ZZ \SimpS_{\mcal{U}}}}\right)$.

\begin{rem}\label{r:non-connected-br}
One might wonder what is the role of the connectivity of $X$ in the proof of theorem ~\ref{t:main}. The situation here is very similar to that of Remark ~\ref{r:non-connected-fin}. Again by ~\ref{p:zero-dim} we recall that
if $X = \spec(K) \coprod \spec(K)$ then
$$ X(\A)^{\ZZ h} = X(K) $$
On the other hand if $K$ is not totaly real then
$$ X(\A)^{\Br} \neq X(K) $$
As in the case of Remark ~\ref{r:non-connected-fin} the fault lies in the behavior of the complex places. Indeed since $\Br \CC = 0$ the pairing
$$ X(\A)\times \Br X\lrar \QQ / \ZZ $$
is not effected by the complex coordinate. Similarly to the case in Remark ~\ref{r:non-connected-fin}  this is the only fault, i.e. given Theorem ~\ref{t:main}  it is easy to see that even if $X$ is not geometrically connected, if one ignores the complex places (or if $K$ is totally real) then
$$ X(\A)^{\ZZ h} = X(\A)^{\Br} $$
\end{rem}


We now proceed to prove all the claims made in the above sketch of proof.
First we need to construct the maps
$$ \Psi_{\mcal{U},n}: \HH^2\left(K,\widehat{\uline{P_n(\ZZ \SimpS_{\mcal{U}})}}\right) \lrar H^2_{\acute{e}t}(X,t^*{\GG_m}) $$
We shall do so by constructing for every hypercovering $\mcal{U}$ a map
$$ \Psi_{\mcal{U}}:\mathbb{H}^i\left(K,\widehat{\uline{\ZZ \SimpS_{\mcal{U}}}}\right) \lrar H^i_{\acute{e}t}(X,{t^*\mathbb{G}_m}) $$
and defining for every hypercovering $\mcal{U}$ and $n\geq 0$
$$ \Psi_{\mcal{U}, n} = \Psi_\mcal{U}\circ \widehat{\rho_{\mcal{U},n}}_*$$
where
$$ \widehat{\rho_{\mcal{U},n}}_*:
\HH^2\left(K, \widehat{\uline{P_n(\ZZ \SimpS_{\mcal{U}})}}\right) \lrar
\HH^2\left(K, \widehat{\uline{\ZZ \SimpS_{\mcal{U}}}}\right) $$
is the natural map induced form the map:
$$\rho_{\mcal{U},n}: \ZZ \SimpS_{\mcal{U}} \lrar P_n(\ZZ \SimpS_{\mcal{U}})$$

We will show that for every $\mcal{U}$ the maps $\what{\rho_{\mcal{U},n}}_*$ induce together an isomorphism
$$ \lim \limits_{\stackrel{\lrar}{n}} \HH^2\left(\Gam_K,\widehat{\uline{P_n(\ZZ \SimpS_{\mcal{U}})}}\right) \lrar\HH^2\left(\Gam_K,\widehat{\uline{\ZZ \SimpS_{\mcal{U}}}}\right)
$$
Note that the map $\rho_{\mcal{U},n}$ induces an isomorphism on $\pi_i$ for $i=0,...,n$ (i.e. is $n$-connected). Hence the claim above will follow immediately from the following lemma, which implies that $\what{\rho_{\mcal{U},n}}_*$ is actually an isomorphism for $n \geq 2$:

\begin{lem}
Let $f: \mathbf{A} \lrar \mathbf{B}$ be a map of simplicial $\Gam_K$-modules which is $n$-connected. Then the induced map
$$ \widehat{f}^*: \HH^k\left(K,\widehat{\uline{\mathbf{B}}}\right) \lrar \HH^i\left(K,\widehat{\uline{\mathbf{A}}}\right) $$
is an isomorphism for $k=0,...,n$.
\end{lem}
\begin{proof}
Note that the complexes $\what{\uline{\mathbf{A}}}$ and $\what{\uline{\mathbf{B}}}$ are bounded from above by dimension $0$ and that the map
$$ \what{f}:\what{\uline{\mathbf{B}}} \lrar \what{\uline{\mathbf{A}}} $$
induces an isomorphism on the $i$'th homology group for $-n \leq i \leq 0$ (recall that the functor $\widehat{\bullet}$ is exact). Now let $k = 0,...,n$. Then $\widehat{f}$ induces a homotopy equivalence
$$ \ovl{P^{-}_0\left(\Sig^k\what{\uline{\mathbf{B}}}\right)}
\lrar \ovl{P^{-}_0\left(\Sig^k\what{\uline{\mathbf{A}}}\right)} $$
and hence by definition an isomorphism
$$ \what{f}^*:\HH^k\left(\Gam, \what{\uline{\mathbf{B}}}\right) \lrar \HH^k\left(\what{\uline{\mathbf{A}}}\right) $$
\end{proof}

We shall now proceed to construct the maps $\Psi_{\mcal{U}}$ and show that they fit together to form an isomorphism
$$ \lim \limits_{\stackrel{\lrar}{\mcal{U}}} \HH^2\left(\Gam_K,\widehat{\uline{\ZZ \SimpS_{\mcal{U}}}}\right) \lrar H^2_{\acute{e}t}(X,t^*{\GG_m}) $$
We will then show that the maps $\Psi_{\mcal{U},n}$ respect the pairing.

\begin{thm}\label{t:MIT-1}
Let $K$ be an arbitrary field with absolute Galois group $\Gam_K$ and $X$ and algebraic variety over $K$. Let $A$ be a $\Gam_K$-module and consider the \'{e}tale sheaf $t^*A$ on $X$ of locally constant maps into $A$ (see definition above). Then there exist natural maps
$$ \Psi_{\mcal{U}}:\mathbb{H}^i(K,\uline{\Hom}(\ZZ \SimpS_{\mcal{U}},A)) \lrar H^i_{\acute{e}t}(X,{t^*A}) $$
which induce an isomorphism
$$ \lim \limits_{\stackrel{\lrar}{\mcal{U}}}\mathbb{H}^i(K,\uline{\Hom}(\uline{\ZZ \SimpS_{\mcal{U}}},A)) \stackrel{\simeq}{\lrar} H^i_{\acute{e}t}(X,{t^*A}) $$
where $A$ is considered as a complex concentrated at degree $0$.
\end{thm}
\begin{proof}
First of all we note that
$$ H^i_{\acute{e}t}(X,{t^*A}) = \Ext^i_{X,\acute{e}t}(t^*\ZZ,{t^*A}) $$
We start by computing the left hand side. The hypercovering $\mcal{U}$ can be used to construct the sheaves
$$ \mcal{P}_n(\mcal{V}) = t^*\ZZ\left(\mcal{V} \times_X \mcal{U}_n\right) $$
which fit in a resolution
$$ ... \lrar \mcal{P}_2 \lrar \mcal{P}_1 \lrar \mcal{P}_0 \lrar t^*\ZZ $$
of $t^*\ZZ$. This is unfortunately not a projective resolution and so the resulting cohomology groups
$$ H^i(\Gam(\mcal{P}_\bullet)) = H^i\left(\Hom_{\Gam_K}(\uline{\ZZ \SimpS_{\mcal{U}}}_\bullet, A)\right) =
\left[\uline{\ZZ \SimpS_{\mcal{U}}},A\right]^i_{\Gam_K} $$

are not equal to the \'{e}tale cohomology groups. We do, however, get a map
$$ \Phi_{\mcal{U}}:\left[\uline{\ZZ \SimpS_{\mcal{U}}},A\right]^i_{\Gam_K} = H^i(\Gam(\mcal{P}_\bullet)) \lrar H^i_{\acute{e}t}(X,{t^*A}) $$ 

and we know by Verdier's hypercovering theorem  for \'{e}tale cohomology ([~\cite{SGA4},
Expos\'e V, 7.4.1(4)]) that by taking the direct limit over all \'{e}tale hypercoverings $\mcal{U}$ we get an isomorphism
$$ \Phi: \lim \limits_{\stackrel{\lrar}{\mcal{U}}} [\uline{\ZZ \SimpS_{\mcal{U}}},A]^i_{\Gam_K} \stackrel{\simeq}{\lrar} H^i_{\acute{e}t}(X,{t^*A}) $$

Since $\uline{\Hom}(\uline{\ZZ \SimpS_{\mcal{U}}},A)$ is bounded from above we can compute the right hand side by ~\ref{l:hyper-formula} and get

$$ \HH^i(K,\uline{\Hom}(\uline{\ZZ \SimpS_{\mcal{U}}},A))  =
\lim \limits_{\stackrel{\lrar}{L/K}}
\left[
\uline{\ZZ \E G_L}, \uline{\Hom}(\uline{\ZZ\SimpS_{\mcal{U}}},A)\right]^i_{\Gam_K} = $$
$$
\lim \limits_{\stackrel{\lrar}{ L/K }}
\left[\uline{\ZZ \E G_L} \otimes \uline{\ZZ\SimpS_{\mcal{U}}}, A\right]^i_{\Gam_K}
$$

We will prove the theorem by constructing a natural equivalence between the functors
$$ X \mapsto \left\{\uline{\ZZ \SimpS_{\mcal{U}}}\right\}_{\mcal{U}} $$
and
$$ X \mapsto \left\{\uline{\ZZ \E G_L} \otimes \uline{\ZZ\SimpS_{\mcal{U}}}\right\}_{\mcal{U}, L/K} $$
as functors from algebraic varieties over $K$ to the pro-category of $\Ho(\mcal{C}\Mod_{\Gam_K})$.

In order to construct a transformation
$$ F: \left\{\uline{\ZZ \E G_L} \otimes \uline{\ZZ\SimpS_{\mcal{U}}}\right\}_{\mcal{U}, L/K} \lrar
      \left\{\uline{\ZZ \SimpS_{\mcal{U}}}\right\}_{\mcal{U}} $$
we need to pick (compatibly) for each \'{e}tale hypercovering $\mcal{U}$ an \'{e}tale hypercovering $\U'$, a finite Galois extension $L/K$ and a map
$$ F_{\mcal{U}} :
\uline{\ZZ \E G_L} \otimes \uline{\ZZ\SimpS_{\U'}} \lrar
\uline{\ZZ \SimpS_{\mcal{U}}}  $$
Our choice here is simple. Take $\U' = \mcal{U}$ and $\Gam_L = \Gam_K$. Then choose $F_{\mcal{U}}$ to be the natural map
$$ F_{\mcal{U}} :
\uline{\ZZ \E(\Gam_K/\Gam_K)} \otimes \uline{\ZZ\SimpS_{\mcal{U}}} \lrar
\uline{\ZZ \SimpS_{\mcal{U}}}  $$

The other direction is more tricky. In order to construct a transformation
$$ G: \left\{\uline{\ZZ \SimpS_{\mcal{U}}}\right\}_{\mcal{U}} \lrar
      \left\{\uline{\ZZ \E G_L} \otimes \uline{\ZZ\SimpS_{\mcal{U}}}\right\}_{\mcal{U}, L/K} $$
we need to pick (compatibly) for each \'{e}tale hypercovering $\mcal{U}$ and a finite Galois extension $L/K$ an \'{e}tale hypercovering $\U'$ and a map
$$ G_{\mcal{U}, \Gam_L} : \uline{\ZZ \SimpS_{\mcal{U}'}} \lrar
      \uline{\ZZ \E G_L} \otimes \uline{\ZZ\SimpS_{\mcal{U}}} $$

Let
$$ \check{X}_L \lrar X $$
be as in definition ~\ref{d:X-L}. We now define
$$ \U' = \U_L = \check{X}_L \times_X \mcal{U}  $$
There is a natural map
$$ \SimpS_{\mcal{U}_L} \lrar \SimpS_{\check{X}_L} \times \SimpS_{\mcal{U}}  = \E G_L \times \SimpS_{\mcal{U}} $$
which gives a map
$$ \uline{\ZZ\SimpS_{\mcal{U}_L}} \lrar \uline{\ZZ(\E G_L \times \SimpS_{\mcal{U}})} $$

By composing with the Kunneth map (which is a homotopy equivalence of complexes)
$$ \uline{\ZZ( \E G_L \times \SimpS_{\mcal{U}})} \lrar
   \uline{\ZZ \E G_L} \otimes \uline{\ZZ \SimpS_{\mcal{U}}} $$
we construct our map
$$ G_{\mcal{U}, \Gam_L} : \uline{\ZZ \SimpS_{\U_L}} \lrar
      \uline{\ZZ \E G_L} \otimes \uline{\ZZ\SimpS_{\mcal{U}}} $$

The map $F \circ G$ is clearly the identity (in the category $\Pro\Ho(\mcal{C}\Mod_{\Gam_K})$).
Now consider
$$ G \circ F:
\left\{\uline{\ZZ \E G_L} \otimes \uline{\ZZ\SimpS_{\mcal{U}}}\right\}_{\mcal{U}, L/K} \lrar
\left\{\uline{\ZZ \E G_L} \otimes \uline{\ZZ\SimpS_{\mcal{U}}}\right\}_{\mcal{U}, L/K}
$$
This is the pro-map that for each $\U,\Gam_L$ chooses $\U' = \U_L$, $\Gam_L' = \Gam_K$ and the map
$$ \uline{\ZZ \E(\Gam_K/\Gam_K)} \otimes \uline{\ZZ \SimpS_{\U_L}} \lrar
      \uline{\ZZ \E G_L} \otimes \uline{\ZZ\SimpS_{\mcal{U}}} $$
obtained as above. In order to show that this pro-map represents the identity in $\Pro \Ho(\mcal{C}\Mod_{\Gam_K})$ we will show that the following diagram commutes in $\Ho(\mcal{C}\Mod_{\Gam_K})$:
$$ \xymatrix{
& \uline{\ZZ \E G_L} \otimes \uline{\ZZ \SimpS_{\U_L}} \ar[d]^{r_2} \ar[dl]^{r_1} \\
\uline{\ZZ \E(\Gam_K/\Gam_K)} \otimes \uline{\ZZ \SimpS_{\U_L}}  \ar^{G \circ F}[r] & \uline{\ZZ \E G_L} \otimes \uline{\ZZ \SimpS_{\mcal{U}}} \\
}
$$
where the $r_1,r_2$ are refinement maps which are the structure maps of our pro-object. The two maps $G \circ F \circ r_1$ and $r_2$ both factor through $\uline{\ZZ \E G_L} \otimes \uline{\ZZ \E G_L} \otimes \uline{\ZZ \SimpS_{\mcal{U}}}$ as
$$ G \circ F \circ r_1 = f_1 \circ (Id \otimes G) $$
and
$$ r_2 = f_2 \circ (Id \otimes G) $$
where
$$ f_1 = p \otimes Id \otimes Id : \uline{\ZZ \E G_L} \otimes \uline{\ZZ \E G_L} \otimes \uline{\ZZ \SimpS_{\mcal{U}}} \lrar \uline{\ZZ \E G_L} \otimes \uline{\ZZ \SimpS_{\mcal{U}}} $$
$$ f_2 = Id \otimes p \otimes Id : \uline{\ZZ \E G_L} \otimes \uline{\ZZ \E G_L} \otimes \uline{\ZZ \SimpS_{\mcal{U}}} \lrar \uline{\ZZ \E G_L} \otimes \uline{\ZZ \SimpS_{\mcal{U}}} $$
and $p:\uline{\ZZ \E G_L} \lrar \uline{\ZZ}$ is the natural projection.

Hence it is enough to prove that the square
$$ \xymatrix{
\uline{\ZZ \E G_L} \otimes \uline{\ZZ \SimpS_{\U_L}} \ar[r]^{Id \otimes G} \ar^{Id \otimes G}[d] &
\uline{\ZZ \E G_L} \otimes \uline{\ZZ \E G_L} \otimes \uline{\ZZ \SimpS_{\mcal{U}}}  \ar^{f_2}[d] \\
\uline{\ZZ \E G_L} \otimes \uline{\ZZ \E G_L} \otimes \uline{\ZZ \SimpS_{\mcal{U}}} \ar^{f_1}[r]& \uline{\ZZ \E G_L} \otimes \uline{\ZZ \SimpS_{\mcal{U}}} \\
}
$$
commutes up to equivariant homotopy, or simply that $f_1$ and $f_2$ are equivariantly homotopic as chain maps.

Let $H = \Gam_K/\Gam_L$. Since the action of $\Gam_K$ on $\ZZ \E G_L$ factors through $H$ it is enough to show that
$$ p \otimes Id,Id \otimes p :
\uline{\ZZ \E H} \otimes \uline{\ZZ \E H} \lrar
\uline{\ZZ \E H} $$
are $H$-equivariantly homotopic as chain maps. Recall that $\uline{\ZZ \E H} \otimes \uline{\ZZ \E H}$ is equivariantly homotopy equivalent to $\uline{\ZZ(\E H \times \E H)}$. Hence it is enough to show that the two projections

$$ p_1,p_2: \E H \times \E H \lrar \E H $$
Are equivariantly homotopic. Note that both $\E H$ and $\E H \times \E H$ are contractible free $H$-spaces and so the equivariant mapping space from $\E H \times \E H$ to $\E H$ is homotopy equivalent to $\E H^{hH}$ which is contractible. This means that $p_1$ and $p_2$ are $H$-equivariantly homotopic and we are done (the fact that this homotopy can be done simplicially can be seen using the projective model structure on simplicial $H$-sets).

Now the maps
$$ \Psi_{\mcal{U}}:\HH^2\left(\Gam_K,\widehat{\uline{\ZZ \SimpS_{\mcal{U}}}}\right) \lrar H^2_{\acute{e}t}(X,t^*{\GG_m}) $$
are obtained as the composition
$$ (*)\;\; \HH^2\left(\Gam_K,\widehat{\uline{\ZZ \SimpS_{\mcal{U}}}}\right) =
\lim \limits_{\stackrel{\lrar}{ L/K  }}
\left[\uline{\ZZ \E G_L} \otimes \uline{\ZZ\SimpS_{\mcal{U}}}, \GG_m\right]^2_{\Gam_K} \stackrel{G_{\mcal{U}}^*}{\lrar} $$
$$ \lim \limits_{\x{\lrar}{\U_L}}\left[\uline{\ZZ \SimpS_{\mcal{U}_L}},\GG_m\right]^2_{\Gam_K} \stackrel{\Phi_{\mcal{U'}}}{\lrar} H^2_{\acute{e}t}(X,t^*{\GG_m}) $$

\begin{rem}\label{r:trivial-case}
Note that for $X = \spec(K)$ and every hypercovering $\U \lrar \spec(K)$ we have a quasi-isomorphism
$$ \uline{\Hom}\left(\uline{\ZZ \SimpS_\mcal{U}},A\right) \simeq \uline{\Hom}\left(\uline{\ZZ \SimpS_X},A\right) = \uline{\Hom}\left(\uline{\ZZ},\uline{A}\right) = \uline{A} $$
where $X \lrar X$ is the trivial hypercovering. Substituting $\U = X$ in $(*)$ one sees that $G_{\mcal{U}}^*$ becomes the identity and $\Psi_{\U}$ becomes the standard identification
$$ \HH^i(K,A) = H^i(K,A) \x{\simeq}{\lrar} H^i_{\acute{e}t}(\spec(K),A) $$
\end{rem}

\end{proof}
Now we shall now show that the maps $\Psi_{\mcal{U},n}$ fit into the pairing diagram
$$ \xymatrix{
X(\A) \ar@{=}[d] & \times & H^2_{\acute{e}t}(X,{\GG_m}) \ar[r] & \QQ / \ZZ \ar@{=}[d]\\
X(\A) \ar[d] & \times & H^2_{\acute{e}t}(X,t^*{\GG_m})\ar@{->>}[u]^{i_*} \ar[r] & \QQ / \ZZ \ar@{=}[d]\\
\HH^0(\A,\uline{\ZZ\SimpS_{\mcal{U},n}}) & \times &
\HH^2\left(K_\nu,\widehat{\uline{\ZZ \SimpS_{\mcal{U},n}}}\right)\ar^{\Psi_{\mcal{U},n}}[u]\ar[r]\ar[r] & \QQ / \ZZ \\
}$$

The pairings
$$ \xymatrix{
X(\A) \ar@{=}[d] & \times & H^2_{\acute{e}t}(X,{\GG_m}) \ar[r] & \QQ / \ZZ \ar@{=}[d]\\
X(\A) & \times & H^2_{\acute{e}t}(X,t^*{\GG_m})\ar@{->>}[u]^{i_*} \ar[r] & \QQ / \ZZ \\
} $$
are defined via pullbacks. To be precise, a point $(x_\nu) \in X(\A)$ can be considered as a sequence of maps $x_\nu: \Spec(K_\nu) \lrar X$ and we can pull back the element $u \in H^2_{\acute{e}t}(X,{\GG_m})$ to an element
$$ x_\nu^*u \in H^2_{\acute{e}t}(\Spec(K_\nu)) \stackrel{\inv}{\cong} \QQ / \ZZ $$
The pairing of $(x_\nu)$ and $u$ is then defined to be
$$ \sum_\nu \inv(x_\nu^*u) \in \QQ / \ZZ $$

Clearly we can do the same with $H^2_{\acute{e}t}(X,t^*_{\GG_m})$ instead of $H^2_{\acute{e}t}(X,{\GG_m})$ and we would get a commutative diagram of pairings.

The pairing
$$ \xymatrix{
\HH^0(\A,\uline{\ZZ\SimpS_{\mcal{U},n}})  & \times & \HH^2\left(K,\widehat{\uline{\ZZ \SimpS_{\mcal{U},n}}}\right) \ar[r] & \QQ / \ZZ \\
}$$
is defined as follows. Given
$$ x \in \HH^0(\A,\uline{\ZZ\SimpS_{\mcal{U},n}}), y \in \HH^2\left(K,\widehat{\uline{\ZZ \SimpS_{\mcal{U},n}}}\right) $$
we go over all places $\nu$ of $K$, project $x$ to $\HH^0(K_\nu,\uline{\ZZ\SimpS_{\mcal{U},n}})$ and restrict $y$ to $\HH^2\left(K_{\nu},\widehat{\uline{\ZZ \SimpS_{\mcal{U},n}}}\right)$. We then pair using the cup product (note that all complexes here are homologically bounded) to get an element in
$$ \HH^2(K_{\nu}, \GG_m) = H^2(K_{\nu},\GG_m) \stackrel{\inv}{\cong} \QQ / \ZZ $$
then as before we sum up the invariants in all the places to get the pairing $(x,y) \in \QQ / \ZZ$.

This means that in order to show the compatibility of the parings it is enough to show the compatibility of
$$ \xymatrix{
X(K_\nu) \ar[d] & \times & H^2_{\acute{e}t}(X_\nu,t^*{\GG_m}) \ar[r] & \QQ / \ZZ \ar@{=}[d]\\
\HH^0\left(K_\nu,\uline{\ZZ\SimpS_{\U_\nu,n}}\right) & \times &
\HH^2\left(K_\nu,\what{\uline{\ZZ\SimpS_{\U_\nu,n}}}\right)\ar^{\Psi_{\U_\nu,n}}[u]\ar[r]\ar[r] & \QQ / \ZZ \\
}
$$
where $X_\nu,\U_\nu$ are the base changes of $X$ and $\U$ respectively from $K$ to $K_\nu$ (note that this base change doesn't change geometric connected components so $\SimpS_{\U_\nu, n}$ is actually the same simplicial set as $\SimpS_{\U,n}$).

Since this diagram is functorial in $X_\nu$ is is enough to prove it for $X_\nu = \spec(K_\nu)$. In that case $\uline{\ZZ\SimpS_{\U_\nu,n}}$ is quasi-isomorphic to $\ZZ$ considered as a complex concentrated at degree $0$ (for all hypercoverings $\U_\nu \lrar \spec(K_\nu)$ and all $n$) and we get the diagram
$$ \xymatrix{
\{\bullet\} \ar[d] & \times & H^2_{\acute{e}t}(\spec(K_\nu),t^*{\GG_m}) \ar[r] & \QQ / \ZZ \ar@{=}[d]\\
\ZZ & \times &
\HH^2\left(K_\nu,\GG_m\right)\ar^{\Psi_{\mcal{U}}}[u]\ar[r] & \QQ / \ZZ \\
}
$$
where the point $\bullet$ is mapped to $1 \in \ZZ$. But this pairing diagram compatible in view of remark ~\ref{r:trivial-case} so we are done.

It is now only remains to prove that $i_*$ is a surjective.

\begin{thm}\label{t:surj_br}
Let $X$ be a smooth variety. Then the map
$$ i_{*}: H^2_{\acute{e}t}(X,t^*{\GG_m})  \lrar H^2_{\acute{e}t}(X,\GG_m),  $$
is surjective.
\end{thm}
\begin{proof}
For $n\in \NN$ Consider the map of short Kummer sequences

$$
\xymatrix{
0 \ar[r] & t^*{\mu_n} \ar@{=}[d] \ar[r] & t^*{\GG_m} \ar[d] \ar[r]^{\times n} &  t^*{\GG_m} \ar[d] \ar[r] & 0 \\
0 \ar[r] & t^*{\mu_n} \ar[r] & \GG_m \ar[r]^{\times n} &  \GG_m \ar[r] & 0 \\
}
$$

This give rise to a map of long Kummer sequences
$$
\xymatrix{
  ... \ar[r] & H^2_{\acute{e}t}(X,t^*{\mu_n}) \ar@{=}[d] \ar[r] & H^2_{\acute{e}t}(X,t^*{\GG_m}) \ar[d] \ar[r]^{\times n} & H^2_{\acute{e}t}(X,t^*{\GG_m}) \ar[d] \ar[r] & ...  \\
  ... \ar[r] & H^2_{\acute{e}t}(X,t^*{\mu_n}) \ar[r] & H^2_{\acute{e}t}(X,\GG_m) \ar[r]^{\times n} &  H^2_{\acute{e}t}(X,\GG_m) \ar[r] & ... \\
}
$$

Now from an easy diagram chase we get that
$$ i_{*}: H^2_{\acute{e}t}(X,t^*{\GG_m})[n]  \lrar H^2_{\acute{e}t}(X,\GG_m)[n] $$
is surjective. Since $X$ is a smooth variety $H^2_{\acute{e}t}(X,\GG_m)$ is torsion group (e.g. by ~\cite{Lie08} Corollary 3.1.3.4).

\end{proof}

\subsection{Proof of Arithmetic Duality Results }\label{s:arithmetic-duality}
In this section we will proof the main auxiliary result (Theorem ~\ref{t:arithmetic-duality}) we need in order to prove the equivalence of the homological and the Brauer-Manin obstruction (Theorem ~\ref{t:main}). We will need various generalizations of results from the theory of arithmetic duality of Galois modules (See ~\cite{Mil06}) to Galois complexes. Similar and related results appear in ~\cite{HSz05}, ~\cite{De09b} and in ~\cite{Jos09}.

Let $K$ be a field with absolute Galois group $\Gam$. Consider the module
$$ \mathfrak{J} = \lim \limits_{\stackrel{\lrar}{F}} \lim \limits_{\stackrel{\lrar}{T}} \lim \limits_{\stackrel{\lrar}{F/K}} \prod \limits_{\om \in T} F_\om \times
\prod \limits_{\om \notin T} O_\om  $$
as a complex concentrated at degree $0$ (where $F$ runs over finite extensions of $K$ and $\om \in T$ means that $\om$ is a place of $F$ which lies over a place in $T$).

Note that there is an inclusion
$$\ovl{K}^* \subseteq \mathfrak{J}$$

We shall denote
$$\mfrak{C} = \mathfrak{J}/\ovl{K}^*$$

Then $\mfrak{C}$ is a \textbf{class formation}. This means that using the Yoneda product one obtains for every $\Gam$-module $M$ a pairing
$$ H^{2-r}(\Gam,M) \times \Ext^r_{\Gam}(M,\mfrak{C}) \lrar H^2(\Gam,\mfrak{C}) \stackrel{\simeq}{\lrar} \QQ/\ZZ $$
which is the basis for all arithmetic duality results.

We would like to generalize these notions from $\Gam$-modules to $\Gam$-complexes. We replace $\Ext$ with the appropriate notion in $\cal{C}{\Mod_{\Gam}}$ we denote by $\EExt$.

Since the Yoneda product still exists we get an analogous pairing
$$ \HH^{2-r}(\Gam,M) \times \EExt_{\Gam}^r(M,\mfrak{C}) \lrar \HH^2(\Gam,\mfrak{C}) \stackrel{\simeq}{\lrar} \QQ/\ZZ $$
for every $\Gam$-complex $M$. We think of $\mfrak{C}$ as a complex concentrated in degree $0$.
Note that this  pairing induces maps:
$$ \alp^r_M: \EExt_{\Gam}^r(M,\mfrak{C}) \lrar \left(\HH^{2-r}(\Gam,M))\right)^* = \left(\HH^{2-r}(M))\right)^*$$

\begin{define}
Let  $C \in \mcal{C}{\Mod_{\Gam}}$ be an excellent complex. We shall say that $C$ is \textbf{a finite reduced complex}
if it satisfies the following conditions
\begin{enumerate}
\item $C$ is homologically bounded.
\item $H_i(C)= 0 ,\quad i\leq 0$.
\item $H_i(C)$ is finite for all $i \in \ZZ$.
\end{enumerate}
\end{define}

\begin{lem}\label{l:new-class}
Let $K$ be a number field and Let $C \in \mcal{C}\Mod_{\Gam_K}$ be a finite reduced complex. Then the map
$$\alpha_{\what{C}}^r: \EExt^r\left(\what{C},\mathfrak{C}\right)\lrar \HH^{2-r}\left(K,\what{C}\right)^*$$
is an isomorphism for every $r \geq 0$ and is surjective for $r = -1$.
\end{lem}

\begin{proof}
We will say that $C$ is $n$-bounded if $H_i(C) =0 $ for $i > 0$. We will prove the lemma for $n$-bounded complexes by induction on $n$.

For $n=1$ we get that $\what{C}$ is quasi-isomorphic to $\Sigma^{-1}M$ for some finite Galois module $M$. The claim then is just Theorem 4.6 in ~\cite{Mil06} after the suitable dimension shift (note that both sides are invariant to quasi-isomorphisms).

Now assume that the lemma is true for $n$-bounded complexes ($n \geq 1$) and let $C$ be an $(n+1)$-bounded finite reduced complex. Consider the short exact sequence
$$ 0 \lrar C \langle 1 \rangle  \lrar C \lrar P^+_1(C) \lrar 0 $$
Since
$$ \HH^{r}\left(K,\what{C\langle 1 \rangle }\right) =  \HH^{r-1}\left(K,\what{\Sigma^{-1}C\langle 1 \rangle }\right)$$
$$\EExt^{r}\left(\what{C\langle 1 \rangle },\mathfrak{C}\right)  = \EExt^{r+1}\left(\what{\Sigma^{-1}C\langle 1 \rangle },\mathfrak{C}\right) $$
and since $\Sigma^{-1}C\langle 1 \rangle$ is $n$-bounded we get that the lemma is true for $\alp^{r}_{C \langle 1\rangle }$. Since $P^+_1(C)$ is $1$-bounded the lemma is true for $\alp^{r}_{C \langle 1\rangle }$ as well. We then observe the map of long exact sequences:

%

$$ \xymatrix@-1.3pc{
\ar[r] & \EExt^{r-1}\left(\what{P^+_1(C)},\mfrak{C}\right) \ar[r]\ar^{\alp^{r-1}_{P^+_1(C)}}[d] &
            \EExt^{r}\left(\what{C\langle 1 \rangle },\mfrak{C}\right) \ar[r]\ar^{\alp^{r}_{C\langle 1 \rangle }}[d] &
            \EExt^r\left(\what{C},\mfrak{C}\right) \ar[r]\ar^{\alp^r_{C}}[d] &
            \EExt^r\left(\what{P^+_1(C)},\mfrak{C}\right) \ar[r]\ar^{\alp^r_{P^+_1(C)}}[d] &
            \EExt^{r+1}\left(\what{C\langle 1 \rangle },\mfrak{C}\right) \ar^{\alp^{r+1}_{C\langle 1 \rangle }}[d] \ar[r] &\\
\ar[r]& \left(\HH^{3-r}(\what{P^+_1(C)}))\right)^* \ar[r] &
            \left(\HH^{2-r}\left(\what{C\langle 1 \rangle }\right)\right)^* \ar[r] &
            \left(\HH^{2-r}\left(\what{C}\right)\right)^* \ar[r] &
            \left(\HH^{2-r}\left(\what{P^+_1(C)}\right)\right)^* \ar[r] &
            \left(\HH^{1-r}\left(\what{C\langle 1 \rangle }\right)\right)^* \ar[r]& \\
}$$
For $r \geq 0$ we get from the five lemma that $\alp^r_{C}$ is an isomorphism (recall that in the five lemma it is enough to assume the left most map is surjective). If $r = -1$ then an analogous diagram chase shows that $\alp^r_{C}$ is surjective.

\end{proof}

\begin{lem}\label{l:ext-to-h}
Let $K$ be a field of characteristic zero and let  $C$ be an excellent homologically bounded complex such that $H_n(C)$ is finitely generated for every $n$. Then there are natural isomorphisms
$$ \HH^i(K,C) \cong \EExt^i\left(\widehat{C}, \ovl{K}^*\right) $$
$$ \HH^i(\A,C) \cong \EExt^i\left(\widehat{C}, \mathfrak{J}\right) $$
\end{lem}
\begin{proof}
By lemma ~\ref{l:assume_bounded} $C$ is quasi-isomorphic to an excellent bounded complex $D$. Since both $C$ and $D$ are excellent we can identify the rational and adelic hypercohomology of $C$ and $D$ (see remark ~\ref{r:adelic-invariant}). Since $\widehat{\bullet}$ is exact we get a quasi-isomorphism from $\widehat{D}$ to $\widehat{C}$ so we can identify the RHS for $C$ and $D$. Hence it is enough to prove the claim for bounded complexes.

For $C$ which is concentrated at degree $0$ (i.e. is actually a $\Gam$-module) this claim is proved in the ~\cite{Mil06} Lemmas $4.12$ and $14.3$. One can then proceed by induction on the length of $C$ using the fact that for bounded complexes both
$\HH^i(\Gam,\bullet)$ and $\HH^i(\A,\bullet)$ transform short exact sequences to long exact sequences (Lemmas ~\ref{l:rational-hyper-ses} and ~\ref{l:adelic-hyper-ses}) and applying the $5$-lemma.

\end{proof}

Now let $C$ be an excellent homologically bounded $\Gam_K$-complex. Let
$$ p: \HH^0(\A,C) \lrar \HH^2\left(K,\what{C}\right)^* $$
be the composition

$$ \HH^0(\A,C) \x{\simeq}{\lrar} \EExt^0\left(\widehat{C}, \mathfrak{J}\right) \lrar \EExt^0\left(\what{C},\mathfrak{C}\right) \x{\alpha_{\what{C}}^0}{\lrar} \HH^{2}\left(K,\what{C}\right)^* $$
where the first isomorphism is the one given in Lemma ~\ref{l:ext-to-h}. Then $p$ induces a pairing
$$ \HH^0(\A,C) \times \HH^2\left(K,\what{C}\right) \lrar \QQ/\ZZ $$
Unwinding the definitions one can see that this pairing is given by using the cup product
$$ \HH^0(K_\nu,C) \times \HH^2\left(K_\nu,\what{C}\right) \lrar \HH^2(K_\nu,\GG_m) \cong \QQ/\ZZ $$
for each $\nu$ and summing the results.

We will need to work with complexes whose zero'th homology group is not finite but a finitely generated permutation module, hence we can't assume that $\alp_{\what{C}}^0$ is an isomorphism. We can, however, make do with the following Theorem:

\begin{thm}\label{t:arithmetic-duality}

Let $C$ be an excellent homologically bounded $\Gam_K$-complex such that:
\begin{enumerate}
\item
$C_n = 0$ for $n < 0$
\item
$H_n(C)$ is finite for all $n > 0$.
\item
The Galois module $M = H_0(C)$ satisfies the property that $H^3(K,\widehat{M}) = 0$.
\end{enumerate}

Denote by $\pi:C \lrar M$ the natural map. Consider the (non-exact) sequence

$$ \HH^0(K,C) \x{\loc}{\lrar} \HH^0(\A,C) \x{p}{\lrar} \HH^2\left(K,\what{C}\right)^* $$

Let $(x_\nu) \in {\HH}^0(\A,C)$ be such that $\pi_*((x_\nu))$ is rational and $p((x_\nu)) = 0$. Then $(x_\nu)$ is rational, i.e. is in the image of $\loc$.

\end{thm}

\begin{proof}

Consider the following diagram with exact rows
$$ \xymatrix{
\EExt^0\left(\what{C}, \ovl{K}^*\right) \ar[r]\ar[d] & \EExt^0\left(\what{C}, \J\right) \ar[d]\ar[r] & \EExt^0\left(\what{C}, \C\right) \ar^{\what{\pi}_*}[d] \\
\EExt^0\left(\what{M}, \ovl{K}^*\right) \ar[r] & \EExt^0\left(\what{M}, \J\right) \ar[r] & \EExt^0\left(\what{M}, \C\right)  & \\
}$$

Using the isomorphisms of Lemma ~\ref{l:ext-to-h} we can consider $(x_\nu)$ as an element $\bet \in \EExt^0\left(\what{C},\J\right)$ whose image in  $\EExt^0\left(\what{M},\J\right)$ comes from $\EExt^0\left(\what{M},\ovl{K}^*\right)$. Let $\bet$ be the image of $(x_\nu)$ in $\EExt^0\left(\what{C},\C\right)$. Then from commutativity we get that
$$ \what{\pi}_*(\bet) = 0 \in \EExt^0\left(\what{M},\C\right) $$

We will show that $\bet = 0$. This would imply that $(x_\nu)$ comes from $\EExt^0\left(\what{C},\ovl{K}^*\right) \cong \HH^0(K,C)$ as desired.

Let $\wtl{C} = \ker(\pi)$ and consider the exact sequence of $\Gam_K$-complexes
$$ 0 \lrar \what{M} \lrar \what{C} \lrar \what{\wtl{C}} \lrar 0 $$
Consider the following commutative diagram with exact columns
$$ \xymatrix{
& & \HH^3\left(K,\what{M}\right) \ar[d] \\
\EExt^0\left(\what{\wtl{C}}, \J\right) \ar[d]\ar[r] & \EExt^0\left(\what{\wtl{C}}, \C\right) \ar^{\alp^{0}_{\wtl{C}}}[r]\ar[d] & \HH^2\left(K,\what{\wtl{C}}\right)^* \ar[d]\\
\EExt^0\left(\what{C}, \J\right) \ar[d]\ar[r] & \EExt^0\left(\what{C}, \C\right) \ar^{\what{\pi}_*}[d]\ar^{\alp^{0}_{C}}[r] & \HH^2\left(K,\what{C}\right)^*  \\
\EExt^0\left(\what{M}, \J\right) \ar[r] & \EExt^0\left(\what{M}, \C\right)  & \\
}$$
Since $\what{\pi}_*(\bet) = 0 $ we get that there is a $\gam \in \EExt^0\left(\what{\wtl{C}}, \C\right)$ mapping to $\bet$. The fact that $p(x_\nu) = 0$ means that $\alp^{0}_{C}(\bet) = 0$. Since we assume that $\HH^3\left(K,\what{M}\right) = 0$ we get that $\alp^{0}_{\wtl{C}}(\gam) = 0$. But by Lemma ~\ref{l:new-class} the map $\alp^{0}_{\wtl{C}}$ is an isomorphism and so $\gam = 0$. This means that $\bet = 0$ as well and we are done.
\end{proof}

\section{The Equivalence of the Homotopy Obstruction and the \'{E}tale Brauer Obstruction }\label{s:equivalence-2}

The main result of this section is the following theorem:

\begin{thm}\label{t:h-is-fin-br}
Let $K$ be a number field and $X$ a smooth geometrically connected variety over $K$. Then
$$X(\A)^{fin, \Br} = X(\A)^{h}$$
\end{thm}

\begin{rem}\label{r:non-connected-fin-br}
If one wants to omit the condition that $X$ is geometrically connected one again face a small problem in the complex places of $K$ (the situation here is completely analogous to the one in Remarks ~\ref{r:non-connected-fin} and ~\ref{r:non-connected-br}).
\end{rem}


The proof will be based on two propositions:
\begin{enumerate}
\item
Proposition ~\ref{p:Et-Zh-subset-h} will show that for any smooth variety over $K$ we have
$$ X(\A)^{fin,\ZZ h} \subseteq X(\A)^{h} $$
\item
Proposition ~\ref{p:h-subset-fin-h} will show that for any smooth variety over $K$ we have
$$X(\A)^{h} \subseteq X(\A)^{fin,h}$$
\end{enumerate}

We then get

$$ X(\A)^{fin,\ZZ h} \subseteq X(\A)^{h}  \subseteq X(\A)^{fin,h} \subseteq X(\A)^{fin,\ZZ h} $$
which means that

$$ X(\A)^{fin,\ZZ h} = X(\A)^{h} $$ 

The theorem will then follow form the lemma

\begin{lem}
Let $K$ be a number field and $X$ a geometrically connected smooth variety over $K$. Then
$$ X(\A)^{fin,\Br} = X(\A)^{fin,\ZZ h} $$
\end{lem}
\begin{proof}
By Theorem ~\ref{t:main} this claim is almost immediate. The only point one would verify  is the geometric connectivity issue (note that even though $X$ is geometrically connected this in not true for every $Y$ which is a torsor over $X$). By Remarks ~\ref{r:non-connected-fin}, ~\ref{r:non-connected-br} and ~\ref{r:non-connected-fin-br} the only problem lies in the complex places and this can be easily treated.
\end{proof}

\begin{prop}\label{p:Et-Zh-subset-h}
Let $K$ be a number field and $X$ a smooth variety over $K$. Then
$$ X(\A)^{fin,\ZZ h} \subseteq X(\A)^{h} $$
\end{prop}
\begin{proof}

From Corollary ~\ref{c:coprod} we can assume without loss of generality that $X$ is connected over $K$. But in this case both sets are empty unless $X$ is geometrically connected (Corollary ~\ref{c:catches-pi-0}). Hence we can assume that $X$ is geometrically connected.

We shall prove the proposition by using the following lemma
\begin{lem}\label{l:pull-to-sc}
Let $K$ be a field and $X$ a smooth geometrically connected variety over $K$ such that the short exact sequence
$$ 1\lrar \pi_1(\ovl{X})\lrar \pi_1(X)\lrar \Gam_K \lrar 1 $$
admits a continuous section. Then for every hypercovering $\U \lrar X$ there exists a geometrically connected torsor $ f: Y \lrar X $ under a finite $K$-group $G$ such that $\SimpS_{f^*(\U)}$ is simply connected.

\end{lem}
Before proving Lemma ~\ref{l:pull-to-sc} we shall explain why it implies the proposition. Let
$$ (x_\nu) \in X(\A)^{fin,\ZZ h} $$
By Corollary ~\ref{c:descent-theorem} it is enough to prove that for every hypercovering $\U \lrar X$ and every $n$ we have
$$ (x_\nu) \in X(\A)^{\U, n} $$
Now since
$$(x_\nu) \in X(\A)^{fin,\ZZ h} \subseteq X(\A)^{fin} $$
we have that
$$ X(\A)^{fin} \neq \emptyset $$
and by Corollary ~\ref{c:Harari-Stix} the sequence

$$ 1 \lrar \pi_1(\ovl{X})\lrar \pi_1(X)\lrar \Gam_K \lrar 1 $$

admits a continuous section. Now by Lemma ~\ref{l:pull-to-sc} there exists a geometrically connected torsor $ f: Y \lrar X $ under a finite $K$-group $G$ such that $\SimpS_{f^*(\U)}$ is simply connected. Since
$$ (x_\nu) \in X(\A)^{fin,\ZZ h} $$
there exists some twist $f^{\alp}:Y^\alp \lrar X^\alp$ of $f$ and a point $(y_\nu) \in Y^\alp(\A)^{\ZZ h}$ such that
$$ (x_\nu) = f^\alp((y_\nu)) $$
Now we have
$$ (y_\nu) \in Y^\alp(\A)^{\ZZ h} \subseteq Y^\alp(\A)^{\ZZ {f^\alp}^*(\U), n} $$
But since $\SimpS_{{f^\alp}^*(\U)}$ (which is isomorphic to $\SimpS_{{f}^*(\U)}$ as a simplicial set) is simply connected we get from Corollary ~\ref{c:cosk2-zzh-h} that
$$ (y_\nu) \in Y^\alp(\A)^{ {f^\alp}^*(\U),n}$$
as well.

Now by an easy diagram chase on the diagram
$$
\xymatrix{
Y^\alp(\A)\ar[d]^{f^\alp}\ar[r]^-{h} & P_n(\SimpS_{{f^\alp}^*(\U)})(h\A) \ar[d]  & \ar[l]^{\loc}\ar[d] P_n(\SimpS_{{f^\alp}^*(\U)})(hK)\\
X(\A)\ar[r]^-{h} & P_n(\SimpS_{\U})(h\A)   & \ar[l]^{\loc} P_n(\SimpS_{\mcal{U}})(hK)
}
$$

we get that
$$ (x_\nu) = f^\alp((y_\nu)) \in X(\A)^{\U,n} $$
as needed.

We shall now prove Lemma ~\ref{l:pull-to-sc}.
\begin{proof}
Let $\U \lrar X$ be a hypercovering. By abuse of notation we shall refer to the pullback of $\U$ to the generic point of $X$ by $\mcal{U}$ as well. Now $\U_0$ can be considered just as a finite $\Gam_{K(X)}$ set.

We denote by
$$ \Stab_{\Gam_{K(X)}} (\U_0) = \{H_1,...,H_t\} $$
the set of stabilizers in  $\Gam_{K(X)}$ of the points of $\U_0$.
Consider the group
$$ E = \left<H_1 \cap {\Gam_{\ovl{K}(\ovl{X})}},...,H_t\cap {\Gam_{\ovl{K}(\ovl{X})}}\right> $$
Let $F/\ovl{K}(\ovl{X})$ be the finite field extension that corresponds to $E$.

Observing Lemma ~\ref{l:Upsilon} and its proof we see that there exists a $G$-torsor
$$ f:\Upsilon\lrar \ovl{X} $$
such that $\ovl{K}(\Upsilon) = F$. Consider now $f^*\U_0$ as a $\Gam_{\ovl{K}(\Upsilon)} = E$ set. This is just the restriction of the $\Gam_{\ovl{K}(\ovl{X})}$-action on $\U_0$ to the subgroup $E$. Hence the set of stabilizers of $f^*\U_0$ is
$$ \{H_1 \cap E,...,H_t\cap E\} $$

Since
$$ E = \left<H_1 \cap {\Gam_{\ovl{K}(\ovl{X})}},...,H_t\cap {\Gam_{\ovl{K}(\ovl{X})}}\right> =
\left<H_1 \cap E,...,H_t\cap E\right> $$

we get from Lemma ~\ref{l:calc_pi1} that $\SimpS_{f^*\mcal{U}}$ is simply connected. Now it is enough to show that we can find a $K$-form of $\Upsilon$. Since $E$ is normal in $\Gam_{K(X)}$ we get that $\rho(E)$ is normal in $\pi_1(X)$. Hence the covering
$$ f:\Upsilon\lrar \ovl{X} $$
is isomorphic over $\ovl{K}$ to all its twists. By standard arguments we get that there exists a $K$-form for $f$ if and only if the short exact sequence
$$ 1 \lrar \pi_1(\ovl{X})/\rho(E) \lrar \pi_1(X)/\rho(E) \lrar \Gam_K \lrar 1 $$
has a continuous section. But this is true because
$$ 1 \lrar \pi_1(\ovl{X})\lrar \pi_1(X)\lrar \Gam_K \lrar 1 $$
has a continuous section.
\end{proof}
\end{proof}

\begin{prop}\label{p:h-subset-fin-h}
Let $K$ be a number field and $X$ a smooth $K$-variety over then
$$ X(\A)^h \subseteq X(\A)^{fin, h} $$
\end{prop}
\begin{proof}
Like in the proof of Proposition ~\ref{p:Et-Zh-subset-h} we can assume without loss of generality that $X$ is geometrically connected.

Let $(x_\nu) \in X(\A)^h$ be a point and $f:Y \lrar X$ a $K$-connected torsor under a finite $K$-group $G$. We need to show that there exists an $\alp \in H^1(K,G)$ such that the corresponding twist
$$ f^\alp:Y^\alp \lrar X $$
satisfies
$$ (f^{\alp})^{-1}((x_\nu)) \cap Y^{\alp}(\A)^h \neq \emptyset$$

Choose a rational homotopy fixed point
$$ r \in X\left(hK\right) $$
such that
$$ \loc(r) = h((x_\nu)) $$

Recall the map
$$ {c_Y}: \SimpS_{\check{Y}}(hK)  \lrar H^1(K,G) $$

defined in the beginning of section \S ~\ref{s:finite-descent}. For each $\alp \in H^1(K,G^\alp)$ we will denote by
$$ c_{Y^\alp}: \SimpS_{\check{Y}^\alp}(hK)  \lrar H^1(K,G^\alp) $$
the corresponding map for the twist $Y^\alp$. As before we will sometimes abuse notation and call the composition
$$ X\left(hK\right) \lrar \SimpS_{\check{Y}^\alp}(hK)  \lrar H^1(K,G^\alp) $$
$c_{Y^\alp}$ as well.

The strategy of proof will be as following:
\begin{enumerate}
\item
We will show that for each $r \in X\left(hK\right)$ there exists an $\alp$ such
$$ c_{Y^\alp}(r) \in H^1(K,G^\alp) $$
is the neutral element.
\item
We will show that if $c_{Y^\alp}(r)$ is neutral and $\loc(r) = h((x_\nu))$ then
$$ (f^{\alp})^{-1}((x_\nu)) \cap Y^{\alp}(\A)^h \neq \emptyset$$
\end{enumerate}

\begin{prop}\label{p:main part 1}
Let $X,K$ and $f:Y \lrar X$ be as above. Let $r \in X(hK)$ be a rational homotopy fixed point. Then there exists an $\alp \in H^1(K,G)$ such that
$${c_Y}_{\alp}(r) \in H^1(K,G^\alp)$$
is the neutral element.
\end{prop}
\begin{proof}

Let
$$\tau_\alpha : H^1(K,G)\lrar  H^1(K,G^\alp) $$
be the inverse of twisting by $\alp$ (hence sending $[\alp] \in H^1(K,G)$ to the neutral element in $H^1(K,G^\alp)$).
We will prove the theorem by showing that
$$ (*)\;\; c_{Y^\alp} = \tau_\alp \circ {c_Y} $$
Then if we simply take $\alp = {c_Y}(r)$ we will get that $c_{Y^\alp}(r)$ is the neutral element. We will now proceed to prove $(*)$.

Let $\alp \in Z^1(K,G)$ be a 1-cocycle and $L/K$ a finite Galois extension such that $\alp$ can be written as a map $\alp: \Gal(L/K) \lrar G(L)$.

Let $f^{\alp}:Y^{\alp} \lrar X$ be the corresponding twist by $\alp$ which is a finite torsor under $G^{\alp}$.
Recall the \'{e}tale hypercovering $$\check{X}_L \lrar X$$
from Definition ~\ref{d:X-L}.

We shall use in the proof the following \'{e}tale hypercoverings (all considered over $X$)

$$ Y_L = Y\times_X X_L $$
$$ Y^{\alp}_L = Y^{\alp}\times_X X_L $$

Now the equality $(*)$ will follow from the following lemma
\begin{lem}\label{l:twist-relations}
There exists a $\Gam_K$-equivariant isomorphism of groupoids

$$\mcal{T}_\alp : \mcal{E}(\Gal(L/K))\times {\mcal{B}G} \lrar \mcal{E}(\Gal(L/K))\times \mcal{B}G^{\alp} $$

Such that
\begin{enumerate}
\item\label{i:1:p:big_diag_grpds}

The following diagram in the category simplicial  $\Gam_K$-set commutes:
$$\xymatrix{
\empty & \SimpS_{\check{Y}} \ar[r]^{{c_Y}}  & \B G\\
\SimpS_{\check{Y}_L}\ar[r] \ar[d]^{T_\alp} \ar[ur]  &  \SimpS_{\check{Y}}\times \SimpS_{\check{X}_L} \ar[u] \ar[r]^-{{c_Y} \times d} &  \B G \times \E(\Gal(L/K)) \ar[d]^{N(\mcal{T}_\alp)}\ar[u]\\
\SimpS_{\check{Y}^{\alp}_L}\ar[r] \ar[dr] & \SimpS_{\check{Y}^\alp}\times \SimpS_{\check{X}_L} \ar[d] \ar[r]^-{c_{Y^\alp} \times d }&  \B G^{\alp} \times \E(\Gal(L/K)) \ar[d]\\
\empty  &\SimpS_{\check{Y}^{\alp}} \ar[r]^{c_{Y^\alp}} & \B G^{\alp}\\
}$$
Where all the unmarked maps are induced from projections and the isomorphism $T_\alp$ is the one defined in Lemma ~\ref{l:twist-relations-1}.
\item\label{i:2:p:big_diag_grpds}

All arrows in the right column of the diagram above are weak equivalences and thus induce a bijection
$$\tau_\alpha :\B G(hK)= H^1(K,G)\lrar  H^1(K,G^\alp) =  \B G^{\alp}(hK)$$
This bijection is the inverse of twisting by $\alp$ and thus sends the $[\alp] \in H^1(K,G)$ to the neutral element in $H^1(K,G^\alp)$.
\end{enumerate}
\end{lem}
\begin{proof}
We begin by defining
$$\mcal{T}_\alp : \mcal{E}(\Gal(L/K))\times {\mcal{B}G} \lrar \mcal{E}(\Gal(L/K))\times \mcal{B}G^{\alp} $$

We define $\mcal{T}_\alp$ to be the identity on the objects. Now Let
$$(\tau,*),(\sig,*)\in Ob(\mcal{E}(\Gal(L/K))\times {\mcal{B}G})  = \mcal{E}(\Gal(L/K))\times \mcal{B}G^{\alp} $$

One can naturally identify

$$ \Hom_{\mcal{E}(\Gal(L/K))\times {\mcal{B}G}}((\tau,*),(\sig,*)) \cong G$$
$$ \Hom_{\mcal{E}(\Gal(L/K))\times \mcal{B}G^{\alp}}((\tau,*),(\sig,*)) \cong G^{\alp}$$
Hence we need to define a map
$${\mcal{T}_\alp}^{(\sig,\tau)}:G\lrar G^\alp$$
We shall take

$${\mcal{T}_\alp}^{(\sig,\tau)}(g) = \alp(\tau)^{-1}g\alp(\sig) $$
We leave it to the reader to verify that $\mcal{T}_\alp$ is indeed a $\Gam_K$-equivariant isomorphism of groupoids.

Now one can easily verify the commutativity of diagram in  (~\ref{i:1:p:big_diag_grpds}).
For proving (~\ref{i:2:p:big_diag_grpds}) note that since $\mcal{T}_\alp$ is an isomorphism so is
$N(\mcal{T}_\alp)$. Thus since $\E(\Gal(L/K))$ is contractible all the maps in right column are weak-equivalences. We leave to the reader to verify that
the map
$$ \tau_\alpha :\B G(hK)= H^1(K,G)\lrar  H^1(K,G^\alp) =  \B G^{\alp}(hK) $$
is indeed the inverse of twisting by $\alp$

\end{proof}

\end{proof}

\begin{prop}\label{p:main part 2}
Let $\alp \in H^1(K,G)$ be such that $c_{Y^\alp}(r)$ is the neutral element. Then
$$ (f^{\alp})^{-1}((x_\nu)) \cap Y^{\alp}(\A)^h \neq \emptyset$$

\end{prop}
\begin{proof}

Note that without the loss of generality we can replace $Y$ by $Y^\alp$ and $\alp$ by the neutral element.


We want to show that there is a point in $f^{-1}((x_\nu))$ which is homotopy rational. For that it will be enough to show that every hypercovering $\U$ of $Y$ and every $n\geq 1$  the subset of $f^{-1}((x_\nu))$ of points which are $(\U,n)$-homotopy rational is \textbf{non-empty} and \textbf{compact}. Then the standard compactness argument shows that
$$ f^{-1}((x_\nu)) \cap Y(\A)^h \neq \emptyset $$

First we will show that it is enough to prove this only for hypercoverings of a certain form.

\begin{lem}\label{l:enough V}
Let $K,X,G$ and $f:Y \lrar X$ be as above. Let $\U \lrar Y$ be a hypercovering. Then there exist a hypercovering $\mcal{V} \lrar X$ which admits a map
$$ \mcal{V} \lrar \check{Y} $$
of hypercoverings over $X$ and a map
$$ f^*(\mcal{V}) \lrar \U $$
of hypercoverings over $Y$.
\end{lem}

\begin{proof}
For a $g \in G$ we denote by $\U^g$ the hypercovering obtained from $\mcal{U}$ by replacing each \'{e}tale map
$$ \U_n \lrar Y $$
with the composition
$$ \U_n \lrar Y \x{g}{\lrar} Y $$

Define $\U^G$ to be the fiber product
$$ \U^G = \prod_{g \in G} \U^g $$
We have a natural action of $G$ on $\mcal{U}^G$ which is free and so we get a hypercovering
$$ \U^G/G \lrar X $$
of $X$. It is not hard to check that
$$ \U^G = f^*\left(\U^G/G\right) $$
Now take the hypercovering
$$ \mcal{V} = \U^G/G \times_X \check{Y} $$
of $X$. Then one has a natural map over $X$
$$ \mcal{V} \lrar \check{Y} $$
obtained by projection and since
$$ f^*\mcal{V} = \mcal{U}^G \otimes_{Y} f^*\check{Y} $$
we have the composition of maps over $Y$
$$ f^*\mcal{V} \lrar \U^G \lrar \mcal{U} $$
where the last one is obtained by projecting on the coordinate which corresponds to $1 \in G$.

\end{proof}


In view of Lemma ~\ref{l:enough V} it is enough to show that the subset of $f^{-1}((x_\nu))$ which is $(f^*\mcal{V},n)$-homotopy rational is non-empty and compact for $n \geq 1$ and every hypercovering $\mcal{V} \lrar X$ which admits a map
$$ \phi:\mcal{V} \lrar \check{Y} $$
over $X$. Such hypercoverings have the following useful property

\begin{lem}\label{l:prin_G}
Let $\mcal{V} \lrar X$ be a hypercovering and $\vphi:\mcal{V} \lrar \check{Y}$ a map of hypercoverings. Then $\SimpS_{\mcal{V}}$ is a free $G$-quotient by the natural map $\SimpS_{f^*\mcal{V}} \lrar \SimpS_{\mcal{V}}$. In other words $\SimpS_{f^*\mcal{V}} \lrar \SimpS_{\mcal{V}}$ is a principle $G$-coverings. Further, the composition
$$ \SimpS_{\mcal{V}} \lrar \SimpS_{\check{Y}} \x{{c_Y}}{\lrar} \B G $$
is the classifying map of $\SimpS_{f^*\mcal{V}} \lrar \SimpS_{\mcal{V}}$ as a principle $G$-cover.
\end{lem}
\begin{proof}
We have a natural action of $G$ on $\SimpS_{f^*\mcal{V}}$ such that $\SimpS_{\mcal{V}}$ is the quotient. Hence by Lemma ~\ref{l:EGBG} it will be enough if we will show that $\SimpS_{f^*\mcal{V}}$ fits in a commutative diagram
$$ \xymatrix{
\SimpS_{f^*\mcal{V}} \ar^{d}[r]\ar[d] & \E G \ar[d] \\
\SimpS_{\mcal{V}} \ar[r] & \B G \\
} $$
such that $d$ is $G$-equivariant. Since we have a map $\vphi:\mcal{V} \lrar \check{Y}$ it is enough to show this for $\mcal{V} = \check{Y}$. This is done by explicit computation.

Consider the augmented functor $F \lrar Id$ from groupoids to groupoids defined as follows. If $C$ is a category then the objects of $F(C)$ are pairs $(X,f)$ where $X \in Ob(C)$ and $f$ is a morphism starting at $X$. The morphisms from $(X,f)$ to $(Y,g)$ are morphisms $h:X \lrar Y$ in $C$ such that $g \circ h = f$. The natural transformation from $F$ to $Id$ sends $(X,f)$ to $X$.

It is then an easy verification to check that $\SimpS_{f^*\check{Y}}$ is the nerve of the groupoid $F(\mcal{Y})$, where $\mcal{Y}$ is the groupoid defined in \S ~\ref{s:finite-descent}, and that $\E G$ is the nerve of $F({\mcal{B}G})$. Then the commutative diagram
$$ \xymatrix{
F(\mcal{Y}) \ar[r]\ar[d] & F({\mcal{B}G}) \ar[d] \\
\mcal{Y} \ar[r] & {\mcal{B}G} \\
}$$
Gives the desired commutative diagram
$$ \xymatrix{
\SimpS_{f^*\check{Y}} \ar[d]\ar[r] & \E G \ar[d] \\
\SimpS_{\check{Y}} \ar[r] & \B G \\
}$$
of simplicial sets. It is left to show that the map $\SimpS_{f^*\check{Y}} \lrar \E G$ is $G$ equivariant. Note that in both $\mcal{Y}$ and $\B G$ we have natural identification
$$ \cup_Y \Hom(X,Y) = G $$
and the action of $G$ on $F(\mcal{Y})$ and $F({\mcal{B}G})$ is given by multiplication
$$ g(X,f) = (X,gf) $$
Hence it is clear that the map $F(\mcal{Y}) \lrar F({\mcal{B}G})$ is $G$-equivariant and we are done.
\end{proof}

Now let $\mcal{V} \lrar X$ be a hypercovering which admits a map of hypercoverings to $\check{Y}$. We have a commutative diagram
$$ \xymatrix{
Y(\A) \ar^{f}[d]\ar^-{h_{f^*\mcal{V},n}} [r] & \SimpS_{f^*\mcal{V},n}(h\A) \ar^{f_h}[d] \\
X(\A) \ar^-{h_{\mcal{V},n}} [r] & \SimpS_{\mcal{V},n}(h\A) \\
}$$

By Lemma ~\ref{l:prin_G} the sequence
$$ \SimpS_{f^*\mcal{V},n} \lrar \SimpS_{\mcal{V},n} \lrar P_n(\B G) = \B G $$
is a principle $G$-fibration sequence. In particular the first map is a normal covering with Galois group $G(\bar{K})$. Hence by Lemma  ~\ref{l:preserves fibrations} the sequence
$$ (*)\; \SimpS_{f^*\mcal{V},n}^{h\Gam_K} \lrar \SimpS_{\mcal{V},n}^{h\Gam_K} \lrar \B G^{h\Gam_K}$$
is a fibration sequence.

From the spectral sequence that computes the homotopy groups of $\B G^{h\Gam_K}$ we see that the connected component of the base point has vanishing homotopy groups in dimension greater then $1$ and that its fundamental group is
$$ H^0(\Gam_K,G) = G^{\Gam_K} $$
which we also denote by $G(K)$.  Hence the trivial connected component of $\B G^{h\Gam_K}$ is a model for $\B G(K)$.

Now from the fibration $(*)$ we get that following sequence of sets/groups (obtained as the tail of the long exact sequence in homotopy groups):
$$G(K) \x{\delta}{\lrar} \SimpS_{f^*\mcal{V},n}(hK) \x{f_H}{\lrar} \SimpS_{\mcal{V},n}(hK) \x{({c_Y}\circ {\phi})_h}{\lrar} H^1(K,G) $$
which is exact is the following sense:
\begin{enumerate}
\item
A point $v \in \SimpS_{\mcal{V},n}(hK)$ can be lifted to $\SimpS_{f^*\mcal{V},n}(hK)$
if and only if $({c_Y}\circ {\phi})^h(v)$ is the neutral element.
\item
If $v$ does lift to $\SimpS_{f^*\mcal{V},n}(hK)$ then $G(K)$ acts transitively (although not necessarily faithfully )  on $f_h^{-1}(v)$.
\end{enumerate}
Note that we can do the same for $\A$ instead of $K$ and get the following diagram:
$$\xymatrix{
G(K) \ar[r]^-{\delta^K} \ar[d]^{loc}  &\SimpS_{f^*\mcal{V},n}(hK) \ar[r]^-{f^{K}_h} \ar[d]^{loc}&\SimpS_{\mcal{V},n}(hK) \ar[r]^-{({c_Y}\circ {\phi})_h}\ar[d]^{loc} &H^1(K,G)\ar[d]^{loc}\\
G(\A) \ar[r]^-{\delta^\A}  &\SimpS_{f^*\mcal{V},n}(h\A) \ar[r]^-{f^{\A}_h} &\SimpS_{\mcal{V},n}(h\A) \ar[r]^-{({c_Y}\circ {\phi})_h} &H^1(\A,G) \\
G(\A)\ar@{=}[u] \ar[r]  &Y(\A)\ar[u]^{h_{f^*\mcal{V},n}} \ar[r]^-{f} & X(\A)\ar[u]^{h_{\mcal{V},n}} \ar[r] & H^1(\A,G) \ar@{=}[u]\\
}$$
Note that Lemma ~\ref{l:classifies-fiber} guaranties that the last row fits commutatively into the diagram. We now proceed to finish the proof of Proposition ~\ref{p:main part 2}.

Let
$$ A = \{v \in \SimpS_{f^*\mcal{V},n}(hK) \left|  \loc(f^{K}_h(v)) = h_{\mcal{V},n}(x_\nu)\right. \} $$
We wish to show that it is non-empty and finite. Let
$$ r_{\mcal{V},n} \in \SimpS_{\mcal{V},n}(hK) $$
be the image of $r$ in $\SimpS_{\mcal{V},n}(hK)$. The map $\phi_h$ sends $r_{\mcal{V},n}$ to the image of $r$ in $\SimpS_{\check{Y},n}(hK)$ and thus
$$ ({c_Y}\circ \phi)_h(r_{\mcal{V},n}) $$
is the neutral element. This means that there exists a
$$ v \in \SimpS_{f^*\mcal{V},n}(hK) $$
such that
$$ f^K_h(v)  = r_{\mcal{V},n} $$
and so we have that
$$ \loc(f^K_h(v))  = \loc(r_{\mcal{V},n}) = h_{\mcal{V},n}(x_\nu) $$
implying that $ A \neq \emptyset $.

On the other hand by Lemma ~\ref{p:finite pre-image} the set of all elements
$s \in  \SimpS_{\mcal{V},n}(hK)$ such that $\loc(s) = h_{\mcal{V},n}(x_\nu)$ is finite and since $G(K)$ is finite so is $A$.

From the compactness of $G(\A)$ and the commutativity of the diagram we see that the subspaces
$$ B = \left(f^{\A}_h\right)^{-1}(h_{\mcal{V},n}(x_\nu)) \subseteq \SimpS_{f^*\mcal{V},n}(h\A) $$
and
$$ f^{-1}(x_\nu) \subseteq Y(\A)$$
are compact and that the continuous map (see Lemma ~\ref{l:adelic-is-adelic})
$$ h_{f^*\mcal{V},n}: Y(\A) \lrar \SimpS_{f^*\mcal{V},n}(h\A) $$
restricts to a continuous map

$$ h_B : f^{-1}(x_\nu) \lrar B $$
From the exactness of the rows we get that $h_B$ is also \textbf{surjective}.

Now from commutativity we see that
$$ \loc(A) \subseteq B $$
and furthermore if $(y_\nu)$ is in $f^{-1}(x_\nu)$ then $h_B((y_\nu)) \in B$ is rational if and only if $ h_B((y_\nu)) \in \loc(A) $.  Hence we get that
$$ f^{-1}((x_\nu)) \cap Y(\A)^{f^*\mcal{V},n} = h^{-1}_B(\loc(A)) $$

Now since $\loc(A)$ is non-empty and finite and since $h_B$ is continuous and surjective we have that $h^{-1}_B(\loc(A))$ is non-empty and closed in the compact space $f^{-1}((x_\nu))$.
Hence
$$ f^{-1}((x_\nu)) \cap Y(\A)^{f^*\mcal{V},n} $$ 
is non-empty and compact.
\end{proof}
\end{proof}

\section{Applications}\label{s:applications}

In this section we shall present some applications of the theory developed in the paper.

\begin{thm}\label{t:product-fin-Br}
Let $K$ be number field and $X,Y$ two smooth geometrically connected $K$-varieties. Then

$$(X\times Y)(\A)^{fin,\Br}  = X(\A)^{fin,\Br} \times Y(\A)^{fin,\Br}$$

\end{thm}
\begin{proof}
Let $\tau: K \hrar \CC$ be an embedding. Then it is known that the pro-object $ \acute{E}t(X) $ is isomorphic to the pro-finite completion of the topological space $(X \otimes_{\tau} \CC)(\CC)$. Since pro-finite completion commutes with products this means that the natural map
$$\acute{E}t(\ovl{X}\times \ovl{Y}) \lrar \acute{E}t(\ovl{X})\times \acute{E}t(\ovl{Y}) $$
is an isomorphism. Recalling Proposition ~\ref{p:bar-is-bar} the result then follows from and Theorem ~\ref{t:obstruction-iso}.
\end{proof}

\begin{thm}\label{t:trivial-pi-2}
Let $K$ be a number field and $X$ a smooth geometrically connected variety over $K$. Assume further that
$$ \pi^{\acute{e}t}_2(\ovl{X}) = 0 $$
Then

$$ X(\A)^{fin} =X(\A)^{fin,\Br} $$

\end{thm}

\begin{proof}
In this case the map $\acute{E}t^2(\ovl{X}) \lrar \acute{E}t^1(\ovl{X})$ is an isomorphism in $\Pro\Ho\left(\Set^{\Del^{op}}\right)$. By Proposition ~\ref{p:bar-is-bar} and Theorem ~\ref{t:obstruction-iso} this means that
$$ X(\A)^{h,1}= X(\A)^{h,2} $$
By Theorems ~\ref{t:fin}, Corollary ~\ref{c:cosk2-final-2} and ~\ref{t:h-is-fin-br} we then have
$$ X(\A)^{fin} = X(\A)^{h,1}= X(\A)^{h,2} = X(\A)^{h} = X(\A)^{fin,\Br} $$
\end{proof}

\begin{cor}\label{c:desc-is-fin}
Let $K$ be a number field and $X$ a smooth geometrically connected proper variety over $K$ such that $\pi^{\acute{e}t}_2(\ovl{X})=0$. Then
$$ X(\A)^{desc} = X(\A)^{fin} $$
\end{cor}

\begin{cor}
Let $C$ be a smooth curve over $K$. Then
$$ C(\A)^{fin} = C(\A)^{fin,\Br} $$
If $C$ is also projective then
$$C(\A)^{fin} = C(\A)^{desc}$$
\end{cor}
\begin{proof}
If $C \cong \PP^1$ over $\ovl{K}$ then
$$ C(\A)^{fin} = C(\A) = C(\A)^{\Br} = C(\A)^{fin,\Br} $$
Otherwise $C$ is geometrically a $K(\pi,1)$ and one can verify that $\pi_1(C(\CC))$ is good in the sense of section $6$ of ~\cite{AMa69}. Hence
$$ \pi^{\acute{e}t}_2(\ovl{C}) = 0 $$
and then the results follows from Theorem ~\ref{t:trivial-pi-2} and Corollary ~\ref{c:desc-is-fin}.
\end{proof}

\begin{thm}
Let $X$ and $K$ be as above. If $H^{\acute{e}t}_2(\ovl{X}) = 0$ then
$$ X(\A)^{\Br} = X(\A)^{fin-ab} $$
\end{thm}
\begin{proof}
In this case we get from ~\ref{t:obstruction-iso} that
$$ X(\A)^{\ZZ h,2} = X(\A)^{\ZZ h,1} $$
and so the result follows from Theorem ~\ref{t:main}, Corollary ~\ref{c:cosk2-final-2} and Theorem ~\ref{t:fin}.
\end{proof}

\begin{thm}
Let $X$ and $K$ be as above and assume that $\pi_1^{\acute{e}t}(\ovl{X})$ is abelian and the Hurewicz map
$$\pi^{\acute{e}t}(\ovl{X})_2 \lrar H^{\acute{e}t}_2(\ovl{X})$$
is an isomorphism. Then

$$ X(\A)^{fin, Br} = X(\A)^{\Br} $$
\end{thm}
\begin{proof}
In this case we get from ~\ref{t:obstruction-iso} that
$$ X(\A)^{h,2} = X(\A)^{\ZZ h,2} $$
and so the result follows from Theorem ~\ref{t:h-is-fin-br}, Corollary ~\ref{c:cosk2-final-2} and Theorem ~\ref{t:main}.

\end{proof}

\end{document}